\documentclass[a4paper,10pt]{article}
\usepackage[utf8x]{inputenc}
\usepackage{dsfont}
\usepackage{amsmath,amsfonts,amssymb,amsthm}
\usepackage{pb-diagram,graphicx, subcaption}
\usepackage{color}
\usepackage{fullpage}
\usepackage{tikz}
\usepackage{stmaryrd}

\newcommand{\dd}{\mathcal{D}}
\newcommand{\R}{\mathds{R}}
\newcommand{\Om}{\Omega}
\newcommand{\Omc}{\Omega^c}

\newcommand{\D}{\mathcal{D}}
\newcommand{\OO}{\mathds{P}(\D)}
\newcommand{\C }{\mathcal{C}}

\newcommand{\divv}{\operatorname{div}}
\newcommand{\dn}{\partial_n}
\newcommand{\sol}{u}
\newcommand{\Jop}{\mathcal{P}}
\newcommand{\tz}{t_0}
\newcommand{\Tt}{T_t}

\newcommand{\po}{{\partial \Om}}
\newcommand{\pd}{{\partial \D}}
\newcommand{\meas}{h}

\newcommand{\dt}{{\frac{d}{dt}}}

\newcommand{\Sb}{\mathbf{S}}

\newcommand{\transp}{\mathsf{T}}

\newcommand{\ben}{\begin{equation}}
\newcommand{\een}{\end{equation}}
\newcommand{\ma}{\mathds{A}}
\newcommand{\mI}{\mathds{I}}

\providecommand{\Div}{\operatorname{div}}          % Divergence
\providecommand*{\meas}[2][]{\Meas{#1}{#2}}
\providecommand{\Det}{\operatorname{det}}                    % determinant
\newcommand{\VS}{{\mathbf{S}}}

\newcommand{\VV}{V}

\providecommand{\bbR}{\mathbb{R}}

\theoremstyle{plain}

\newtheorem{theorem}{Theorem}[section]
\newtheorem{proposition}[theorem]{Proposition}
\newtheorem{lemma}[theorem]{Lemma}
\newtheorem{corollary}[theorem]{Corollary}

\newtheorem{assumption}[theorem]{Assumption}

\theoremstyle{definition}
\newtheorem{definition}[theorem]{Definition}
\newtheorem{remark}[theorem]{Remark}

\definecolor{preto}{RGB}{0,0,0}
\definecolor{vermelho}{RGB}{254,49,49}
\definecolor{azul_escuro}{RGB}{0,0,255}
\definecolor{cinza}{RGB}{195,195,195}
\definecolor{papel_amarelo}{RGB}{255,249,240}

\title{A shape optimization approach for electrical impedance tomography with point measurements}
\author{Yuri Flores Albuquerque\thanks{Instituto de Matem\'atica e Estat\'istica,
Universidade de S\~{a}o Paulo, Rua do Mat\~{a}o 1010, 05508-090, S\~{a}o Paulo, Brazil (\texttt{yuri.falbu@gmail.com, laurain@ime.usp.br})} \and Antoine Laurain\footnotemark[1] \and Kevin Sturm\thanks{Technische Universit\"at Wien, 
Institut f\"ur Analysis und Scientific Computing,
Wiedner Hauptstra\ss e 8-10, 1040 Wien, Austria (\texttt{kevin.sturm@asc.tuwien.ac.at})}}
\date{}

\begin{document}
%%%%%%%%%%%%%%%%%%%%%
\maketitle
%%%%%%%%%%%%%%%%%%%%%
\begin{abstract}
Working within the class of piecewise constant conductivities, the inverse problem of electrical impedance tomography can be recast as a shape optimization problem where the  discontinuity interface is the unknown.  
Using Gr\"oger's $W^{1}_p$-estimates for mixed boundary value problems, the averaged adjoint method is extended to the case of Banach spaces, which allows to compute the derivative of shape functionals involving point evaluations. 
We compute the corresponding distributed expression of the shape derivative and show that it may contain Dirac measures in addition to the usual domain integrals.
We use this distributed shape derivative to devise a numerical algorithm, show various numerical results supporting the method, and based on these results we discuss the influence of the point measurements patterns on the quality of the reconstructions.
\end{abstract}
%%%%%%%%%%%%%%%%%%%%%
\begin{center}
\end{center}

%%%%%%%%%%%%%%%%%%%%%%%%%%%%%%%%%%%
\section{Introduction}
%%%%%%%%%%%%%%%%%%%%%%%%%%%%%%%%%%%
Electrical impedance tomography (EIT) is a low cost, noninvasive, radiation free and portable
imaging modality with various applications in medical imaging, geophysics, civil engineering and nondestructive testing.
In particular, it is an active field of research in medical imaging, where devices based on EIT are already used in practice, with applications to lung imaging such as diagnosis of pulmonary embolism \cite{MARTINS2019}, monitoring patients undergoing mechanical ventilation, breast imaging, 
acute cerebral stroke, or cardiac activity monitoring; we refer to the reviews \cite{Bera_2018,MR1955896} and the references therein.

Two mathematical models for EIT have been actively investigated over the last few decades. 
The {\it continuum model} has been widely studied in the case where applied currents and voltage measurements are supposed to be known on the entire boundary.
This model is closely related to the {\it Calderón problem}, which has attracted the attention of a large community of mathematicians in the last decades; see \cite{Bera_2018,MR1955896}.
It consists in determining the uniqueness and stability properties of the conductivity reconstruction  
when the full Dirichlet-to-Neumann map is known, which corresponds, roughly speaking, to the availability of an infinite amount of applied currents and boundary measurements.

Despite its usefulness, the continuum model is not realistic for applications, indeed, in the case of medical imaging for instance, it does not take into account the fact that currents are applied through electrodes attached by small patches to the patient, and that voltage measurements are also performed through these electrodes. 
Therefore, the applied currents and voltage measurements are available only on a subset of the boundary. 
In the literature, this situation is referred to as {\it partial measurements} as opposed to {\it full measurements} for the standard Calderón problem.
This leads to the more realistic {\it electrode model} \cite{MR1174044}, which also takes into account the electro-chemical reaction occurring at the interface between the electrode and the skin.

As the field of EIT is growing more mature, the awareness of these restrictions has increased also among mathematicians.
As a consequence, the study of the continuum model with partial boundary data has attracted much attention in the recent years.
Uniqueness results with partial boundary data in dimension $n\geq 3$ were obtained in \cite{MR2262748}, in \cite{MR2299741} for $\C^2$-conductivities, and in \cite{MR2209749} for $W^{3/2+\delta,2n}$-conductivities with $\delta>0$. 
Uniqueness results were extended to conductivities of class $\C^{1,\infty}(\overline{\Om})\cap H^{3/2}(\Om)$ and conductivities in $W^{1,\infty}(\Om)\cap H^{3/2+\delta}(\Om)$ with $0 <\delta < 1/2$ arbitrarily small but fixed in \cite{MR3551265}. 
We refer to \cite{MR3221605} for a review of theoretical results on the Calderón problem with partial data.
Regarding numerical methods, sparsity priors are used to improve the reconstruction using partial data in  \cite{MR3462483,MR3447196}. 
D-bar methods in two dimensions were investigated in \cite{MR3642251,MR3626801}  and resistor networks in \cite{MR2608623}.

Due to the small size of the electrodes compared to the rest of the boundary in many practical applications, the idea of modeling small electrodes by point
electrodes using Dirac measures is appealing from the mathematical standpoint.
This point of view has been introduced  as a {\it point electrode model} and justified in \cite{MR2819201}; see also \cite{MR3400033,MR2553181,MR3023421}.
We also observe that mathematical models using point measurements are highly relevant for  large-scale inverse problems such as full-waveform inversion where the dimensions of the receivers are several orders of magnitude smaller than the dimensions of the physical domain of the model; see \cite{Virieux2009}.   

The problem of reconstructing conductivities presenting sharp interfaces in EIT, also known as the inclusion detection problem, has attracted significant interest in the last three decades, starting from the pioneering works \cite{MR873244,MR1017325}.
Several numerical methods have been developed for reconstructing discontinuous conductivities including the factorization method introduced in \cite{MR1776481,MR1662460}; see also the review \cite{MR3095385},  
monotonicity-based shape reconstructions \cite{MR3621830,MR3628886,MR3126995}, the enclosure method for reconstructing the convex hull of a set of inclusions \cite{MR1694840,MR1776482}, the MUSIC algorithm for determining the locations of small inclusions \cite{MR2168949}, a  nonlinear integral equation method  \cite{MR2309659}, and topological derivative-based methods \cite{MR2888256,MR2517928,MR2536481,MR2886190}.
Shape optimization techniques, which are the basis of the present paper, have also been employed to tackle this problem: based on level set methods \cite{MR2132313,MR3535238}, for a polygonal partition of the domain \cite{MR3723652}, using second-order shape sensitivity \cite{MR2407028}, and using a  single boundary measurement \cite{MR2329288, MR1607628}.

In this framework, the conductivity is assumed to be piecewise constant or piecewise smooth, and it is then convenient to reformulate the problem as a shape optimization problem \cite{MR1215733} in order to investigate the sensitivity with respect to perturbations of a trial interface. 
This sensitivity analysis relies on the calculation of the {\it shape derivative}, which can be written either in a strong form, usually as a boundary integral, or in a weak form which often presents itself as a domain integral involving the derivative of the perturbation field.
The usefulness of the weak form of the shape derivative, often called {\it domain expression} or {\it distributed shape derivative}, is known since the pioneering works  \cite{MR800331,MR860040} but has been seldom used since then in comparison with the boundary expression. 
A revival of the distributed shape derivative has been observed since \cite{MR2642680}, and 
this approach has been further developed in the context of EIT and level set methods in \cite{MR3535238}, see also \cite{MR3660456}.

An important contribution of the present paper is to extend the framework developed in \cite{MR3535238} to the case of point measurements in EIT.
The main issue for shape functionals involving point evaluations is that one needs the continuity of the state, for which the usual $H^1$-regularity in two dimensions is insufficient.
Functionals with point evaluations and pointwise constraints have been studied intensively in the optimal control literature; see \cite{MR2551487,MR2583281}.
In particular, a convenient idea from optimal control is to use Gr\"oger's  $W^{1}_p$-estimates \cite{MR990595,a_GRRE_1989a} with $p>2$  to obtain continuity of the state in two dimensions.  
Here, we adapt this idea in the context of shape optimization and of the averaged adjoint method, in the spirit of \cite{MR3584578}.
We show that in general the shape derivative  contains Dirac measures, and that the adjoint state is slightly less regular than $H^1$ due to the presence of Dirac measures on the right-hand side.
Another important contribution of this paper is to investigate the relations between the domain and boundary expressions of the shape derivative depending on the interface regularity, and  the minimal regularity of the interface for which the boundary expression of the shape derivative can be obtained in the context of EIT with point measurements. 

We start by recalling in Section \ref{sec:preliminaries} the  $W^{1}_p$-estimates for mixed boundary value problems introduced in \cite{MR990595}.
We then formulate in Section \ref{sec:EIT_point} the shape optimization approach for the inverse problem of EIT and show how the averaged adjoint method can be adapted to the context of Banach spaces. 
Then, we compute the distributed shape derivative and prove its validity for a conductivity inclusion which is only open.
When the inclusion is Lipschitz polygonal or $\C^1$, we also obtain the boundary expression of the shape derivative.
Finally, in Section \ref{sec:numerics} we explain the numerical algorithm based on the distributed shape derivative and we present a set of results showing the efficiency of the approach. 
Introducing an error measure for the reconstruction, we also discuss the quality of reconstructions depending on the number of point measurements, applied boundary currents and noise level.
More details about the averaged adjoint method are given in an appendix for the sake of completeness.

%%%%%%%%%%%%%%%%%%%%%%%%%%%%%%%%%%%
\section{Preliminaries}\label{sec:preliminaries}
%%%%%%%%%%%%%%%%%%%%%%%%%%%%%%%%%%%
%%%%%%%%%%%%%%%%%%%%%%%%%%%%%%%%%%%
\subsection{Mixed boundary value problems in $W^1_p$}\label{sec:prel1}
%%%%%%%%%%%%%%%%%%%%%%%%%%%%%%%%%%%
In this section we recall the framework introduced in \cite{MR990595} for obtaining a $W^1_p$-estimate for solutions to
mixed boundary value problems for second order elliptic PDEs.
\begin{definition}[see \cite{MR990595,MR2551487}]\label{def1b}
Let $\D\subset\R^2$ and $\Gamma\subset\partial\D$ be given. We say that $\D\cup\Gamma$ is
regular (in the sense of Gr\"oger) if $\D$ is a bounded Lipschitz domain, $\Gamma$ is a relatively
open part of the boundary $\pd$, $\Gamma_0 := \pd\setminus\Gamma$ has positive measure, and $\Gamma_0$ is a finite
union of closed and nondegenerated (i.e., not a single point) curved pieces of $\pd$. 
\end{definition}

Let $\D,\Gamma$ and $\Gamma_0$ be  as in Definition \ref{def1b} and define for $d\geq 1$:
$$ \C^\infty_\Gamma (\D,\R^d): = \{ f|_{\D}\ | \ f\in\C^\infty(\R^2,\R^d),\, \operatorname{supp} f\cap\Gamma_0 = \emptyset\}.$$
In the scalar case, i.e. for $d=1$, we write $\C^\infty_\Gamma (\D)$ instead of $\C^\infty_\Gamma (\D,\R)$ and use a similar notation for the other function spaces. 
We denote by $W^1_p(\D)$, $1\le p \le \infty$ the Sobolev space of weakly differentiable functions with weak derivative in $L^p(\D)$. 
For $p,p'\geq 1$ satisfying $\frac{1}{p} + \frac{1}{p'} = 1$, we define the Sobolev space
$$W^1_{\Gamma,p}(\D,\R^d) := \overline{\C^\infty_\Gamma (\D,\R^d)}^{W^1_p}, $$
where $W^1_p$ stands for the usual norm in $W^{1,p}(\D,\R^d)$, and the dual space
$$W^{-1}_{\Gamma,p}(\D,\R^d) := (W^1_{\Gamma,p'}(\D,\R^d))^*. $$
We use the notation $\text{id}$ for the identity function in $\R^2$, and $\mI$ for the $2\times 2$ identity matrix.

Let $2\leq q <\infty$ and $1\leq q'\leq 2$ satisfying $\frac{1}{q} + \frac{1}{q'} = 1$.
Let $\ma \in L^\infty(\D)^{2\times 2}$ be a matrix-valued function satisfying for all $\eta,\theta\in \bbR^2$ and 
$x\in \overline \D$:
\begin{align}\label{assump2}
\ma(x)\theta\cdot \theta &\ge m|\theta|^2 \text{ and }    |\ma(x)\eta|  \le M|\eta|, \text{ with } m>0 \text{ and }M>0, 
\end{align}
where $|\cdot|$ denotes the Euclidean norm and $m\leq M$.
Introduce 
\begin{align*}
a: W^1_{\Gamma,q}(\D)\times W^1_{\Gamma,q'}(\D)   & \to\R \\
(v,w)& \mapsto \int_{\D} \ma\nabla v\cdot \nabla w.
\end{align*}
Then, define the corresponding operator
\begin{align}\label{Aq}
\begin{split}
\mathcal{A}_q: W^1_{\Gamma,q}(\D) &\to W^{-1}_{\Gamma,q}(\D),  \\
v & \mapsto \mathcal{A}_q v := a(v,\cdot).
\end{split}
\end{align}
Let $\Jop$ be defined by, for $u,v\in W^1_{\Gamma,2}(\D)$,
$$\langle \Jop u,v \rangle := \int_{\D} \nabla u\cdot \nabla v  + uv.$$ 
By H\"older's inequality it follows that  $\Jop: W^1_{\Gamma,p}(\D)\to W^{-1}_{\Gamma,p}(\D)$ is a well-defined and continuous operator for all $p\geq 2$.
We also introduce the constant
$$ M_p :=\sup\{ \| v\|_{W^1_p(\D)}\ | \ v\in  W^1_{\Gamma,p}(\D), \|\Jop v\|_{W^{-1}_{\Gamma,p}(\D)} \leq 1 \}.$$
It is easily verified that $M_2 =1$.
Now we define the set of regular domains in the sense of Gr\"oger
$$ \Xi : = \{ (\D,\Gamma)\ |\  \D\subset\R^2,\Gamma\subset\pd, \text{ and }\D\cup\Gamma\text{ is regular}\}.$$
\begin{definition}
Denote by $R_q$, $2\leq q<\infty$, the set of regular domains $(\D,\Gamma)\in\Xi$ for which $\Jop$ maps $W^1_{\Gamma,q}(\D)$ onto $W^{-1}_{\Gamma,q}(\D)$. 
\end{definition}
Then we have the following results from \cite[Lemma 1]{a_GRRE_1989a}.
\begin{lemma}\label{lemma01}
Let  $(\D,\Gamma)\in R_q$ for some $q>2$. 
Then $(\D,\Gamma)\in R_p$ for $2\leq p\leq q$ and $M_p \leq M_q^\theta$ if $\frac{1}{p} = \frac{1-\theta}{2} + \frac{\theta}{q}$.
\end{lemma}
We can now state an adapted version of \cite[Theorem 1]{a_GRRE_1989a} which plays a key role in our investigations.
%%%%%%%%%%%%%%%%
\begin{theorem}{\cite[Theorem 1]{a_GRRE_1989a}}\label{thm01}
Let  $(\D,\Gamma)\in R_{q_0}$ for some $q_0>2$. 
Suppose that $\ma$ satisfies assumptions \eqref{assump2} for $q_0$ and let $\mathcal{A}_q$ be defined by \eqref{Aq}. 
Then $\mathcal{A}_q: W^1_{\Gamma,q}(\D) \to W^{-1}_{\Gamma,q}(\D)$ is an isomorphism provided that $q\in [2,q_0]$ and $M_q k<1$, where $k:= (1-m^2/M^2)^{1/2}$, and
\begin{equation}\label{Aq_iso}
\| \mathcal{A}_q^{-1}\|_{L(W^{-1}_{\Gamma,q}(\D),W^1_{\Gamma,q}(\D))} \leq c_q, 
\end{equation}
where $c_q: = mM^{-2} M_q (1-M_q k)^{-1}$. Finally, $M_q k <1$ is satisfied if
$$\frac{1}{q}> \frac{1}{2} - \left(\frac{1}{2} - \frac{1}{q_0} \right) \frac{|\log k|}{\log M_{q_0}} . $$
\end{theorem}
%%%%%%%%%%
\begin{remark}\label{rem2.5}
\begin{itemize}
\item If $(\D,\Gamma)\in R_q$, then $M_q<\infty$.
\item For every regular $(\D,\Gamma)\in\Xi$, there exists a $q_0>2$ so that $(\D,\Gamma)\in R_{q_0}$; cf \cite[Theorem 3]{MR990595}.
\item For sufficiently small $q>2$, the constant $c_q$ in \eqref{Aq_iso} can be chosen to be independent of $q$; see \cite[Corollary 5]{MR3584578}.
\end{itemize}
\end{remark}
%%%%%%%%%%%%
We now explain how a particular case of the theory described above can be applied to our problem.
Let  $\sigma\in L^\infty(\D)$ satisfying pointwise a.e. $\overline{\sigma} \geq \sigma\geq \underline{\sigma}>0$ where $\overline{\sigma},\underline{\sigma}>0$ are constants. It is clear that $\ma:= \sigma \mI\in L^\infty(\D)^{2\times 2}$ satisfies assumptions~\eqref{assump2}. 
In view of Remark \ref{rem2.5}, there exists $q_0>2$ such that $(\D,\Gamma)\in R_{q_0}$.
For $q\in [2,q_0]$, $f\in L^q(\D)$ and $g\in L^\infty(\partial\D)$, the functional
\[
    \langle F,v\rangle := \int_\D fv  + \int_{\Gamma} gv  , \quad  v\in W_{\Gamma,q'}^1(\D)
\]
defines an element in $(W_{\Gamma,q'}^1(\D))^* =  W^{-1}_{\Gamma,q}(\D)$. 
Therefore, it follows from Theorem~\ref{thm01} that there is a unique $u\in W^1_{\Gamma,q}(\D)$ solution  to 
\[
    \int_{\D} \sigma \nabla u\cdot \nabla v   = \int_\D fv  + \int_{\Gamma} gv  \  \text{ for all } v\in W_{\Gamma,q'}^1(\D),
\]
provided $q\in ]2,q_0]$ is sufficiently close to $2$.

%%%%%%%%%%%%%%%%%%%%%%%%%%%%%%%%%%%
\subsection{Shape optimization framework}\label{sec:prel2}
%%%%%%%%%%%%%%%%%%%%%%%%%%%%%%%%%%%
In this section, we recall basic notions about first and second order Eulerian shape derivatives.
For $k\geq 0$  we define 
\begin{align*}
\C^k_c(\D,\R^2) &:=\{\VV\in \C^k(\D,\R^2)\ |\ \VV\text{ has compact support in } \D\},
\end{align*}
and $\C^\infty_c(\D,\R^2)$ similarly, and we equip these spaces with their usual topologies; see \cite[1.56, pp. 19-20]{MR2424078}.
Consider a vector field $\VV\in \C^1_c(\D,\R^2)$ and the associated flow
$\Tt:\D\rightarrow \D$, $t\in [0,\tz]$ defined for each $x_0\in \D$ as $\Tt(x_0):=x(t)$, where $x:[0,\tz]\rightarrow \R^d$ solves 
\begin{align}\label{Vxt}
\begin{split}
\dot{x}(t)&= \VV(x(t))    \quad \text{ for } t\in [0,\tz],\quad  x(0) =x_0.
\end{split}
\end{align}
Let $\OO$ be the set of all open sets  compactly contained in $\D$, where $\D\subset \R^2$ is assumed to be open and bounded.
For $\Om\in \OO$, 
we consider the family of perturbed domains  
\begin{equation}\label{domain}
\Omega_t := \Tt(\Omega). 
\end{equation}
%%%%%%%%%%%%%%%%%%%%%%%%%%%%%%%%%%%%%%%%%%%%%%%%%%%%%%%%
\begin{definition}[Shape derivative]\label{def1}
Let $J : \OO \rightarrow \R$ be a shape functional.
\begin{itemize}
\item[(i)] The Eulerian semiderivative of $J$ at $\Omega$ in direction $\VV \in \C^1_c(\D,\R^2)$
is defined by, when the limit exists,
\begin{equation}
d J(\Omega)(\VV):= \lim_{t \searrow 0}\frac{J(\Omega_t)-J(\Omega)}{t}.
\end{equation}
\item[(ii)] $J$ is said to be \textit{shape differentiable} at $\Omega$ if it has a Eulerian semiderivative at $\Omega$ for all $\VV \in \C^\infty_c(\D,\R^2)$ and the mapping
\begin{align*}
d J(\Omega): \C^\infty_c(\D,\R^2) &  \to \R,\; \VV     \mapsto d J(\Omega)(\VV)
\end{align*}
is linear and continuous, in which case $d J(\Omega)(\VV)$ is called the \textit{Eulerian shape derivative} at $\Omega$, or simply \textit{shape derivative} at $\Omega$.
\end{itemize}
\end{definition}

%%%%%%%%%%%%%%%%%%%%%%%%%%%%%%%%%%%
\section{EIT with point measurements}\label{sec:EIT_point}
%%%%%%%%%%%%%%%%%%%%%%%%%%%%%%%%%%%
\subsection{Problem formulation}
%%%%%%%%%%%%%%%%%%%%%%%%%%%%%%%%%%
Let $\D\cup\Gamma\in \Xi$ (see Definition~\ref{def1b}) and $\Om\in\OO$.
Denote $\Omc: = \D\setminus \Om$ and $\Omc_t: = T_t(\Omc)$.
In this section, $n$ denotes the outward unit normal vector to $\Om$. 
Introduce the conductivity $\sigma_\Om = \sigma_1\chi_{\Om} + \sigma_0\chi_{\Omc}$ with $(\sigma_0,\sigma_1)$ positive scalars, $\sigma_1 > \sigma_0$, and $f_\Om =f_1\chi_{\Om} + f_0\chi_{\Omc}$ where $f_0,f_1\in H^1(\D)$. 
Here, $\chi_\Omega$ denotes the characteristic function of $\Omega$. 
Let $p>2$ and $1< p'< 2$ satisfying $\frac{1}{p} + \frac{1}{p'} = 1$.
In view of the development in  Section \ref{sec:prel1}, $\sol\in W^1_{\Gamma,p}(\Om)$ is the solution to the mixed boundary value problem
\begin{equation}\label{E:var_form}
\int_\D \sigma_\Om \nabla \sol \cdot \nabla v = \int_\D f_\Om v + \int_{\Gamma} gv  \quad \mbox{ for all }v\in W^{1}_{\Gamma,p'}(\D), 
\end{equation}
with $g\in L^\infty(\pd)$. Observe that $u$ depends on $\Om$ through $\sigma_\Om$.

In EIT,  $g$ represents an input, in this case an electric current applied on the boundary, and $u$ is the corresponding potential.  
Then, measurements $\meas$ of the potential on a subset $\Gamma_h$ of $\overline{\D}$ are performed.
Given the Cauchy data $(g,h)$, the task is to find the best possible approximation of the unknown shape $\Om$.
To obtain a better reconstruction, we apply several input currents $g_i$, $i = 1,\dots,I,$ and the corresponding measurements are denoted by $h_i$.
Assuming $(\sigma_0,\sigma_1)$ are known and denoting $u_i$ the solution of \eqref{E:var_form} with $g=g_i$, 
the EIT problem becomes then: 
\begin{align}\label{EIT-SO}
\begin{split}
\mbox{given }& \mbox{$\{(g_i,h_i)\}_{i=1}^I$, find $\Om$ such that $\sol_i = h_i$ on $\Gamma_h$ for $i=1,\dots,I$.}
\end{split}
\end{align}
However, \eqref{EIT-SO} is idealized since in practice the measurements $h_i$ are corrupted by noise, therefore
we cannot expect that  $u_i = h_i$ be exactly achievable, but rather that $|u_i - h_i|$ should be minimized.
When $\Gamma_h$ is a manifold of one or two dimensions, a common approach is to minimize an appropriate cost functional such as 
\begin{align}
\label{eit3.2} J(\Om) &= \frac{1}{2}\sum_{i=1}^I\int_{\Gamma_h} (\sol _i -  h_i)^2. 
\end{align}
Another popular approach is to use a Kohn-Vogelius type functional; see \cite{MR3535238}.

In this paper we are interested in the case where $\Gamma_h = \{x_k\}_{k=1}^K\in \overline{\D}$ is  a finite set of points, i.e. we  only have a finite collection of point measurements.
In this case, a  Kohn-Vogelius type functional does not seem appropriate since we would need $\meas$ on all of $\partial\D$ for this approach. 
The functional \eqref{eit3.2} on the other hand can be adapted to the case $\Gamma_h = \{x_k\}_{k=1}^K$ in the following way.
For $i=1,\dots,I$, assume that measurements $\{h_i(x_k)\}_{k=1}^K\in\R^K$ are available.
For $\Om\in\OO$ and $x_0\in \overline{\D}$, we consider the shape functional
\begin{equation}\label{E:cost_full}
J(\Om) :=  \frac{1}{2}\sum_{i=1}^I \mu_i\sum_{k=1}^K  \delta_{x_k}((\sol_i - \meas_i)^2) = \frac{1}{2}\sum_{i=1}^I  \mu_i\sum_{k=1}^K  (\sol_i(x_k)- \meas_i(x_k))^2, 
\end{equation}
where $\delta_{x_k}: \C(\overline \D)\to \R$ is the Dirac measure concentrated at $x_k$ and $\mu_i$ are given constants.
Note that in view of the continuous embedding $W^1_{\Gamma,p}(\D)\subset \C(\overline \D)$ for $p>2$ in two dimensions, the point evaluation of $\sol_i$ in \eqref{E:cost_full} is well-defined. 
Without loss of generality, we will compute the shape derivative for the simpler case $I=1$ and $\mu_1 = 1$, in which case the cost functional becomes
\begin{equation}\label{E:cost}
  J(\Om) =   \frac{1}{2}\sum_{k=1}^K  \delta_{x_k}((\sol - \meas)^2) = \frac{1}{2}\sum_{k=1}^K (\sol(x_k)- \meas(x_k))^2.
\end{equation}
The formula of the shape derivative in the general case \eqref{E:cost_full} can then be obtained by summation.
%%%%%%%%%%%%%%%%%%%%%%%
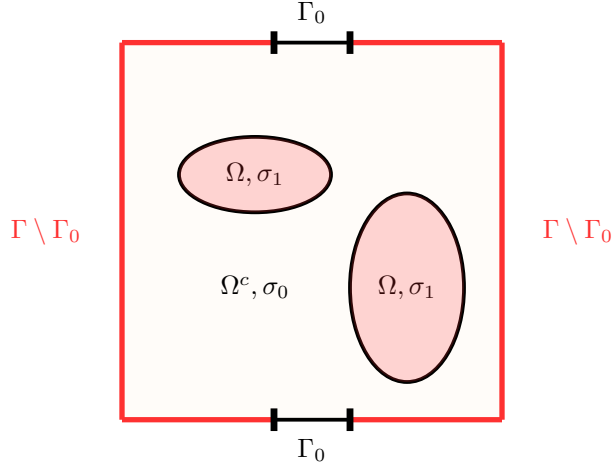
\begin{figure}
\begin{center}
\definecolor{preto}{RGB}{0,0,0}
\definecolor{vermelho}{RGB}{254,49,49}
\definecolor{azul_escuro}{RGB}{0,0,255}
\definecolor{cinza}{RGB}{195,195,195}
\definecolor{papel_amarelo}{RGB}{255, 249, 240}

\begin{tikzpicture}[scale=5]
% background square color	
\fill[papel_amarelo, opacity=0.4] (0,0) rectangle (1,1);
% grid lines
%\draw[step=0.2cm,gray,very thin, opacity=0.5] (1,0) grid (0,1);
% border lines left
\draw[line width=0.7mm, vermelho](0,1) -- (0.4,1);
\draw[line width=0.7mm, vermelho](0,0) -- (0,1);	
\draw[line width=0.7mm, vermelho] (0,0) -- (0.4,0);
% border lines right
\draw[line width=0.7mm, vermelho](0.6,1) -- (1,1);
\draw[line width=0.7mm, vermelho](1,0) -- (1,1);
\draw[line width=0.7mm, vermelho] (0.6,0) -- (1,0);
% tic top - right
\draw[line width=0.9mm, preto](0.6,0.97) -- (0.6,1.03);
% tic top - left
\draw[line width=0.9mm, preto](0.4,0.97) -- (0.4,1.03);
% tic bottom - right
\draw[line width=0.9mm, preto](0.6,-0.03) -- (0.6,0.03);
% tic bottom - left
\draw[line width=0.9mm, preto](0.4,-0.03) -- (0.4,0.03);
%%%%%%%%%%%%%%%%
\draw[line width=0.5mm, preto](0.4,1) -- (0.6,1);
%%%%%%%%%%%%%%%%
\draw[line width=0.5mm, preto](0.4,0) -- (0.6,0);
% ellipses : lengths in the brackets separated by an and, are the x-direction radius and the y-direction radius respectively
% ellipse above
\draw[line width=0.5mm, preto] (0.35,0.65) ellipse (0.2cm and 0.1cm) node[anchor= center] {$\Omega,\sigma_1$};
\fill[vermelho, opacity=0.2] (0.35,0.65) ellipse (0.2cm and 0.1cm);
% below below
\draw[line width=0.5mm, preto] (0.75,0.35) ellipse (0.15cm and 0.25cm) node[anchor= center] {$\Omega,\sigma_1$};
\fill[vermelho, opacity=0.2] (0.75,0.35) ellipse (0.15cm and 0.25cm);
%%%%%%%%%%%%%%%%%%%%%%
\draw (0.35, 0.35) node[anchor= center] {$\Omc,\sigma_0$};
%defining dots size
\def\raio{0.015}
% defining opacity level 
\def\transparenc{0.7}	
% defining step size	
\def\h{0.1};
% draw Gamma
\draw (0.5, -0.08) node[anchor= center] {$\Gamma_0$};
\draw (0.5, 1.08) node[anchor= center] {$\Gamma_0$};
\draw[vermelho] (1.2, 0.5) node[anchor= center] {$\Gamma\setminus\Gamma_0$};
\draw[vermelho] (-0.2, 0.5) node[anchor= center] {$\Gamma\setminus\Gamma_0$};
\end{tikzpicture}
\end{center}
\caption{Partition $\dd=\Omega\cup \Omega^c$. }\label{partition}
\end{figure}
%%%%%%%%%%%%%%%%%%
\subsection{Shape derivative}
%%%%%%%%%%%%%%%%%%%%%%%%%%%%%%%%%%
For $\Om\in\OO$ and $\VV \in \C^1_c(\D,\R^2)$, define $\Om_t$ as in \eqref{domain}.
Since $\VV$ has compact support in $\D$, we have $\Om_t\subset\D$  for all $t\in [0,\tz]$. 
Now we consider the Lagrangian  $\mathcal{L}:\OO\times W^1_{\Gamma,p}(\D)\times W^1_{\Gamma,p'}(\D)\rightarrow \R$ associated with the cost functional \eqref{E:cost} and the PDE constraint \eqref{E:var_form} defined by
%%%%%%%%%%%%%%%
\begin{align}
\label{G}
\begin{split}
\mathcal{L}(\Om,\varphi,\psi) & := \frac{1}{2}\sum_{k=1}^K (\varphi(x_k) -\meas(x_k))^2
+ \int_\D \sigma_{\Om} \nabla \varphi\cdot\nabla \psi - f_\Om\psi - \int_{\Gamma} g\psi.
\end{split}
\end{align}
Aiming at applying the averaged adjoint method \cite{a_ST_2015a}, we introduce the {\it shape-Lagrangian} $G:[0,\tz]\times W^1_{\Gamma,p}(\D)\times W^1_{\Gamma,p'}(\D) \rightarrow \R$ as
\begin{align*}
G(t,\varphi,\psi) &:= \mathcal{L}(\Om_t,\varphi\circ T_t^{-1},\psi\circ T_t^{-1})\\
&= \frac{1}{2}\sum_{k=1}^K (\varphi\circ T_t^{-1} - \meas)^2(x_k) + \int_{\D} \sigma_{\Omega_t}\nabla(\varphi\circ T_t^{-1})\cdot \nabla(\psi\circ T_t^{-1}) -f_{\Om_t} \psi\circ T_t^{-1}   - \int_{\Gamma} g\psi\circ T_t^{-1},
\end{align*}
see Appendix 1 for a detailed explanation. 
Notice that for all $p\ge 1$ we have $\varphi \in W^1_{\Gamma,p}(\D)$ if and only if $\varphi\circ T_t \in W^1_{\Gamma,p}(\D)$; see \cite[Theorem 2.2.2, p. 52]{b_ZI_1989a}.
Observe that 
\begin{align*}
\sigma_{\Om_t}\circ \Tt
& =\sigma_1\chi_{\Om_t}\circ \Tt + \sigma_0\chi_{\Omc_t}\circ \Tt 
= \sigma_1\chi_{\Om} + \sigma_0\chi_{\Omc}= \sigma_\Om,\\
f_{\Om_t}\circ \Tt
& =f_1\circ \Tt\, \chi_{\Om_t}\circ \Tt
+f_0\circ \Tt\, \chi_{\Omc_t}\circ \Tt
=f_1\circ \Tt\, \chi_{\Om}
+f_0\circ \Tt\, \chi_{\Omc},
\end{align*}
and we introduce the function $f^t := f_1\circ \Tt\, \chi_{\Om} +f_0\circ \Tt\, \chi_{\Omc}$.

Using the fact that $\Tt = \text{id}$ on $\partial\D$ and proceeding with the change of variables $x\mapsto T_t(x)$ inside the integrals in $G(t,\varphi,\psi)$, we obtain using the chain rule
\begin{equation}\label{G_reduced}
G(t,\varphi,\psi) = \frac{1}{2}\sum_{k=1}^K (\varphi\circ T_t^{-1} - \meas)^2(x_k)  + \int_{\D} \sigma_\Omega A(t)\nabla\varphi\cdot \nabla\psi - f^t\psi  - \int_{\Gamma} g\psi,
\end{equation}
where $A(t):= \Det(D T_t) D T_t^{-1}D T_t^{-\transp}$. 

For $t\in [0,\tz]$, let us define the perturbation $\mathcal{A}^t_q$ of $\mathcal{A}_q$ defined in \eqref{Aq} as follows:
\begin{align}\label{Aqt}
\begin{split}
\mathcal{A}^t_q: W^1_{\Gamma,q}(\D) &\to W^{-1}_{\Gamma,q}(\D),  \\
v & \mapsto \left(w\mapsto \langle\mathcal{A}^t_q v,w\rangle: =  \int_{\D} \sigma_\Omega A(t)\nabla v\cdot \nabla w \right).
\end{split}
\end{align}
By continuity of $t\mapsto A(t):[0,t_0] \to C(\overline{\D})^{2\times 2}$, for every $\epsilon>0$ there exists $\delta>0$ so that the following result (see \cite[Lemma 13]{MR3584578}) follows immediately:
\begin{align}
\label{A01} A(t)(x)\eta\cdot\eta & \geq (1-\epsilon) |\eta|^2\quad   \text{ for all }\eta\in\R^2 \text{ and all } (t,x)\in [0,\delta]\times\overline{\D},\\
\label{A02}  |A(t)(x)| & \leq 1+\epsilon\quad    \text{ for all } (t,x)\in [0,\delta]\times\overline{\D}.
\end{align}
The continuity properties \eqref{A01}-\eqref{A02} imply the following 
perturbed version of Theorem~\ref{thm01}.
%%%%%%%%%%%%%%%%
\begin{lemma}\label{lemma02}
For each  $(\D,\Gamma)\in \Xi$ there exists $q_0>2$, $\epsilon>0$ and $\delta>0$ so that for all $t\in [0,\delta]$ and all $q\in [2,q_0]$ satisfying $M_q k < 1$,  where $k:= (1-m^2/M^2)^{1/2}<1$ with $m = \underline{\sigma}(1-\epsilon)$ and $M = \overline{\sigma}(1+\epsilon)$, the mapping  $\mathcal{A}^t_q: W^1_{\Gamma,q}(\D) \to W^{-1}_{\Gamma,q}(\D)$ defined by \eqref{Aqt} is an isomorphism.
Moreover, we have for all $t\in [0,\delta]$ that
\begin{equation}\label{Aq_iso_t}
\|  (\mathcal{A}^t_q)^{-1} \|_{L(W^{-1}_{\Gamma,q}(\D),W^1_{\Gamma,q}(\D))} \leq c_q ,
\end{equation}
where $c_q: = mM^{-2} M_q (1-M_q k)^{-1}$ is independent of $t$. 
Finally, $M_q k <1$ is satisfied if
\[
    \frac{1}{q}> \frac{1}{2} - \left(\frac{1}{2} - \frac{1}{q_0} \right) \frac{|\log k|}{\log M_{q_0}} . 
\]
\end{lemma}
%%%%%%%%%%
\begin{proof}
We have $\overline{\sigma} \geq \sigma\geq \underline{\sigma}>0$ with $ \underline{\sigma}: = \min\{\sigma_0,\sigma_1\}$, $\overline{\sigma}: = \max\{\sigma_0,\sigma_1\}$ and $\underline{\sigma}< \overline{\sigma}$. 
Let us choose $\epsilon<1 $ and $\delta$ such that \eqref{A01},\eqref{A02} is satisfied, and let $t\in [0,\delta)$. 
In view of \eqref{A01},\eqref{A02} it immediately follows that $\ma = \sigma_\Omega A(t)\in L^\infty(\D)^{2\times 2}$ satisfies assumptions~\eqref{assump2} with $m := \underline{\sigma}(1-\epsilon)$ and $M:= \overline\sigma(1+\epsilon)$. 
Hence, the result follows directly from Theorem \ref{thm01}, since $M$ and $m$ are independent of $t$.
\end{proof}
%%%%%%%%%%%%%%%%%%%%%%%%%%%
The main statement of this section is the following theorem. 
\begin{theorem}[distributed shape derivative]\label{T:shape}
Let $\D\cup\Gamma\subset\R^2$ be a regular domain in the sense of Gr\"oger and $\Om\in\OO$. 
Assume that $\Gamma_h\cap\po =\emptyset$ and $f_\Om =f_1\chi_{\Om}+f_0\chi_{\Omc}$ where $f_0,f_1\in H^1(\D)$.
Then the shape derivative of $J$ at $\Omega$ in direction $\VV\in\C^1_c(\D,\R^2)$
is given by 
\begin{equation}
\label{T:tensor_shape_deriv}
dJ(\Om)(\VV) 
= \VS_0(\VV) + \int_{\D} \VS_1: D \VV,
\end{equation}
where $\VS_1\in L^1(\D)^{2\times 2}$ and $\VS_0\in (\C^0(\overline{\D},\R^2))^*$ are defined by 
\begin{align}
\label{S1_1}\Sb_1 & = - 2 \sigma_\Om \nabla u\odot\nabla p + (\sigma_\Om\nabla u\cdot\nabla p -fp)\mI,\\
\Sb_0(\VV) & =  \Sb_0^s(\VV) +  \int_\D \Sb_0^r \cdot \VV\\ 
\Sb_0^r & = - p\widetilde{\nabla} f  \\
\Sb_0^s & = - \sum_{k=1}^K \bigg((u - \meas)\nabla u\bigg)(x_k)\delta_{x_k},  
\end{align}
where $\nabla u\odot\nabla p := (\nabla u\otimes\nabla p + \nabla p\otimes\nabla u)/2$,
$\widetilde{\nabla} f := \nabla f_1 \, \chi_{\Om} + \nabla f_0\, \chi_{\Omc}$ and $\Sb_0^r\in L^1(\D,\R^2)$.

Also, there exists $q>2$ such that the adjoint $p\in W^1_{\Gamma,q'}(\D) $ is the solution to
\begin{equation}
  \int_\D \sigma_\Omega \nabla p \cdot \nabla \varphi = -\sum_{k=1}^K (\sol(x_k) -h(x_k))\varphi(x_k)  \quad \mbox{ for all } \varphi\in W^1_{\Gamma,q}(\D).
\end{equation}
\end{theorem}
%%%%%%%%%%%%%%
\begin{proof}
We employ the averaged adjoint approach of \cite{a_ST_2015a}, we refer to Appendix 1 for details about the method. 
We closely follow  the argumentation of \cite{MR3584578}. 
Let us define the perturbed state $u^t\in W^1_{\Gamma,q}(\D)$ solution of 
\ben\label{eq:perturbed_state}
\int_\D \sigma_\Om A(t) \nabla u^t \cdot \nabla \varphi  
= \int_\D f^t \varphi  
+ \int_{\Gamma} g\varphi \quad \text{ for all } \varphi \in W^1_{\Gamma,q'}(\D).
\een 
The mapping $F_t : W^1_{\Gamma,q'}(\D)\to \bbR$ defined by
\[
\langle F_t,v \rangle := \int_\D f^t v  + \int_{\Gamma} gv  \quad \text{ for } v\in W^1_{\Gamma,q'}(\D) 
\]
is well-defined and continuous. 
Consequently, thanks to Theorem~\ref{thm01} there is a unique solution to \eqref{eq:perturbed_state} in $W^1_{\Gamma,q}(\D)$ for $q>2$ sufficiently close to $2$.  
Using \eqref{Aq_iso_t} we get
\begin{align*}
 \|  u^t\|_{W^1_{\Gamma,q}(\D)} \leq c_q \| F_t \|_{W^{-1}_{\Gamma,q}(\D)}
 \leq   C(\| f^t \|_{L^2(\D)} + \| g \|_{L^\infty(\partial\D)}).
\end{align*}
It follows that for some constant $C$ independent of $t$, we have
\begin{align}\label{006}
 \|  u^t\|_{W^1_{\Gamma,q}(\D)} \leq C.
\end{align}

Following \eqref{averated_}, the {\it averaged adjoint equation} reads: find $p^t\in W^1_{\Gamma,q'}(\D)$, such that
\ben\label{E:averaged1}
\int_0^1 d_\varphi G(t,su^t + (1-s)u^0; p^t)(\varphi)\;ds =0 \quad \text{ for all } \varphi\in W^1_{\Gamma,q}(\D),
\een
which is equivalent to, using the fact that $A(t)^\transp = A(t)$, 
\begin{align}\label{E:averaged2}
\begin{split}
& \int_\D \sigma_\Omega A(t)\nabla p^t \cdot  \nabla \varphi  \\
& \hspace{1cm} = \frac{1}{2}\sum_{k=1}^K (u^t\circ T_t^{-1}(x_k) + u^0\circ T_t^{-1}(x_k) - 2\meas(x_k))\varphi\circ T_t^{-1}(x_k) \quad \text{ for all } \varphi\in W^1_{\Gamma,q}(\D).
\end{split}
\end{align}
Let us introduce the adjoint operator
\begin{align}\label{Aqt_adjoint}
\begin{split}
(\mathcal{A}^t_q)^*: W^{-1}_{\Gamma,q}(\D)^* = W^1_{\Gamma,q'}(\D)  &\to W^1_{\Gamma,q}(\D)^* = W^{-1}_{\Gamma,q'}(\D),  \\
w & \mapsto \left(v\mapsto \langle(\mathcal{A}^t_q)^* w,v\rangle: =  \langle w, \mathcal{A}^t_q v\rangle \right).
\end{split}
\end{align}
Using \eqref{Aqt} and  the fact that $A(t)^\transp = A(t)$ we get for $w\in  W^1_{\Gamma,q'}(\D)$ and $v\in W^1_{\Gamma,q}(\D)$,
\[
\langle(\mathcal{A}^t_q)^* w,v\rangle = \int_{\D} \sigma_\Omega A(t)\nabla w\cdot \nabla v.
\]
Now, in view of Lemma \ref{lemma02} there exists $q>2$ and $\delta>0$ such that the mapping  $\mathcal{A}^t_q: W^1_{\Gamma,q}(\D) \to W^{-1}_{\Gamma,q}(\D)$  is an isomorphism for all $t\in [0,\delta]$.
Thus, the adjoint mapping $(\mathcal{A}^t_q)^*: W^1_{\Gamma,q'}(\D) \to W^{-1}_{\Gamma,q'}(\D)$  is also an isomorphism.

Now the functional $R_t: W^1_{\Gamma,q}(\D)\to \bbR$ defined by
\[
\langle R_t ,v \rangle := \frac{1}{2}\sum_{k=1}^K (u^t\circ T_t^{-1}(x_k) + u^0\circ T_t^{-1}(x_k) - 2\meas(x_k))v\circ T_t^{-1}(x_k)  \text{ for } v\in W^1_{\Gamma,q}(\D). 
\]
is well-defined and continuous. 
Therefore, since $(\mathcal{A}^t_q)^*$ is an isomorphism,  the averaged adjoint equation $(\mathcal{A}^t_q)^* p^t = R_t$ has a unique solution $p^t\in W_{\Gamma,q'}^1(\D)$.

Using the continuous embedding of $W^1_{\Gamma,q}(\D)$ into the space of continuous functions $\C(\overline{\D})$ for $q>2$ in two dimensions, it also follows that 
\begin{align*}
\|p^t\|_{W^1_{\Gamma,q'}(\D)}
& \le  C \max_{k\in\{1,\dots,K\}} |(u^t\circ T_t^{-1} + u^0\circ T_t^{-1}-2\meas )(x_k)|\\
& \le C \left(\|u^t\|_{W^1_{\Gamma,q}(\D)} + \|u^0\|_{W^1_{\Gamma,q}(\D)} +\max_{k\in\{1,\dots,K\}}|\meas(x_k)| \right) . 
\end{align*}
Then using \eqref{006} we get, for some constant $C$ independent of $t$,
\begin{align}\label{007}
\|p^t\|_{W^1_{\Gamma,q'}(\D)}\leq C.
\end{align}
With the estimate \eqref{007} we readily verify that $p^t\rightharpoonup p^0$ weakly in $W^1_{\Gamma,q'}(\D)$ as $t\searrow 0$. 
Using \eqref{eq:main_averaged} we have
\ben
  \frac{G(t,u^t,p^t)-G(0,u^0,p^0)}{t} 
  = \frac{G(t,u^0,p^t)-G(0,u^0,p^t)}{t} 
  \een
and then in view of \eqref{G_reduced}
\begin{equation}\label{E:rhs_diff_G}
\begin{split}
  \frac{G(t,u^0,p^t)-G(0,u^0,p^t)}{t} & = \frac{1}{2} \sum_{k=1}^K \frac{(u^0\circ T_t^{-1} - \meas)^2(x_k) - (u^0 - \meas)^2(x_k) }{t}\\
  &\qquad  + \int_\D \sigma_\Omega \frac{A(t)-\mI}{t} \nabla u^0\cdot \nabla p^t - \frac{f^t - f^0}{t} p^t.
  \end{split}
\end{equation}
Using the assumption $\Gamma_h\cap\po =\emptyset$, we have for all $k=1,\dots,K$  that  $x_k$ belongs either to $\Om$, to $\D\setminus\overline{\Om}$ or to $\partial\D$.
Assume first that $x_k$ belongs either to $\Om$ or to $\D\setminus\overline{\Om}$.
Since $\sigma_\Om$ is constant in $\Om$ and in $\D\setminus\overline{\Om}$, $u$ is harmonic in these sets, therefore using elliptic regularity results we have $u\in\C^\infty(B(x_k,r_k))$ for sufficiently small $r_k$, where $B(x_k,r_k)$ denotes the open ball of center $x_k$ and radius $r_k$.
Thus, the first term on the right hand side of \eqref{E:rhs_diff_G} converges as $t\searrow 0$.
Now if $x_k\in\partial\D$, then $T_t(x_k)=x_k$ due to $\VV\in\C^1_c(\D,\R^2)$, and the first term on the right hand side of \eqref{E:rhs_diff_G} is equal to zero, so we obtain the same formula as in the case  $x_k\in\D\setminus\overline{\Om}$.
Also, using  $\VV\in\C^1_c(\D,\R^2)$ we have the following convergence properties (see \cite[Lem. 3.1]{a_KAKUST_2018a} and \cite[Lemma 2.16]{phdKevin}) 
\begin{align}
\frac{A(t)-I}{t} \rightarrow A'(0) & := \Div(\VV) - D \VV - D \VV^\transp \quad \text{ strongly in } \C(\overline \D)^{2\times 2}, \\
\frac{f^t - f^0}{t} \rightarrow \widetilde  f'(0) & := f_\Om\Div(\VV) + \widetilde\nabla  f\cdot \VV \quad \text{ strongly in } L^2(\D),
\end{align}
and we conclude that the right hand side of \eqref{E:rhs_diff_G} converges to 
\ben\label{E:derivative_proof}
- \sum_{k=1}^K (u^0 - \meas )(x_k) \nabla u^0 (x_k) \cdot \VV(x_k) 
+ \int_\D \sigma_\Omega A'(0)\nabla u^0\cdot \nabla p^0 - \widetilde f'(0) p^0 . 
\een
In view of \eqref{G_reduced} this shows 
\ben
\lim_{t\searrow 0}\frac{G(t,u^t,p^t)-G(0,u^0,p^0)}{t} = \partial_t G(0,u^0,p^0),
\een
which shows that Assumption \ref{H1} is satisfied.
Using tensor calculus, it is then readily verified  that \eqref{E:derivative_proof} can be brought into  expression \eqref{T:tensor_shape_deriv}.
The regularity $\VS_1\in L^1(\D)^{2\times 2}$ is due to $u\in W^1_{\Gamma,q}(\D)$, and $p\in W^1_{\Gamma,q'}(\D)$ and the regularity of $\VS_0^r$ is a consequence of the regularity of $p$ and $f_\Om$. 
\end{proof}
%%%%%%%%%%%%%%%%%%
An interesting feature of Theorem 3.2 is to show that the distributed shape derivative exists when $\Om$ is only open.
Another relevant issue is to determine the minimal regularity of $\Om$ for which we can obtain the boundary expression of the shape derivative.
The rest of this section is devoted to the study of this question.
We start with the following well-known result which describes the structure of the boundary expression of the shape derivative; see \cite[pp. 480-481]{MR2731611}. 
%%%%%%%%%%%%%%%%%%%
\begin{theorem}[Zol\'esio's structure theorem]\label{thm:structure_theorem}
Let $\Om$ be open with $\partial \Omega $ compact and of class $\C^{k+1}$, $k\geq 0$.
Assume $J$ has a Eulerian shape derivative at $\Om$ and $d J(\Omega )$ is continuous for the $\C^k(\D,\R^d)$-topology.
Then, there exists a linear and continuous functional $l: \C^k(\partial \Omega )\rightarrow \R$ such that
\begin{equation}
\label{volume}
d J(\Omega)(\VV)= l(\VV_{|\partial \Omega }\cdot n)\  \text{ for all } \VV\in  \C^k_c(\D,\R^d).
\end{equation}
\end{theorem}
%%%%%%%%%%%%%%%%%
Theorem \ref{thm:structure_theorem} requires $\Om$ to be at least $\C^1$, however we show in Proposition \ref{tensor_relations} that even for Lipschitz domains one can obtain a boundary expression for the shape derivative, see \eqref{158}, even though we get a weaker structure than \eqref{volume} since the tangential component of $\VV$ may be present in  \eqref{158}.  
Recall that a bounded domain is called Lipschitz if it is locally representable as the graph of a Lipschitz function.
In this case, it is well-known that the surface measure  is well-defined on $\partial\Om$ and there exists an outward pointing normal vector $n$ at almost every point on $\partial\Om$; see  \cite[Section 4.2, p. 127]{MR1158660}.

In the rest of the paper, the exponents $+$ and $-$ denote the restrictions of functions to $\Omega$ and to $\D\setminus\overline \Omega$, respectively, and the notation  $\llbracket \phi\rrbracket : = \phi^+|_{\partial\Om} - \phi^-|_{\partial\Om} $ denotes the jump across the interface $\partial\Om$ of a given function $\phi$. 
%%%%%%%%%%%%%%
\begin{proposition}\label{tensor_relations}
Suppose that $\Gamma_h\cap\po =\emptyset$,  $\Omega\in\OO$ and $\VV \in \C^1_c(\D\setminus \Gamma_h,\R^2)$, then we have 
\begin{align} \label{eq:equvilibrium_strong1}
\divv(\Sb_1^+)  &= (\Sb_0^r)^+ \quad \text{ a.e. in } \Omega\setminus \Gamma_h, \\
\label{eq:equvilibrium_strong2} \divv(\Sb_1^-) &= (\Sb_0^r)^-  \quad \text{ a.e. in } (\D\setminus\overline \Omega)\setminus \Gamma_h.
\end{align}
If  $\Sb_1^+\in W^{1,1}(\Om,\R^{2\times 2})$ and $\Sb_1^-\in W^{1,1}(\dd\setminus\overline{\Om},\R^{2\times 2})$, then
\begin{equation}\label{eq:first_order_tensor_2}
dJ(\Omega)(\VV) = \int_\Om \divv(\Sb_1^\transp\VV) 
+ \int_{\D\setminus \overline \Omega} \divv(\Sb_1^\transp\VV). %\,dx  .
\end{equation}
If in addition $\Om$ is Lipschitz, we also have the boundary expression
\begin{equation}\label{158}
dJ(\Omega)(\VV) = \int_{\partial \Omega} \llbracket \Sb_1 \rrbracket  n \cdot\VV . 
\end{equation}
If in addition $\Om$ is of class $\C^1$, we obtain the boundary expression 
\begin{equation}\label{eq:general_boundary_exp}
dJ(\Omega)(\VV) = \int_{\partial \Omega} (\llbracket \Sb_1 \rrbracket   n\cdot n )\, \VV\cdot n .%\, ds.
\end{equation}
\end{proposition}
%%%%%%%%%%%%%
\begin{proof} In view of \cite[Theorem 2.2]{MR3535238}, if $\VV$ has compact support in $\Om$ then the shape derivative vanishes.
Assume $\VV \in \C^{1}_c(\Om\setminus \Gamma_h,\R^2)$ and denote $U:=\operatorname{supp}\VV\subset \Om\setminus \Gamma_h$, then $u$ and $p$ are clearly harmonic on $U$ since $\sigma$ is constant on $U$. 
In view of \eqref{S1_1} and the regularity of $f_\Om$, this yields $\Sb_1\in L^1(U)$ and  $\divv(\Sb_1)\in L^1(U)$.
Thus, we can write the tensor relation $\divv(\Sb_1^\transp \VV) = \Sb_1 : D\VV + \VV\cdot \divv(\Sb_1)$.
For such $\VV$ we also have $ \VS^s_0(\VV)   = 0$, so we obtain
\begin{align}
\notag dJ(\Omega)(\VV) & = \VS^s_0(\VV)   +  \int_\D \Sb_1 : D\VV +\Sb_0^r\cdot \VV \\
\label{eq:77} & = \int_U  \divv(\Sb_1^\transp\VV) +\VV\cdot( \Sb_0^r - \divv\Sb_1)
= 0\qquad \mbox{ for all }\VV \in \C^{1}_c(\Om\setminus \Gamma_h,\R^2). 
\end{align} 
Since $\operatorname{supp}\VV =U \subset \Om\setminus \Gamma_h$, we can extend $\Sb_1^\transp\VV$ and $\VV\cdot( \Sb_0^r - \divv\Sb_1)$ by zero on $\mathcal{B}$, where $\mathcal{B}$ is a sufficiently large open ball which contains $U$.  
We keep the same notation for the extensions for simplicity.
Since the extension satisfies $\Sb_1^\transp\VV\in W^{1,1}(\mathcal{B},\R^2)$, using the divergence theorem (for instance \cite[Section 4.3, Theorem 1]{MR1158660}) in $\mathcal{B}$ we get
\begin{align*}
\int_U  \divv(\Sb_1^\transp\VV) +\VV\cdot( \Sb_0^r - \divv\Sb_1)
& = \int_{\mathcal{B}}  \divv(\Sb_1^\transp\VV) +\VV\cdot( \Sb_0^r - \divv\Sb_1)\\
&=  \int_{\partial\mathcal{B}} (\Sb_1^\transp\VV)\cdot n +\int_{\mathcal{B}}\VV\cdot( \Sb_0^r - \divv\Sb_1)\\
& =  \int_{\Om}\VV\cdot( \Sb_0^r - \divv\Sb_1) = 0, \qquad\mbox{ for all }\VV\in \C^{1}_c(\Om\setminus \Gamma_h,\R^2), 
\end{align*}
which proves \eqref{eq:equvilibrium_strong1}.
Then, we can prove  \eqref{eq:equvilibrium_strong2} in a similar way by taking a vector $\VV\in \C^{1}_c((\D\setminus\overline \Omega)\setminus \Gamma_h,\R^2)$.

Now let us assume that $\VV \in \C^{1}_c(\D\setminus \Gamma_h,\R^2)$ and denote  $U_2:=\operatorname{supp}\VV\subset \D\setminus \Gamma_h$.
By standard elliptic regularity, we have $u\in H^1(U_2)$ and $p\in H^1(U_2)$.
If we assume $\Sb_1^+\in W^{1,1}(\Om,\R^{2\times 2})$ and $\Sb_1^-\in W^{1,1}(\dd\setminus\overline{\Om},\R^{2\times 2})$, 
taking $\VV \in \C^{1}_c(\D\setminus \Gamma_h,\R^2)$ and using \eqref{eq:equvilibrium_strong1}-\eqref{eq:equvilibrium_strong2} we obtain
\begin{align*}
dJ(\Omega)(\VV) & = \VS_0(\VV)   +  \int_\D \Sb_1 : D\VV +\Sb_0^r\cdot \VV \\
& = \int_{\Om} \Sb_1 : D\VV +\Sb_0^r\cdot \VV + \int_{\D\setminus \overline \Omega} \Sb_1 : D\VV +\Sb_0^r\cdot \VV\\
& = \int_{\Om}  \divv(\Sb_1^\transp\VV) +\VV\cdot( \Sb_0^r - \divv\Sb_1)
+ \int_{\D\setminus \overline \Omega}  \divv(\Sb_1^\transp\VV) +\VV\cdot( \Sb_0^r - \divv\Sb_1)\\ 
& = \int_\Om \divv(\Sb_1^\transp\VV) 
+ \int_{\D\setminus \overline \Omega} \divv(\Sb_1^\transp\VV),
\end{align*} 
which yields \eqref{eq:first_order_tensor_2}.
If in addition $\Om$ is Lipschitz, applying the divergence theorem to \eqref{eq:first_order_tensor_2} we get \eqref{158}.

In view of \eqref{158}, we have that $dJ(\Omega )$ is continuous for the $\C^0(\dd ,\R^d)$-topology.
Thus, if $\Om$ is of class $\C^1$, we can apply Theorem  \ref{thm:structure_theorem} with $k=0$.
With $\Om$ of class $\C^1$, we also have $n\in\C^0(\po,\R^2)$ and $(\VV_{|\partial \Omega }\cdot n)n\in\C^0(\po,\R^2)$.
Let $\hat\VV\in\C^0_c(\D\setminus \Gamma_h ,\R^2)$ be an extension of $(\VV_{|\partial \Omega }\cdot n)n$, then using Theorem  \ref{thm:structure_theorem} we obtain
\begin{align*}
dJ(\Omega)(\VV) 
= l(\VV_{|\partial \Omega }\cdot n)
& = l(\hat\VV_{|\partial \Omega }\cdot n)
= dJ(\Omega)(\hat\VV) \\
= \int_{\partial \Omega} ((\Sb_1^+ - \Sb_1^-) n) \cdot\hat\VV 
& = \int_{\partial \Omega} (\llbracket \Sb_1 \rrbracket  n) \cdot ((\VV\cdot n) n) 
= \int_{\partial \Omega} (\llbracket \Sb_1 \rrbracket  n\cdot n) \VV\cdot n, 
\end{align*}
which yields expression \eqref{eq:general_boundary_exp}.
\end{proof}
\begin{remark}
Proposition \ref{tensor_relations} is in fact valid for any shape functional whose distributed shape derivative can be written using a tensor expression of the type  \eqref{T:tensor_shape_deriv}, and which satisfies the appropriate regularity assumptions. 
Note that in general, one should not expect that the assumption  $\Sb_1^+\in W^{1,1}(\Om,\R^{2\times 2})$ and $\Sb_1^-\in W^{1,1}(\dd\setminus\overline{\Om},\R^{2\times 2})$ in Proposition \ref{tensor_relations} can be satisfied for any Lipschitz set $\Om$.
For instance in the case of the Dirichlet Laplacian, one can actually build pathological Lipschitz domains for which $\Sb_1$ does not have such regularity; see 
\cite[Corollary 3.2]{Costabel2019}.
However, these regularity assumptions for $\Sb_1^+,\Sb_1^-$ can be fulfilled for polygonal domains, as shown in Corollary \ref{cor:11}.
\end{remark}
%%%%%%%%%%%%%%%%
\begin{corollary}\label{cor:11}
Suppose that $\Gamma_h\cap\po =\emptyset$, $f_0\in\C^\infty(\D)$, $\Omega\in\OO$ and $\VV \in \C^{1}_c(\D\setminus \Gamma_h,\R^2)$.
If $\Om$ is Lipschitz polygonal or if $\Om$ is of class $\C^1$, then we have
\begin{align}\label{bdr_expr}
dJ(\Omega)(\VV) = \int_{\partial \Omega} (\llbracket \sigma \dn \sol \dn p\rrbracket  
+ \llbracket\sigma\rrbracket  \nabla_{\partial \Omega} \sol\cdot\nabla_{\partial \Omega} p  -\llbracket f_\Om\rrbracket  p) \VV\cdot n,
\end{align} 
where $ \nabla_{\partial \Omega}$ denotes the tangential gradient on $\partial\Om$.
\end{corollary}
\begin{proof}
In the case where $\Om$ is of class $\C^1$, a quick calculation using \eqref{eq:general_boundary_exp} and \eqref{S1_1} yields \eqref{bdr_expr}.

In the case where $\Om$ is polygonal, we can proceed in the following way. 
Let $\widehat\D$ be a smooth open set such that $\overline{\Om}\subset\widehat\D\subset\D$ and the boundaries of $\Om$ and $\widehat\D$ are at a positive distance. 
Since $f_0\in\C^\infty(\D)$, using elliptic regularity  we get that $u$ and $p$ are $\C^\infty$ on $\partial\widehat\D$.
Thus, $u|_{\widehat\D}$ and $p|_{\widehat\D}$ are also solutions of transmission problems defined in $\widehat\D$ with smooth inhomogeneous Dirichlet conditions on $\partial\widehat\D$, and consequently we are in the framework considered in  \cite{MR1274152}.
Denote $L$ the number of vertices of the polygon $\Om$.
We apply \cite[Theorem 7.3]{MR1274152} in the case $k=0$, $m=1$ and for the regularity $W^{2,4/3}$.
This yields the decomposition $u|_{\widehat\D} = u_0 + \sum_{\ell\in L} S_\ell$ with $u_0^+\in  W^{2,4/3}(\Om)$, $u_0^-\in  W^{2,4/3}(\widehat\dd\setminus\overline{\Om})$ and $S_\ell$ are singular functions with support in the neighbourhood of the vertices of $\Om$. 
Here $S_\ell(r_\ell,\theta_\ell)$ are of the type $r_\ell^{\lambda_\ell} v(r_\ell,\theta_\ell)$, where $(r_\ell,\theta_\ell)$ are local polar coordinates at the vertex $\ell$ and $v(\theta_\ell)$ is a linear combination of $\sin(\lambda_\ell\theta_\ell)$ and $\cos(\lambda_\ell\theta_\ell)$. 
It is shown in \cite[Theorem 8.1(ii)]{MR1713241} that $\lambda_\ell>1/2$ for all $\ell=1,\dots, L$.
Thus, we also obtain $\sum_{\ell\in L} S_\ell^+\in  W^{2,4/3}(\Om)$ and $\sum_{\ell\in L} S_\ell^-\in  W^{2,4/3}(\widehat\dd\setminus\overline{\Om})$.

Proceeding in a similar way for $p$ and gathering the results, we obtain the regularity $u^+,p^+\in W^{2,4/3}(\Om)$ and $u^-,p^-\in W^{2,4/3}(\widehat\dd\setminus\overline{\Om})$.
Then we have $\nabla(\nabla u\cdot\nabla p) = D^2 u p + D^2p u$ and using $(D^2u)^+,(D^2p)^+\in L^{4/3}(\Om)$ and $(\nabla u)^+,(\nabla p)^+\in W^{1,4/3}(\Om)\subset L^4(\Om)$ and the same regularity on $\widehat\dd\setminus\overline{\Om}$, we obtain  $\Sb_1^+\in W^{1,1}(\Om,\R^{2\times 2})$ and $\Sb_1^-\in W^{1,1}(\widehat\dd\setminus\overline{\Om},\R^{2\times 2})$.

Then, using the fact that $\VV \in \C^{1}_c(\D\setminus \Gamma_h,\R^2)$ we obtain in view of \eqref{158} of Proposition \ref{tensor_relations}
\begin{align*}
dJ(\Omega)(\VV) 
& = \int_{\partial \Omega} 
\llbracket \sigma \dn \sol\rrbracket   \nabla_{\partial \Omega} \sol\cdot\VV
+ \llbracket \sigma \dn p\rrbracket   \nabla_{\partial \Omega} p\cdot\VV
+(\llbracket \sigma \dn \sol \dn p\rrbracket 
+ \llbracket\sigma\rrbracket  \nabla_{\partial \Omega} \sol\cdot\nabla_{\partial \Omega} p  
-\llbracket f_\Om\rrbracket  p) \VV\cdot n.
\end{align*} 
Finally, using the fact that $\llbracket \sigma \dn \sol\rrbracket =0$ and $\llbracket \sigma \dn p\rrbracket  = 0$ we obtain \eqref{bdr_expr}.
\end{proof}

\begin{remark}
Expressions similar to \eqref{bdr_expr} are known when $\Om$ is at least $\C^1$, see \cite{MR3535238} and \cite{MR2329288}.
It is remarkable that one obtains the same expression \eqref{bdr_expr} when $\Om$ is only Lipschitz polygonal.
Also, note that \eqref{bdr_expr} is similar to the formula obtained in \cite{MR3723652} for a polygonal inclusion in EIT, which was obtained in the framework of the perturbation of identity method. 
In \cite{MR3723652}, an estimate of the singularity of the gradient in the neighbourhood of the vertices of the polygonal inclusion was used to obtain the boundary expression. 
Here, we have used higher regularity of $u$ and $p$ in the subdomains  $\Om$ and $\widehat\dd\setminus\overline{\Om}$ to obtain \eqref{bdr_expr}.
The key idea of these two approaches is to control the singularity of the gradients of $u$ and $p$ near the vertices of the polygonal inclusion.  
\end{remark}

%%%%%%%%%%%%%%%%%%%%%%%%%%%%%%%%%%%%%%%%%%%
\section{Numerical experiments}\label{sec:numerics}
%%%%%%%%%%%%%%%%%%%%%%%%%%%%%%%%%%%%%%%%%%%%%
We use the software package FEniCS for the implementation; see \cite{ans20553,langtangen2017solving,fenics:book}.  
For the numerical tests  the conductivity values are set to $\sigma_0 = 1$ and $\sigma_1 = 10$. 
We choose $f_\Om\equiv 0$,  $\D=(0,1)\times(0,1)$ and 
$$\Gamma = \partial D\setminus ([0.4,0.6]\times \{0\}\cup [0.4,0.6]\times \{1\}).$$
The domain $\D$ is meshed using a regular grid of $128 \times 128$ elements.
For the measurement points we choose $\Gamma_h = \{x_k\}_{k=1}^{K}\subset\Gamma$.
Recall that no measurements are performed on $\Gamma_0=\partial\D\setminus\Gamma$ and that $u$ satisfies Dirichlet boundary condition on $\Gamma_0$.

Synthetic measurements $\{h_i(x_k)\}_{k=1}^K$ are obtained by taking the trace on $\Gamma$ of the solution
of \eqref{E:var_form} using the ground truth domain $\Om^\star$,  $f_{\Om^\star}\equiv 0$ and  currents $g_i$, $i = 1, \dots, I$. 
To simulate noisy EIT data, each measurement $h_i$
is corrupted by adding a normal Gaussian noise with mean zero and standard deviation $\delta * \|h_i\|_{\infty}$ , where $\delta$ is
a parameter. 
The noise level is then computed as
\begin{align}
noise =& \frac{\sum_{i=1}^I (\sum_{k=1}^K |h_i(x_{k}) - \tilde{h}_i(x_k)|^2)^{1/2}}{\sum_{i=1}^I (\sum_{k=1}^K |h_i(x_k)|^2)^{1/2}},
\end{align}
where $h_i(x_k)$ and $\tilde{h}_i(x_k)$ are respectively the noiseless and noisy point measurements at $x_k$ corresponding to the current $g_i$.

In the numerical tests, we use two different sets of fluxes, i.e. $I\in\{3,7\}$, to obtain measurements.
Denote $\Gamma_{\rm{upper}}$, $\Gamma_{\rm{lower}}$, $\Gamma_{\rm{left}}$ and $\Gamma_{\rm{right}}$ the four sides of the square $\Omega$.
When $I=3$ we take
\begin{align*}
g_1 &= 1\mbox{ on }\Gamma_{\rm{left}}\cup\Gamma_{\rm{right}} \mbox{ and } g_1 = -1\mbox{ on }\Gamma_{\rm{upper}}\cup\Gamma_{\rm{lower}}, \\ 
g_2 &= 1\mbox{ on }\Gamma_{\rm{left}}\cup\Gamma_{\rm{upper}} \mbox{ and } g_2 = -1\mbox{ on }\Gamma_{\rm{right}}\cup\Gamma_{\rm{lower}},\\
g_3 &= 1\mbox{ on }\Gamma_{\rm{left}}\cup\Gamma_{\rm{lower}} \mbox{ and } g_3 = -1\mbox{ on }\Gamma_{\rm{right}}\cup\Gamma_{\rm{upper}}.
\end{align*}
When $I=7$ we take in addition a smooth approximation of the following piecewise constant function:
\begin{align*}
g_4 &= 1\mbox{ on }\Gamma_{\rm{left}}\cap \{x_2 > 0.5\},\
g_4 = -1\mbox{ on }\Gamma_{\rm{left}}\cap\{x_2 \leq 0.5\} \mbox{ and } g_4=0 \mbox{ otherwise}, 
\end{align*}
and $g_5,g_6,g_7$ are defined in a similar way on $\Gamma_{\rm{right}}$, $\Gamma_{\rm{upper}}$, $\Gamma_{\rm{lower}}$, respectively.

For the numerics we use the cost functional given by \eqref{E:cost_full}:
\begin{align}
\label{eit3.4}  J(\Om) & = \frac{1}{2}\sum_{i=1}^I  \mu_i \sum_{k=1}^K (\sol_i(x_k) - h(x_k))^2,
\end{align}
where $u_i$ is the potential associated with the current $g_i$.
The weights $\mu_i$ associated with the current $g_i$ are chosen as the inverse of $\sum_{k=1}^K (\sol_i(x_k) - h(x_k))^2$ computed at the initial guess. 
In this way, each term in the sum over $I$ is equal to $1$ at the first iteration, and the initial value of $J(\Om)$ is equal to $I/2$.

To get a relatively smooth descent direction we solve the following partial differential equation: find $V\in H^1_0(\dd)^2$ such that
$$ \int_{\dd}  \alpha_1 D\VV : D\xi  + \alpha_2 \VV \cdot\xi  =  - dJ(\Om)(\xi)\mbox{ for all }\xi\in H^1_0(\dd)^2.$$
For the numerical tests, we chose $\alpha_1 = 0.3$ and $\alpha_2 = 0.7$.
To simplify the implementation, we use Dirichlet conditions on $\partial\dd$ instead of the compact support condition $\VV \in \C^{1}_c(\D\setminus \Gamma_h,\R^2)$ (see Section \ref{sec:prel2}).
Considering that $f_\Om\equiv 0$ in $\D$, $\VV = 0$ on $\partial\dd$  and that the points $\{x_k\}_{k=1}^{K}$ belong to $\Gamma$, in view of Theorem \ref{T:shape} we get $\Sb_0^s(\VV) =0$ which leads to the following equation for $\VV$:
\begin{align*}
\int_{\dd}   \alpha_1 D\VV : D\xi  + \alpha_2 \VV \cdot\xi  
& =  - \int_{\dd} -2\sigma_\Om (\nabla \sol \odot \nabla p): D\xi + (\sigma_\Om  \nabla u\cdot\nabla p  ) \mI : D\xi\quad \mbox{ for all }\xi\in H^1_0(\dd)^2.
\end{align*} 
%%%%%%%%%%%%%%
The relative reconstruction  error $E(\Om^r)$ is defined as
$$ E(\Om^r) := \frac{\displaystyle\int_{\dd} |\chi_{\Om^\star} - \chi_{\Om^r}|}{\displaystyle\int_{\dd} \chi_{\Om^\star}},$$
where  $\Om^r$ is the set obtained in the last iteration of the minimization algorithm.
We use $E(\Om^r)$ as a measure of the quality of the reconstructions.

We present three numerical experiments.  
In the first experiment, the ground truth consists of two ellipses and we use $I=3$ currents; see Figure \ref{fig:2ellipses}.
In the second experiment, the ground truth is a concave shape with one connected component and we use $I=3$ currents; see Figure \ref{fig:concave}.
In the third experiment, the ground truth consists of two ellipses and one ball and we use $I=7$ currents; see Figure \ref{fig:3ellipses}.
For each experiment, we study the influence of the point measurements patterns by comparing the reconstructions obtained using three different sets $\Gamma_h = \{x_k\}_{k=1}^{K}$ with $K\in\{16,34,70\}$.
The point measurements patterns and the corresponding reconstructions are presented in Figures \ref{2ellipses_nbpt}, \ref{concave_nbpt} and \ref{3ellipses_nbpt}, for the respective experiments.
We observe,  as expected, that the reconstructions improve as $K$ becomes larger.
However, one obtains reasonable reconstructions in the case of the concave shape with $I=3$ currents and in the case of the two ellipses and ball with $I=7$ currents, even for $K=16$ points and in the presence of noise; see Figures \ref{concave_nbpt} and \ref{3ellipses_nbpt}. 
In the case of two ellipses, the deterioration of the reconstruction for $K=16$ points is much stronger compared to the case $K=70$.
This indicates that the number of current $I=3$ is too low to reconstruct two ellipses with only $K=16$ points. 
We conclude from these results that the amount of applied currents is more critical than the number of point measurements to obtain a good reconstruction.  

For each experiment, we also study how the noise level affects the reconstruction depending on the amount of point measurements.
The results are gathered in Tables \ref{fig:noise_influence_2ellipses}, \ref{fig:noise_influence_concave} and \ref{fig:noise_influence}, where the rows correspond to three different levels of noise, and the columns to three different numbers of points $K\in\{16,34,70\}$.
In the case of two ellipses (Table \ref{fig:noise_influence_2ellipses}), the reconstruction using $K=70$ is very robust with respect to noise, whereas it deteriorates considerably using $K=16$.
In the cases of the concave shape (Table \ref{fig:noise_influence_concave}) and of the two ellipses and ball (Table \ref{fig:noise_influence}), the degradations of the reconstructions when the noise becomes larger are of a similar order in terms of reconstruction error, independently of the value of $K$. 
These results indicate that a larger number of points $K$ may improve the robustness of the reconstruction with respect to noise mainly when the number $I$ of currents is low compared to the complexity of the ground truth.\\

% ================================================================
% 2 elipses, 3 currents
% ================================================================
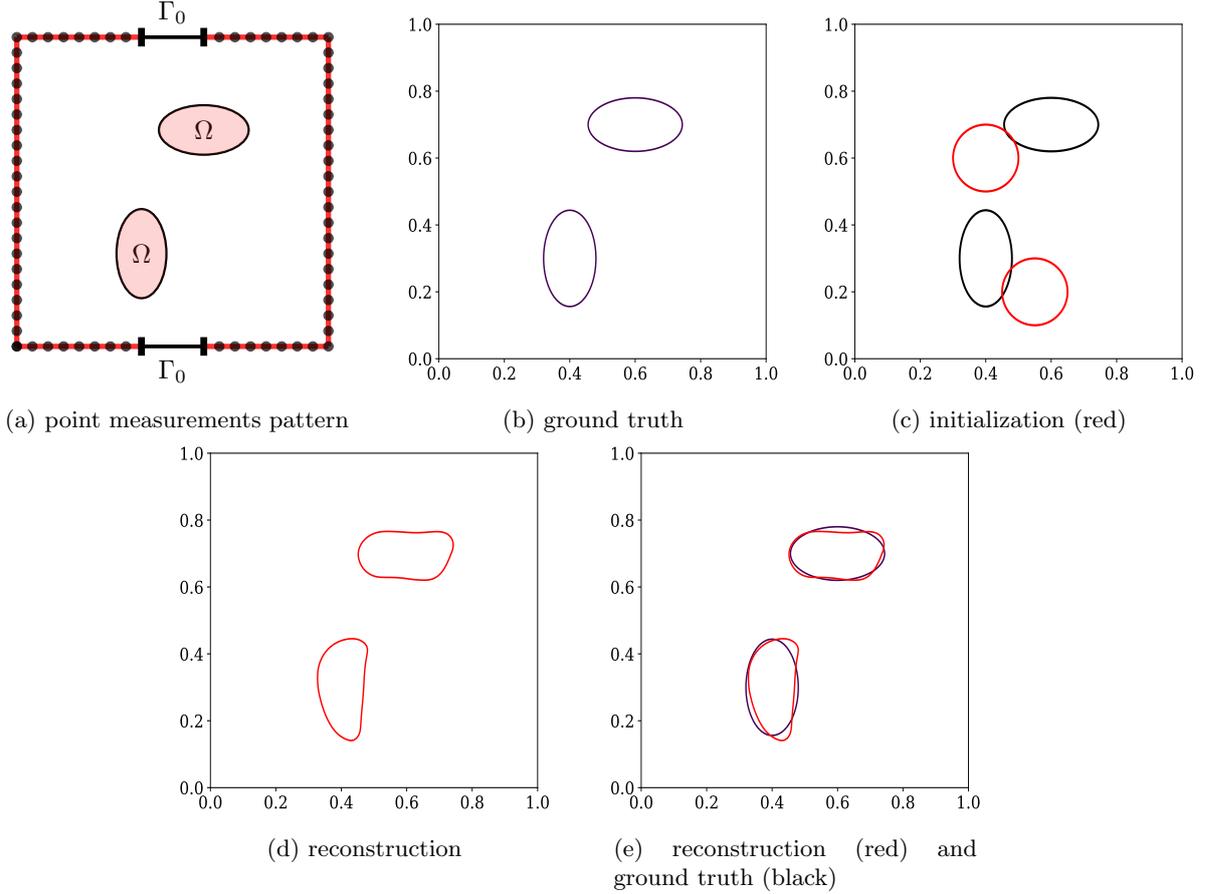
\begin{figure}[ht]
\centering
\begin{subfigure}[t]{0.334\textwidth}
    \centering
\definecolor{preto}{RGB}{0,0,0}
\definecolor{vermelho}{RGB}{254,49,49}
\definecolor{azul_escuro}{RGB}{0,0,255}
\definecolor{cinza}{RGB}{195,195,195}
\definecolor{papel_amarelo}{RGB}{255, 249, 240}

\begin{tikzpicture}[scale=4.1]
	% background square color	
	%\fill[papel_amarelo, opacity=0.4] (0,0) rectangle (1,1);
	%%%%%%%%%%%%%%%%
\draw[line width=0.5mm, preto](0.4,1) -- (0.6,1);
%%%%%%%%%%%%%%%%
\draw[line width=0.5mm, preto](0.4,0) -- (0.6,0);
	% grid lines
	%\draw[step=0.2cm,gray,very thin, opacity=0.5] (1,0) grid (0,1);
	% border lines left
	\draw[line width=0.7mm, vermelho](0,1) -- (0.4,1);
	\draw[line width=0.7mm, vermelho](0,0) -- (0,1);	
	\draw[line width=0.7mm, vermelho] (0,0) -- (0.4,0);
	% border lines right
    \draw[line width=0.7mm, vermelho](0.6,1) -- (1,1);
	\draw[line width=0.7mm, vermelho](1,0) -- (1,1);
	\draw[line width=0.7mm, vermelho] (0.6,0) -- (1,0);
	% tic top - right
	\draw[line width=0.9mm, preto](0.6,0.97) -- (0.6,1.03);
	% tic top - left
	\draw[line width=0.9mm, preto](0.4,0.97) -- (0.4,1.03);
	% tic bottom - right
	\draw[line width=0.9mm, preto](0.6,-0.03) -- (0.6,0.03);
	% tic bottom - left
	\draw[line width=0.9mm, preto](0.4,-0.03) -- (0.4,0.03);
	% ellipses : lengths in the brackets separated by an and, are the x-direction radius and the y-direction radius respectively
	\draw[line width=0.3mm, preto] (0.6,0.7) ellipse (0.144cm and 0.08cm) node[anchor= center] {$\Om$};
\fill[vermelho, opacity=0.2] (0.6,0.7) ellipse (0.144cm and 0.08cm);
% below below
\draw[line width=0.3mm, preto] (0.4,0.3) ellipse (0.08cm and 0.144cm) node[anchor= center] {$\Om$};
\fill[vermelho, opacity=0.2] (0.4,0.3) ellipse (0.08cm and 0.144cm);
	%defining dots size
	%defining dots size
	\def\raio{0.015}
	% defining opacity level 
	\def\transparenc{0.7}	
	% defining step size	
	\def\h{0.05};
	% draw border dots left
	\foreach \x in {0,\h,...,1.0}{
		 \draw[fill= preto, opacity=\transparenc] (0,\x) circle (\raio cm);
}
	% draw border dots right
	\foreach \x in {0,\h, ...,1.0}{
		 \draw[fill= preto, opacity=\transparenc] (1,\x) circle (\raio cm);
}
	% draw border dots top-left
	\foreach \x in {0.0,\h , ..., 0.4}{
		 \draw[fill= preto, opacity=\transparenc] (\x,1) circle (\raio cm);
}
	% draw border dots top-righ
	\foreach \x in {0.65, 0.7, ...,1.05}{
		 \draw[fill= preto, opacity=\transparenc] (\x,1) circle (\raio cm);
}
	% draw border dots bottom-left
	\foreach \x in {0.0, 0.05, ...,0.4}{
		 \draw[fill= preto, opacity=\transparenc ] (\x,0) circle (\raio cm);
}
	% draw border dots bottom-righ
	\foreach \x in {0.65,0.7, ..., 1}{
		 \draw[fill= preto, opacity=\transparenc] (\x,0) circle (\raio cm);
}
	% draw Gamma
	\draw (0.5, -0.08) node[anchor= center] {$\Gamma_0$};
	\draw (0.5, 1.08) node[anchor= center] {$\Gamma_0$};
\end{tikzpicture}
    \caption{point measurements pattern}
\end{subfigure}%
~    
    \begin{subfigure}[t]{0.334\textwidth} 
        \centering
        \includegraphics[height=1.99in,width=1.99in]{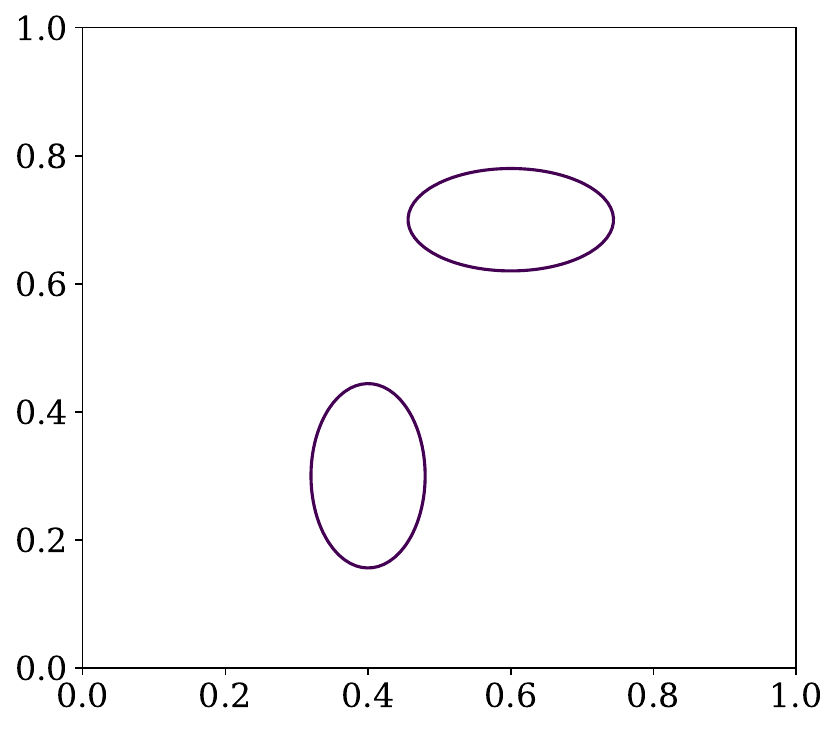}
        \caption{ground truth}
    \end{subfigure}%
~
    \begin{subfigure}[t]{0.334\textwidth} 
        \centering
        \includegraphics[height=1.99in,width=1.99in]{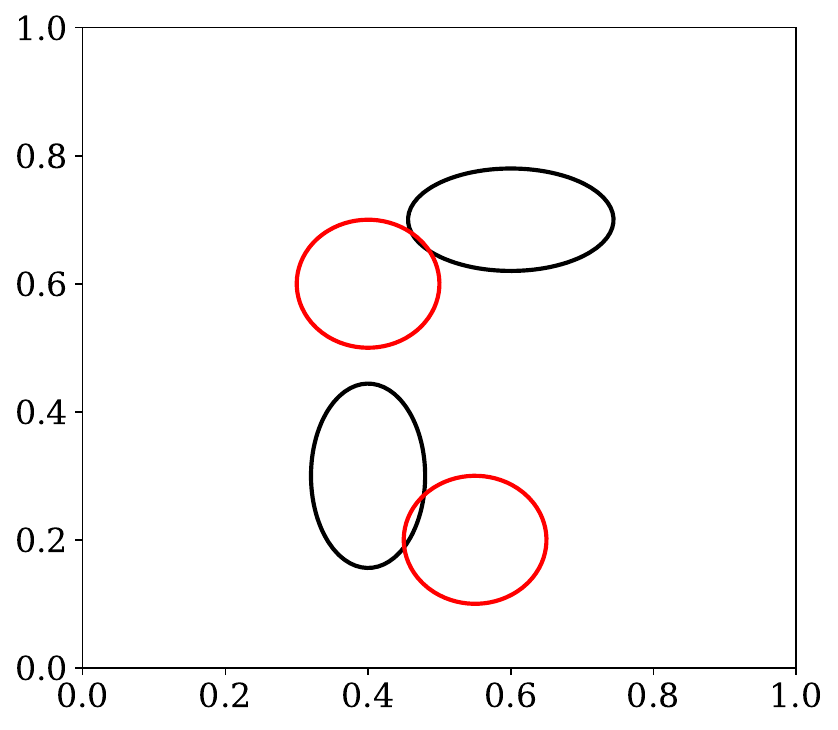}
        \caption{initialization (red)}
    \end{subfigure}%
  
    \begin{subfigure}[t]{0.334\textwidth} 
        \centering
        \includegraphics[height=1.99in,width=1.99in]{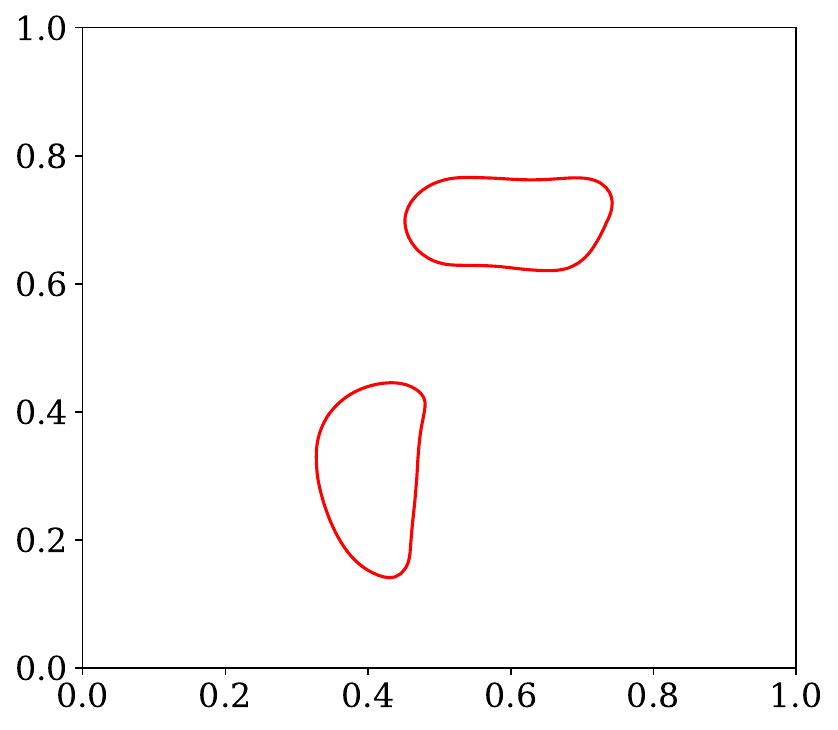}
        \caption{reconstruction}
    \end{subfigure}
    ~
    \begin{subfigure}[t]{0.334\textwidth} 
        \centering
        \includegraphics[height=1.99in,width=1.99in]{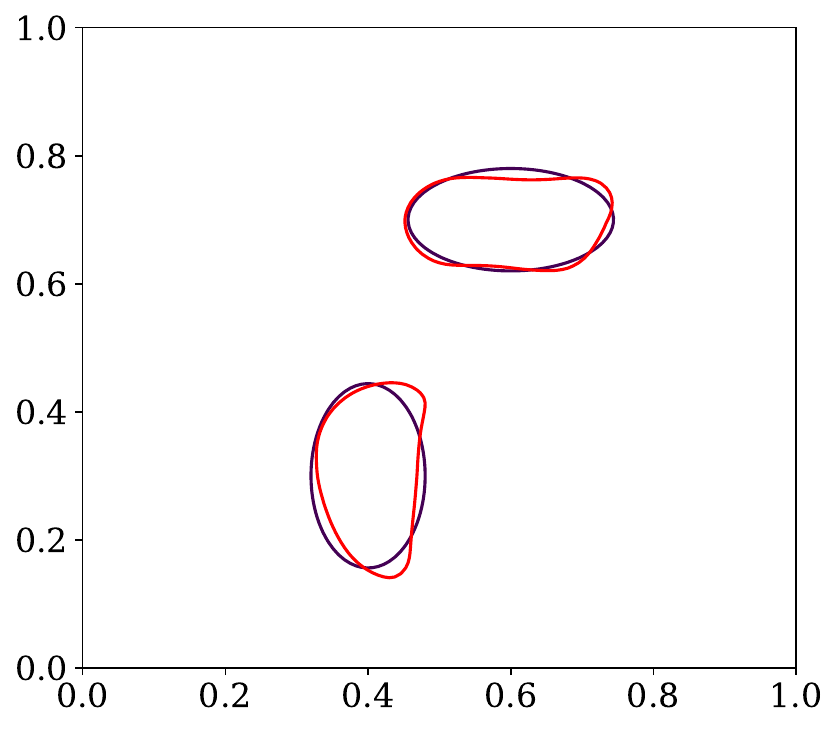}
        \caption{reconstruction (red) and ground truth (black)}
    \end{subfigure}
    \caption{Reconstruction of two ellipses using $I=3$ currents  and $K=70$ point measurements with $1.13\%$ noise.}\label{fig:2ellipses}
\end{figure}

% ================================================================
% 2 elipses, 3 currents, compare reconstructions when increasing number of points
% ================================================================
\begin{figure}[ht]
    \centering
    \begin{subfigure}[t]{0.334\textwidth}
        \centering
\definecolor{preto}{RGB}{0,0,0}
\definecolor{vermelho}{RGB}{254,49,49}
\definecolor{azul_escuro}{RGB}{0,0,255}
\definecolor{cinza}{RGB}{195,195,195}
\definecolor{papel_amarelo}{RGB}{255, 249, 240}

\begin{tikzpicture}[scale=4.1]
	% background square color	
	%\fill[papel_amarelo, opacity=0.4] (0,0) rectangle (1,1);
	%%%%%%%%%%%%%%%%
\draw[line width=0.5mm, preto](0.4,1) -- (0.6,1);
%%%%%%%%%%%%%%%%
\draw[line width=0.5mm, preto](0.4,0) -- (0.6,0);
	% grid lines
	%\draw[step=0.2cm,gray,very thin, opacity=0.5] (1,0) grid (0,1);
	% border lines left
	\draw[line width=0.7mm, vermelho](0,1) -- (0.4,1);
	\draw[line width=0.7mm, vermelho](0,0) -- (0,1);
	\draw[line width=0.7mm, vermelho] (0,0) -- (0.4,0);
	% border lines right
    \draw[line width=0.7mm, vermelho](0.6,1) -- (1,1);
	\draw[line width=0.7mm, vermelho](1,0) -- (1,1);
	\draw[line width=0.7mm, vermelho] (0.6,0) -- (1,0);
	% tic top - right
	\draw[line width=0.9mm, preto](0.6,0.97) -- (0.6,1.03);
	% tic top - left
	\draw[line width=0.9mm, preto](0.4,0.97) -- (0.4,1.03);
	% tic bottom - right
	\draw[line width=0.9mm, preto](0.6,-0.03) -- (0.6,0.03);
	% tic bottom - left
	\draw[line width=0.9mm, preto](0.4,-0.03) -- (0.4,0.03);
	% ellipses : lengths in the brackets separated by an and, are the x-direction radius and the y-direction radius respectively
	% ellipse above
	\draw[line width=0.3mm, preto] (0.6,0.7) ellipse (0.144cm and 0.08cm) node[anchor= center] {$\Om$};
\fill[vermelho, opacity=0.2] (0.6,0.7) ellipse (0.144cm and 0.08cm);
% below below
\draw[line width=0.3mm, preto] (0.4,0.3) ellipse (0.08cm and 0.144cm) node[anchor= center] {$\Om$};
\fill[vermelho, opacity=0.2] (0.4,0.3) ellipse (0.08cm and 0.144cm);
	%defining dots size
	\def\raio{0.015}
	% defining opacity level 
	\def\transparenc{0.7}	
	% defining step size	
	\def\h{0.2};

	% draw border dots left
	\foreach \x in {0,\h,...,1.0}{
		 \draw[fill= preto, opacity=\transparenc] (0,\x) circle (\raio cm);
}
	% draw border dots right
	\foreach \x in {0,\h,...,1.0}{
		 \draw[fill= preto, opacity=\transparenc] (1,\x) circle (\raio cm);
}
	% draw border dots top-left
	\foreach \x in {\h}{
		 \draw[fill= preto, opacity=\transparenc] (\x,1) circle (\raio cm);
}
	% draw border dots top-righ
	\foreach \x in {4*\h}{
		 \draw[fill= preto, opacity=\transparenc] (\x,1) circle (\raio cm);
}
	% draw border dots bottom-left
	\foreach \x in {\h}{
		 \draw[fill= preto, opacity=\transparenc ] (\x,0) circle (\raio cm);
}
	% draw border dots bottom-righ
	\foreach \x in {4*\h}{
		 \draw[fill= preto, opacity=\transparenc] (\x,0) circle (\raio cm);
}
	% draw Gamma
	\draw (0.5, -0.08) node[anchor= center] {$\Gamma_0$};
	\draw (0.5, 1.08) node[anchor= center] {$\Gamma_0$};

\end{tikzpicture}
    \end{subfigure}%
    ~ 
\begin{subfigure}[t]{0.334\textwidth}
        \centering
\definecolor{preto}{RGB}{0,0,0}
\definecolor{vermelho}{RGB}{254,49,49}
\definecolor{azul_escuro}{RGB}{0,0,255}
\definecolor{cinza}{RGB}{195,195,195}
\definecolor{papel_amarelo}{RGB}{255, 249, 240}

\begin{tikzpicture}[scale=4.1]
	% background square color	
	%\fill[papel_amarelo, opacity=0.4] (0,0) rectangle (1,1);
	%%%%%%%%%%%%%%%%
\draw[line width=0.5mm, preto](0.4,1) -- (0.6,1);
%%%%%%%%%%%%%%%%
\draw[line width=0.5mm, preto](0.4,0) -- (0.6,0);
	% grid lines
	%\draw[step=0.2cm,gray,very thin, opacity=0.5] (1,0) grid (0,1);
	% border lines left
	\draw[line width=0.7mm, vermelho](0,1) -- (0.4,1);
	\draw[line width=0.7mm, vermelho](0,0) -- (0,1);	
	\draw[line width=0.7mm, vermelho] (0,0) -- (0.4,0);
	% border lines right
    \draw[line width=0.7mm, vermelho](0.6,1) -- (1,1);
	\draw[line width=0.7mm, vermelho](1,0) -- (1,1);
	\draw[line width=0.7mm, vermelho] (0.6,0) -- (1,0);
	% tic top - right
	\draw[line width=0.9mm, preto](0.6,0.97) -- (0.6,1.03);
	% tic top - left
	\draw[line width=0.9mm, preto](0.4,0.97) -- (0.4,1.03);
	% tic bottom - right
	\draw[line width=0.9mm, preto](0.6,-0.03) -- (0.6,0.03);
	% tic bottom - left
	\draw[line width=0.9mm, preto](0.4,-0.03) -- (0.4,0.03);
	% ellipses : lengths in the brackets separated by an and, are the x-direction radius and the y-direction radius respectively
	\draw[line width=0.3mm, preto] (0.6,0.7) ellipse (0.144cm and 0.08cm) node[anchor= center] {$\Om$};
\fill[vermelho, opacity=0.2] (0.6,0.7) ellipse (0.144cm and 0.08cm);
% below below
\draw[line width=0.3mm, preto] (0.4,0.3) ellipse (0.08cm and 0.144cm) node[anchor= center] {$\Om$};
\fill[vermelho, opacity=0.2] (0.4,0.3) ellipse (0.08cm and 0.144cm);
	%defining dots size
	%defining dots size
	\def\raio{0.015}
	% defining opacity level 
	\def\transparenc{0.7}	
	% defining step size	
	\def\h{0.1};
	% draw border dots left
	\foreach \x in {0,\h,...,1.1}{
		 \draw[fill= preto, opacity=\transparenc] (0,\x) circle (\raio cm);
}
	% draw border dots right
	\foreach \x in {0,\h, ...,1.1}{
		 \draw[fill= preto, opacity=\transparenc] (1,\x) circle (\raio cm);
}
	% draw border dots top-left
	\foreach \x in {0.0,\h , ..., 0.4}{
		 \draw[fill= preto, opacity=\transparenc] (\x,1) circle (\raio cm);
}
	% draw border dots top-righ
	\foreach \x in {0.7, 0.8, 0.9}{
		 \draw[fill= preto, opacity=\transparenc] (\x,1) circle (\raio cm);
}
	% draw border dots bottom-left
	\foreach \x in {0.1, 0.2,0.3}{
		 \draw[fill= preto, opacity=\transparenc ] (\x,0) circle (\raio cm);
}
	% draw border dots bottom-righ
	\foreach \x in {0.7,0.8, 0.9}{
		 \draw[fill= preto, opacity=\transparenc] (\x,0) circle (\raio cm);
}
	% draw Gamma
	\draw (0.5, -0.08) node[anchor= center] {$\Gamma_0$};
	\draw (0.5, 1.08) node[anchor= center] {$\Gamma_0$};
\end{tikzpicture}
    \end{subfigure}%
~
\begin{subfigure}[t]{0.334\textwidth}
        \centering
\definecolor{preto}{RGB}{0,0,0}
\definecolor{vermelho}{RGB}{254,49,49}
\definecolor{azul_escuro}{RGB}{0,0,255}
\definecolor{cinza}{RGB}{195,195,195}
\definecolor{papel_amarelo}{RGB}{255, 249, 240}

\begin{tikzpicture}[scale=4.1]
	% background square color	
	%\fill[papel_amarelo, opacity=0.4] (0,0) rectangle (1,1);
	%%%%%%%%%%%%%%%%
\draw[line width=0.5mm, preto](0.4,1) -- (0.6,1);
%%%%%%%%%%%%%%%%
\draw[line width=0.5mm, preto](0.4,0) -- (0.6,0);
	% grid lines
	%\draw[step=0.2cm,gray,very thin, opacity=0.5] (1,0) grid (0,1);
	% border lines left
	\draw[line width=0.7mm, vermelho](0,1) -- (0.4,1);
	\draw[line width=0.7mm, vermelho](0,0) -- (0,1);	
	\draw[line width=0.7mm, vermelho] (0,0) -- (0.4,0);
	% border lines right
    \draw[line width=0.7mm, vermelho](0.6,1) -- (1,1);
	\draw[line width=0.7mm, vermelho](1,0) -- (1,1);
	\draw[line width=0.7mm, vermelho] (0.6,0) -- (1,0);
	% tic top - right
	\draw[line width=0.9mm, preto](0.6,0.97) -- (0.6,1.03);
	% tic top - left
	\draw[line width=0.9mm, preto](0.4,0.97) -- (0.4,1.03);
	% tic bottom - right
	\draw[line width=0.9mm, preto](0.6,-0.03) -- (0.6,0.03);
	% tic bottom - left
	\draw[line width=0.9mm, preto](0.4,-0.03) -- (0.4,0.03);
	% ellipses : lengths in the brackets separated by an and, are the x-direction radius and the y-direction radius respectively
	\draw[line width=0.3mm, preto] (0.6,0.7) ellipse (0.144cm and 0.08cm) node[anchor= center] {$\Om$};
\fill[vermelho, opacity=0.2] (0.6,0.7) ellipse (0.144cm and 0.08cm);
% below below
\draw[line width=0.3mm, preto] (0.4,0.3) ellipse (0.08cm and 0.144cm) node[anchor= center] {$\Om$};
\fill[vermelho, opacity=0.2] (0.4,0.3) ellipse (0.08cm and 0.144cm);
	%defining dots size
	%defining dots size
	\def\raio{0.015}
	% defining opacity level 
	\def\transparenc{0.7}	
	% defining step size	
	\def\h{0.05};
	% draw border dots left
	\foreach \x in {0,\h,...,1.0}{
		 \draw[fill= preto, opacity=\transparenc] (0,\x) circle (\raio cm);
}
	% draw border dots right
	\foreach \x in {0,\h, ...,1.0}{
		 \draw[fill= preto, opacity=\transparenc] (1,\x) circle (\raio cm);
}
	% draw border dots top-left
	\foreach \x in {0.0,\h , ..., 0.4}{
		 \draw[fill= preto, opacity=\transparenc] (\x,1) circle (\raio cm);
}
	% draw border dots top-righ
	\foreach \x in {0.65, 0.7, ...,1.05}{
		 \draw[fill= preto, opacity=\transparenc] (\x,1) circle (\raio cm);
}
	% draw border dots bottom-left
	\foreach \x in {0.0, 0.05, ...,0.4}{
		 \draw[fill= preto, opacity=\transparenc ] (\x,0) circle (\raio cm);
}
	% draw border dots bottom-righ
	\foreach \x in {0.65,0.7, ..., 1}{
		 \draw[fill= preto, opacity=\transparenc] (\x,0) circle (\raio cm);
}
	% draw Gamma
	\draw (0.5, -0.08) node[anchor= center] {$\Gamma_0$};
	\draw (0.5, 1.08) node[anchor= center] {$\Gamma_0$};
\end{tikzpicture}
    \end{subfigure}
    
    \begin{subfigure}[t]{0.334\textwidth} 
        \centering
        \includegraphics[height=1.99in,width=1.99in]{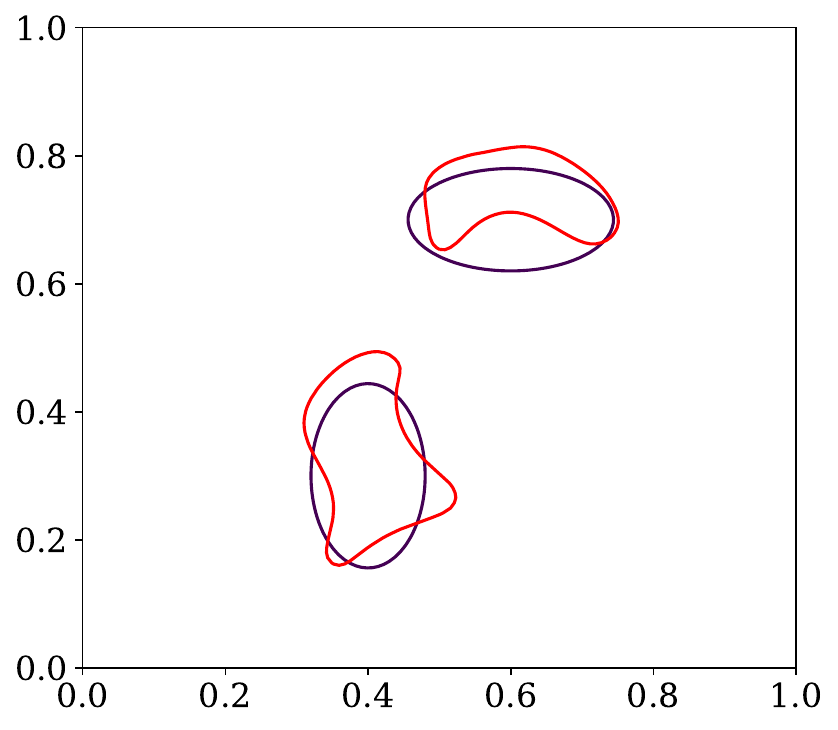}
        \caption{$K=16$ point measurements,\\ $1.18 \%$ noise, $54.00\%$ relative error}
    \end{subfigure}%
~
    \begin{subfigure}[t]{0.334\textwidth} 
        \centering
        \includegraphics[height=1.99in,width=1.99in]{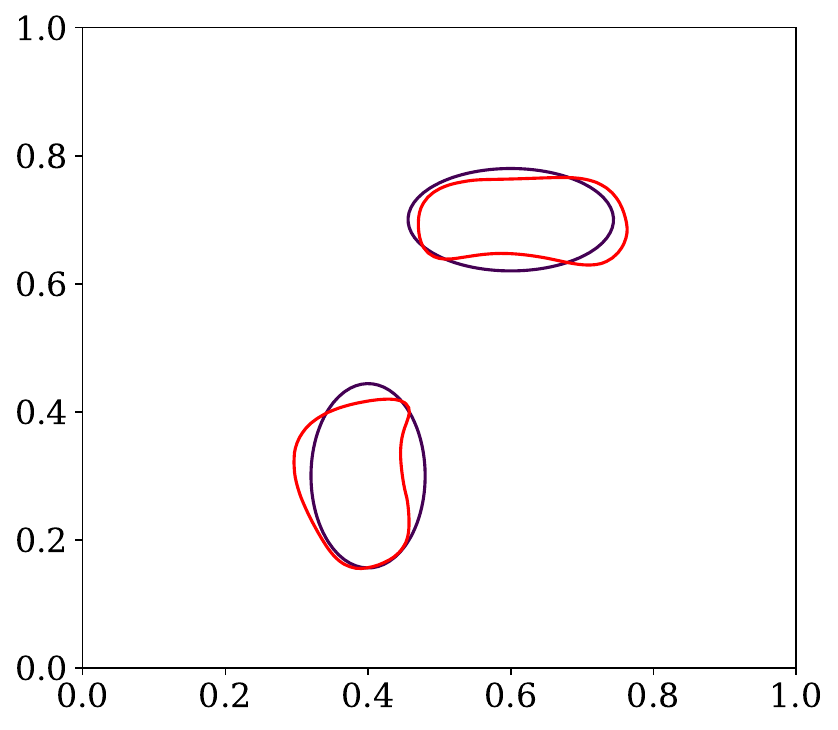}
        \caption{$K=34$ point measurements,\\ $1.18 \%$ noise, $27.13\%$ relative error}
    \end{subfigure}%
    ~
    \begin{subfigure}[t]{0.334\textwidth} 
        \centering
        \includegraphics[height=1.99in,width=1.99in]{./tests1/1-1-2-2/Difference.pdf}
        \caption{$K=70$ point measurements,\\ $1.13 \%$ noise, $17.19\%$ relative error}
    \end{subfigure}
    \caption{Reconstruction of two ellipses using $I=3$ currents  and three different sets of point measurements shown in the first row.}\label{2ellipses_nbpt}
\end{figure}

% ================================================================
% influence of noise on the reconstruction of two ellipses
% ================================================================
\begin{table}[ht]
  \centering
  \begin{tabular}{cccc}
    \hline
     noise & $K=16$ points &  $K=34$ points &  $K=70$ points \\ \hline
     $0\%$ 
     &
    \begin{minipage}{0.25\textwidth}
            \centering{\vspace{0.1cm}{\scriptsize error: $33.4\%$}}\\
      \includegraphics[width=1\textwidth, height = 1\textwidth]{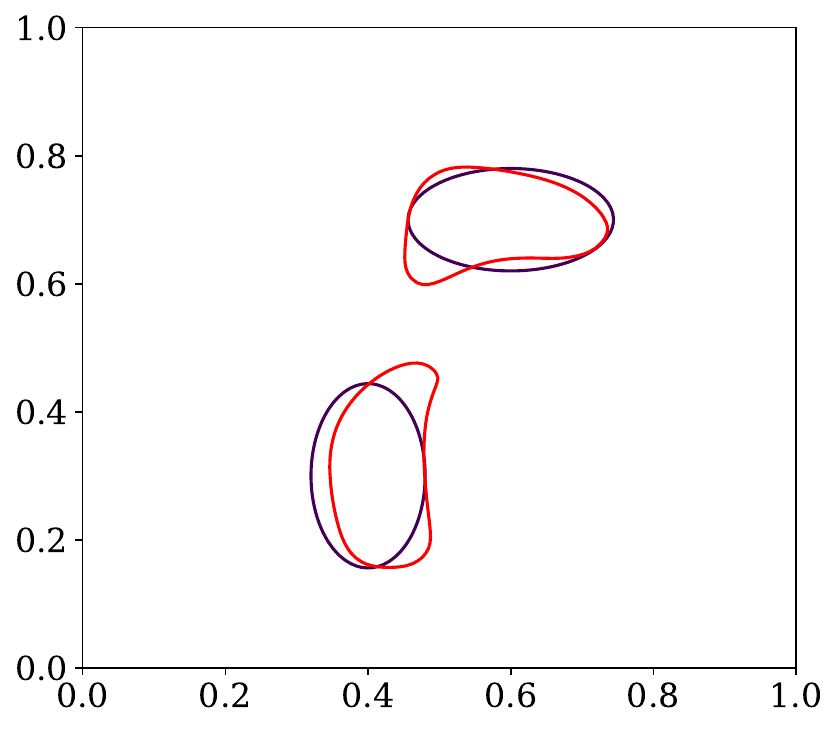}
      \end{minipage}
    &
    \begin{minipage}{0.25\textwidth}
                \centering{\vspace{0.1cm}{\scriptsize error: $18.8\%$}}\\
      \includegraphics[width=1\textwidth, height = 1\textwidth]{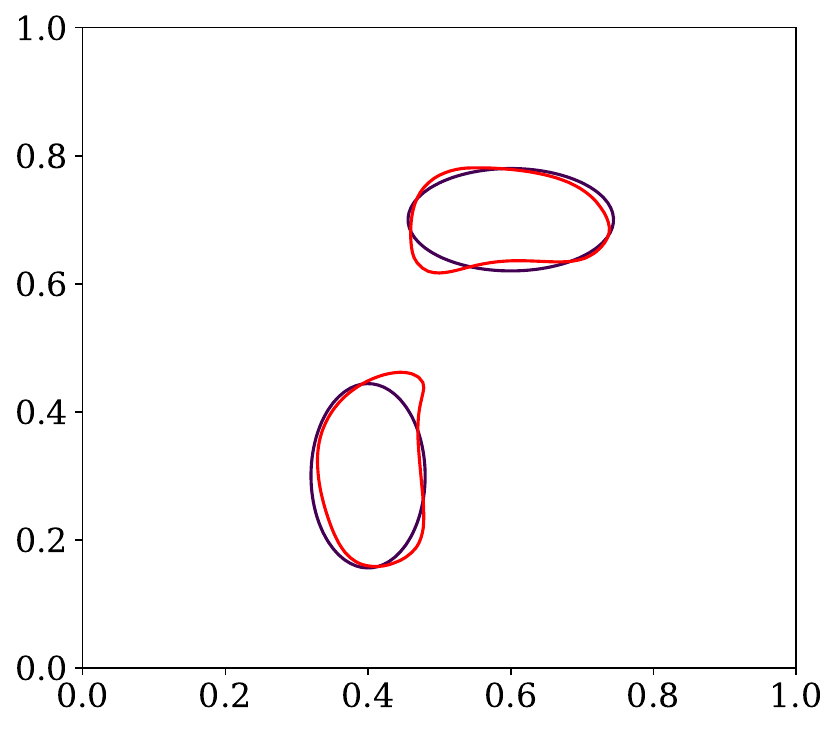}
    \end{minipage}
    & 
    \begin{minipage}{0.25\textwidth}
                \centering{\vspace{0.1cm}{\scriptsize error: $16.3\%$}}\\
      \includegraphics[width=1\textwidth, height = 1\textwidth]{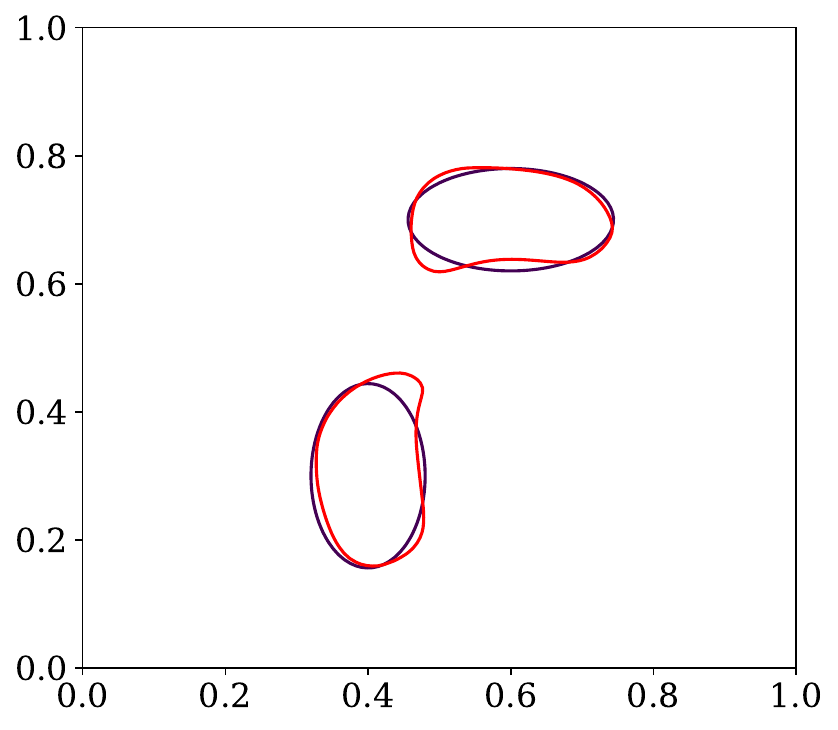}
    \end{minipage}
    \\ \hline   
    %%%%%%%%%%%%%%%%%
    $0.51\%$
    &
    \begin{minipage}{0.25\textwidth}
                \centering{\vspace{0.1cm}{\scriptsize error: $29.7\%$}}\\
      \includegraphics[width=1\textwidth, height = 1\textwidth]{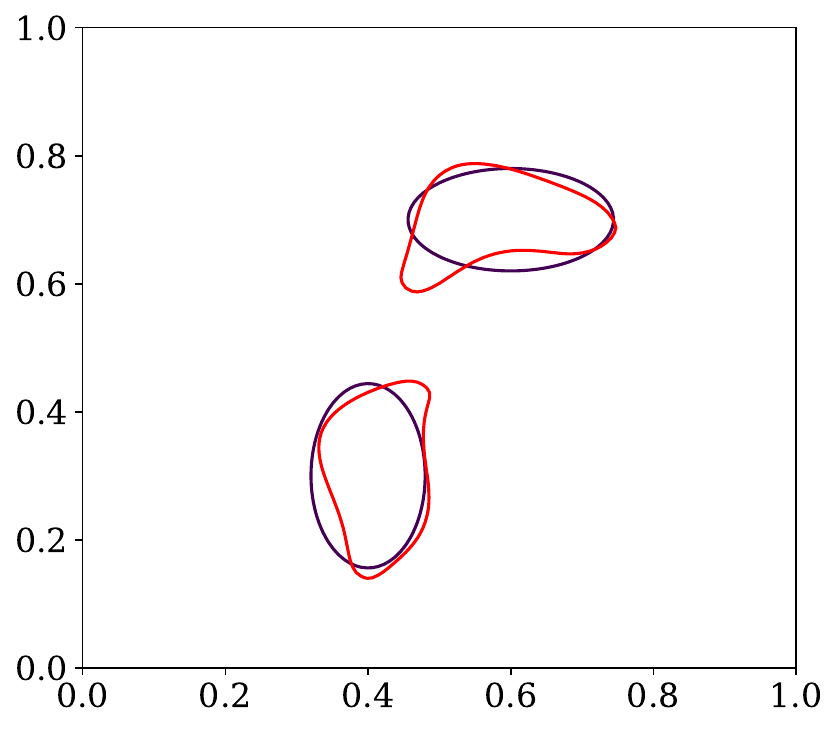}
    \end{minipage}
    &
    \begin{minipage}{0.25\textwidth}
                \centering{\vspace{0.1cm}{\scriptsize error: $17.5\%$}}\\
      \includegraphics[width=1\textwidth, height = 1\textwidth]{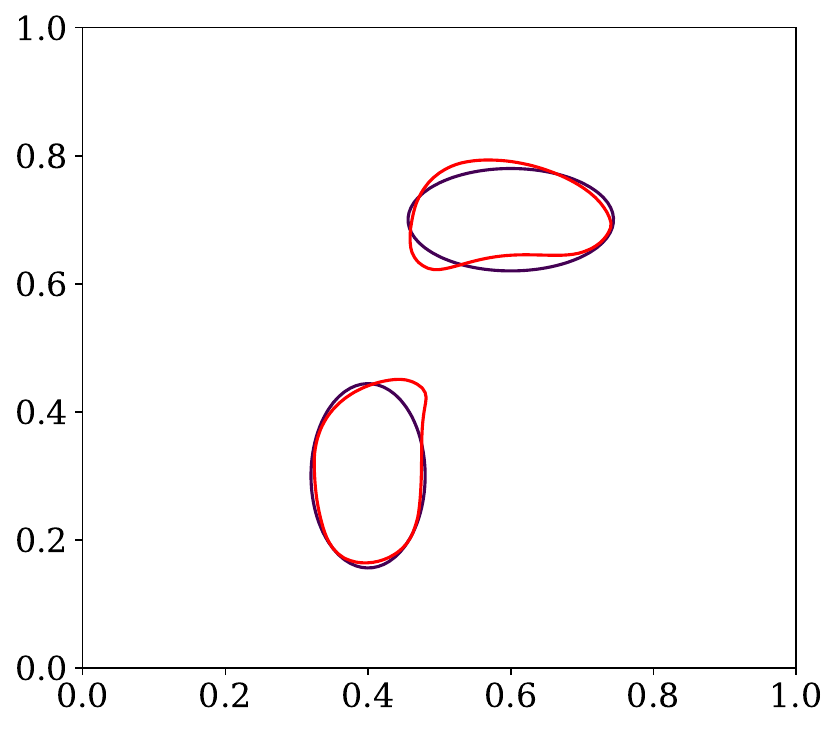}
    \end{minipage}
    & 
    \begin{minipage}{0.25\textwidth}
                \centering{\vspace{0.1cm}{\scriptsize error: $19.8\%$}}\\
      \includegraphics[width=1\textwidth, height = 1\textwidth]{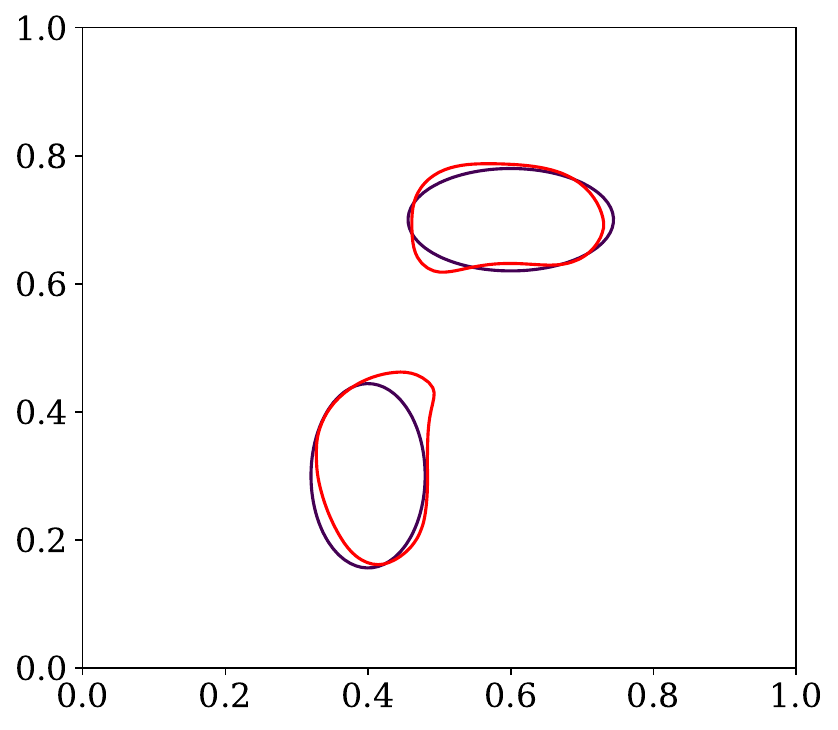}
    \end{minipage}
    \\ \hline  
    %%%%%%%%%%%%%%%%%%
    $1.16\%$
    &
    \begin{minipage}{0.25\textwidth}
                \centering{\vspace{0.1cm}{\scriptsize error: $54.0\%$}}\\
      \includegraphics[width=1\textwidth, height = 1\textwidth]{./tests1/1-1-0-2/Difference.pdf}
    \end{minipage}
    &
    \begin{minipage}{0.25\textwidth}
                \centering{\vspace{0.1cm}{\scriptsize error: $27.1\%$}}\\
      \includegraphics[width=1\textwidth, height = 1\textwidth]{./tests1/1-1-1-2/Difference.pdf}
    \end{minipage}
    & 
    \begin{minipage}{0.25\textwidth}
                \centering{\vspace{0.1cm}{\scriptsize error: $17.2\%$}}\\
      \includegraphics[width=1\textwidth, height = 1\textwidth]{./tests1/1-1-2-2/Difference.pdf}
    \end{minipage}
  \end{tabular}
  \caption{Influence of noise and number of point measurements on the reconstruction of two ellipses using $I=3$ currents (the noise value is the average over the noise values for the three levels of point measurements).}\label{fig:noise_influence_2ellipses}
\end{table}

% ================================================================
% Concave shape, 3 currents
% ================================================================
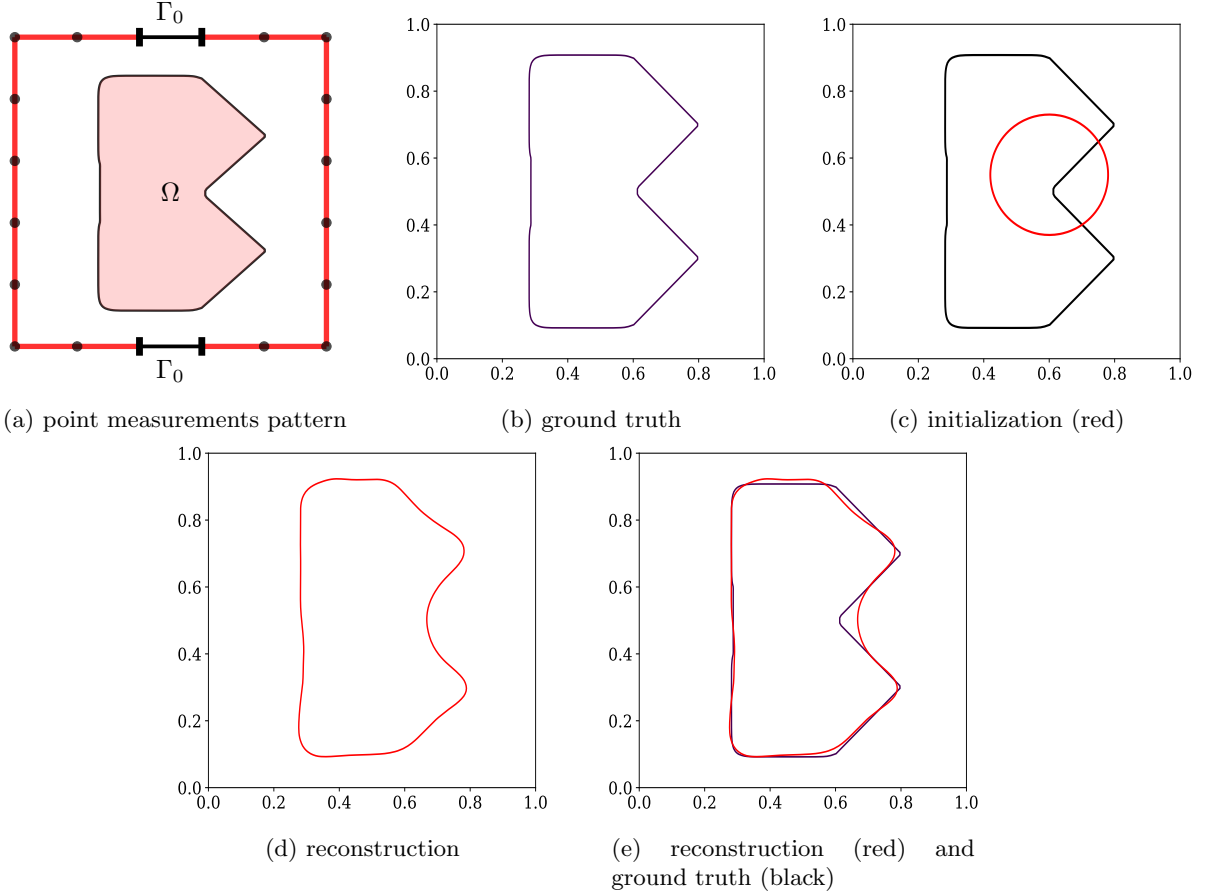
\begin{figure}[ht]
\centering
\begin{subfigure}[t]{0.334\textwidth}
    \centering
% \definecolor{preto}{RGB}{0,0,0}
% \definecolor{vermelho}{RGB}{254,49,49}
% \definecolor{azul_escuro}{RGB}{0,0,255}
% \definecolor{cinza}{RGB}{195,195,195}
% \definecolor{papel_amarelo}{RGB}{255, 249, 240}

\begin{tikzpicture}[scale=4.1]
	% background square color	
	%\fill[papel_amarelo, opacity=0.4] (0,0) rectangle (1,1);
	% grid lines
	%\draw[step=0.2cm,gray,very thin, opacity=0.5] (1,0) grid (0,1);
	\draw[line width=0.5mm, preto](0.4,1) -- (0.6,1);
        \draw[line width=0.5mm, preto](0.4,0) -- (0.6,0);
	% border lines left
	\draw[line width=0.7mm, vermelho](0,1) -- (0.4,1);
	\draw[line width=0.7mm, vermelho](0,0) -- (0,1);	
	\draw[line width=0.7mm, vermelho] (0,0) -- (0.4,0);
	% border lines right
    \draw[line width=0.7mm, vermelho](0.6,1) -- (1,1);
	\draw[line width=0.7mm, vermelho](1,0) -- (1,1);
	\draw[line width=0.7mm, vermelho] (0.6,0) -- (1,0);
	% tic top - right
	\draw[line width=0.9mm, preto](0.6,0.97) -- (0.6,1.03);
	% tic top - left
	\draw[line width=0.9mm, preto](0.4,0.97) -- (0.4,1.03);
	% tic bottom - right
	\draw[line width=0.9mm, preto](0.6,-0.03) -- (0.6,0.03);
	% tic bottom - left
	\draw[line width=0.9mm, preto](0.4,-0.03) -- (0.4,0.03);
	%defining dots size
	\def\raio{0.015}
	% defining opacity level 
	\def\transparenc{0.7}	
	% defining step size	
	\def\h{0.2};

	% draw border dots left
	\foreach \x in {0,\h,...,1.0}{
		 \draw[fill= preto, opacity=\transparenc] (0,\x) circle (\raio cm);
}
	% draw border dots right
	\foreach \x in {0,\h,...,1.0}{
		 \draw[fill= preto, opacity=\transparenc] (1,\x) circle (\raio cm);
}
	% draw border dots top-left
	\foreach \x in {\h}{
		 \draw[fill= preto, opacity=\transparenc] (\x,1) circle (\raio cm);
}
	% draw border dots top-righ
	\foreach \x in {4*\h}{
		 \draw[fill= preto, opacity=\transparenc] (\x,1) circle (\raio cm);
}
	% draw border dots bottom-left
	\foreach \x in {\h}{
		 \draw[fill= preto, opacity=\transparenc ] (\x,0) circle (\raio cm);
}
	% draw border dots bottom-righ
	\foreach \x in {4*\h}{
		 \draw[fill= preto, opacity=\transparenc] (\x,0) circle (\raio cm);
}
	% draw Gamma
	\draw (0.5, -0.08) node[anchor= center] {$\Gamma_0$};
	\draw (0.5, 1.08) node[anchor= center] {$\Gamma_0$};

\begin{scope}[shift={(-0.145,-0.06)} ,scale=0.303]
      \path[draw=black, line join=round, line cap=butt, miter limit=10.00, line width=0.3mm, opacity=0.85] 
        (1.520565,0.584200) -- (1.493806,0.587955) -- (1.467047,0.594101)
        -- (1.440288,0.603245) -- (1.428873,0.608925) -- (1.413529,0.619704) --
        (1.400350,0.632931) -- (1.384588,0.656938) -- (1.376503,0.680944) --
        (1.371107,0.704950) -- (1.367595,0.728956) -- (1.364009,0.776969) --
        (1.362501,0.848988) -- (1.362278,1.089050) -- (1.362915,1.377125) --
        (1.364525,1.425138) -- (1.368951,1.473150) -- (1.373207,1.497156) --
        (1.379672,1.521163) -- (1.381099,1.545169) -- (1.380360,1.641194) --
        (1.380713,2.097313) -- (1.381099,2.121319) -- (1.379672,2.145325) --
        (1.373207,2.169331) -- (1.368951,2.193338) -- (1.364525,2.241350) --
        (1.362915,2.289363) -- (1.362285,2.409394) -- (1.363203,2.865513) --
        (1.365372,2.913525) -- (1.367595,2.937531) -- (1.371107,2.961538) --
        (1.376503,2.985544) -- (1.386770,3.013969) -- (1.400350,3.033556) --
        (1.413529,3.046783) -- (1.428873,3.057563) -- (1.440288,3.063243) --
        (1.467047,3.072386) -- (1.493806,3.078533) -- (1.520565,3.082287) --
        (1.574083,3.085293) -- (1.654360,3.086593) -- (1.948708,3.086798) --
        (2.269816,3.086079) -- (2.323334,3.084420) -- (2.361765,3.081569) --
        (2.376852,3.079484) -- (2.403611,3.073823) -- (2.457129,3.057562) --
        (3.126104,2.457406) -- (3.125238,2.433400) -- (2.671184,2.025294) --
        (2.510646,1.881256) -- (2.495429,1.857250) -- (2.495429,1.809237) --
        (2.510647,1.785231) -- (3.125238,1.233087) -- (3.126104,1.209081) --
        (2.457129,0.608925) -- (2.403611,0.592664) -- (2.361765,0.584919) --
        (2.323334,0.582067) -- (2.243057,0.580057) -- (2.055744,0.579690) --
        (1.600842,0.580507) -- (1.547324,0.582345) -- (1.520565,0.584200) --
        (1.520565,0.584200) -- cycle;
        
              \fill[vermelho, opacity=0.2]  
        (1.520565,0.584200) -- (1.493806,0.587955) -- (1.467047,0.594101)
        -- (1.440288,0.603245) -- (1.428873,0.608925) -- (1.413529,0.619704) --
        (1.400350,0.632931) -- (1.384588,0.656938) -- (1.376503,0.680944) --
        (1.371107,0.704950) -- (1.367595,0.728956) -- (1.364009,0.776969) --
        (1.362501,0.848988) -- (1.362278,1.089050) -- (1.362915,1.377125) --
        (1.364525,1.425138) -- (1.368951,1.473150) -- (1.373207,1.497156) --
        (1.379672,1.521163) -- (1.381099,1.545169) -- (1.380360,1.641194) --
        (1.380713,2.097313) -- (1.381099,2.121319) -- (1.379672,2.145325) --
        (1.373207,2.169331) -- (1.368951,2.193338) -- (1.364525,2.241350) --
        (1.362915,2.289363) -- (1.362285,2.409394) -- (1.363203,2.865513) --
        (1.365372,2.913525) -- (1.367595,2.937531) -- (1.371107,2.961538) --
        (1.376503,2.985544) -- (1.386770,3.013969) -- (1.400350,3.033556) --
        (1.413529,3.046783) -- (1.428873,3.057563) -- (1.440288,3.063243) --
        (1.467047,3.072386) -- (1.493806,3.078533) -- (1.520565,3.082287) --
        (1.574083,3.085293) -- (1.654360,3.086593) -- (1.948708,3.086798) --
        (2.269816,3.086079) -- (2.323334,3.084420) -- (2.361765,3.081569) --
        (2.376852,3.079484) -- (2.403611,3.073823) -- (2.457129,3.057562) --
        (3.126104,2.457406) -- (3.125238,2.433400) -- (2.671184,2.025294) --
        (2.510646,1.881256) -- (2.495429,1.857250) -- (2.495429,1.809237) --
        (2.510647,1.785231) -- (3.125238,1.233087) -- (3.126104,1.209081) --
        (2.457129,0.608925) -- (2.403611,0.592664) -- (2.361765,0.584919) --
        (2.323334,0.582067) -- (2.243057,0.580057) -- (2.055744,0.579690) --
        (1.600842,0.580507) -- (1.547324,0.582345) -- (1.520565,0.584200) --
        (1.520565,0.584200) -- cycle;
\end{scope}
\draw[line width=0.9mm, preto] (0.5,0.5) node[anchor= center] {$\Omega$};

\end{tikzpicture}
    \caption{point measurements pattern}
\end{subfigure}%
~    
    \begin{subfigure}[t]{0.334\textwidth} 
        \centering
        \includegraphics[height=1.99in,width=1.99in]{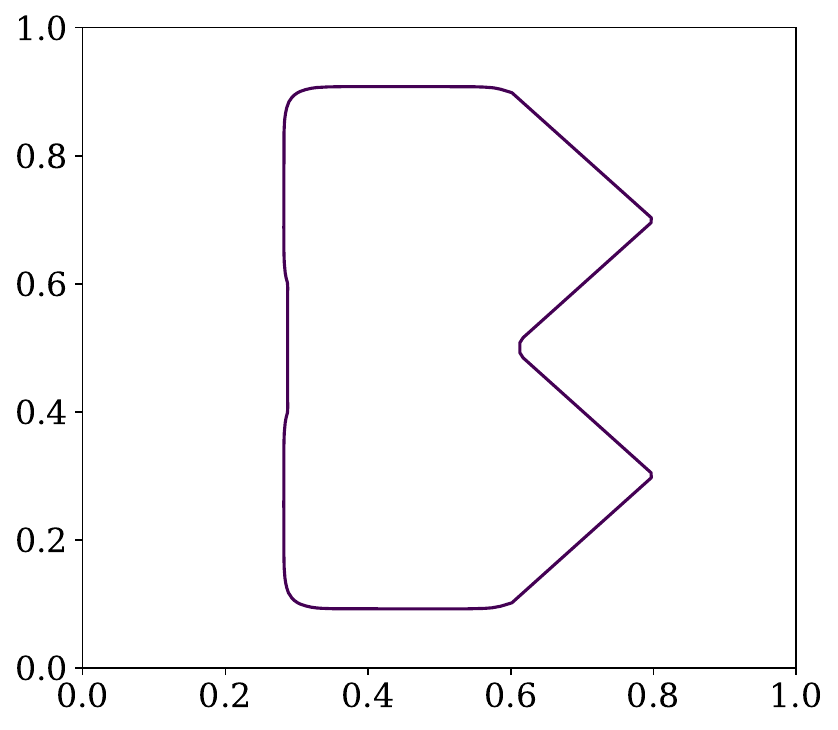}
        \caption{ground truth}
    \end{subfigure}%
~
    \begin{subfigure}[t]{0.334\textwidth} 
        \centering
        \includegraphics[height=1.99in,width=1.99in]{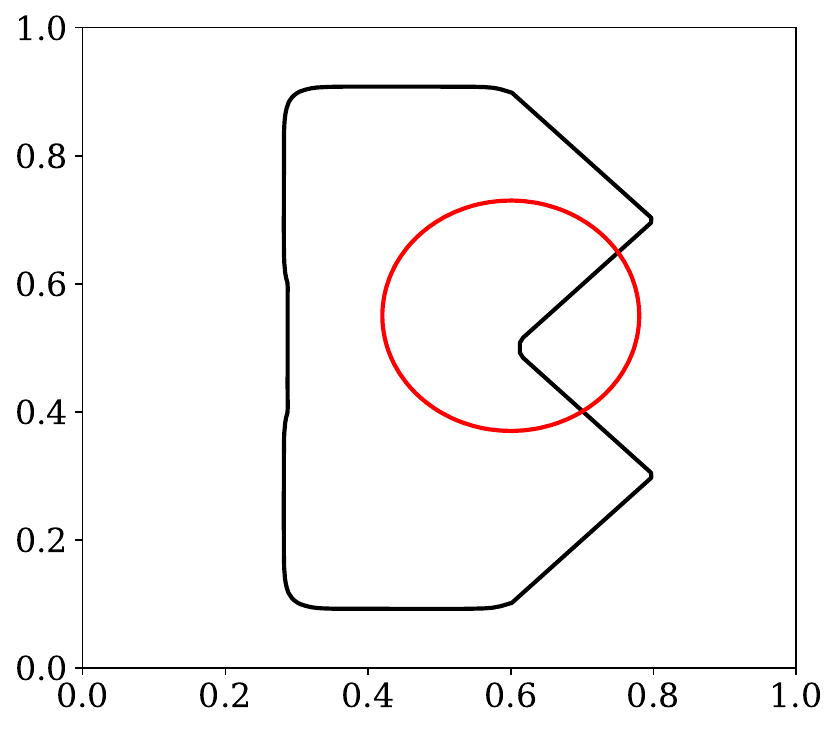}
        \caption{initialization (red)}
    \end{subfigure}%
  
    \begin{subfigure}[t]{0.334\textwidth} 
        \centering
        \includegraphics[height=1.99in,width=1.99in]{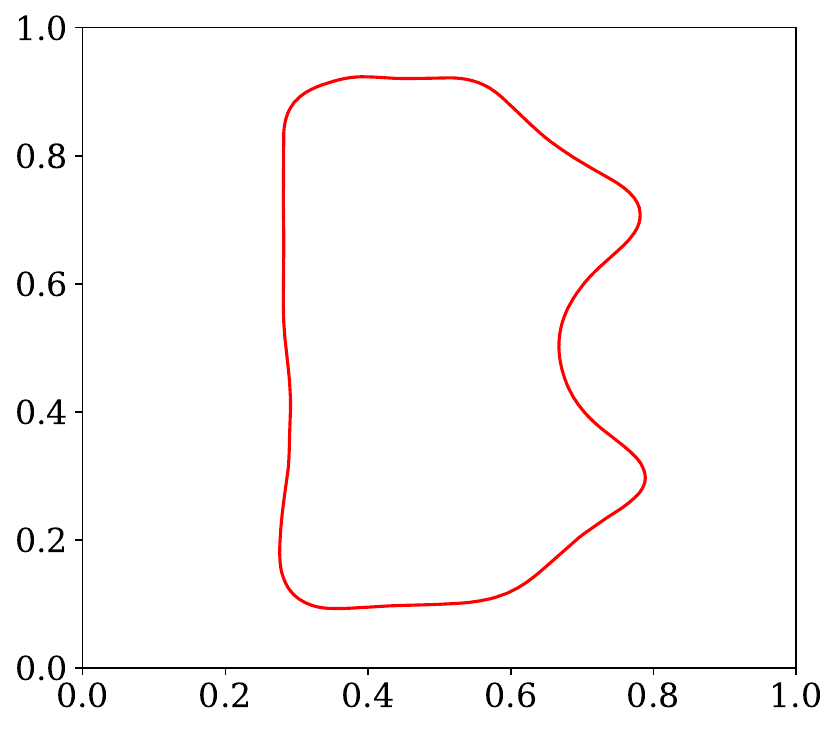}
        \caption{reconstruction}
    \end{subfigure}
    ~
    \begin{subfigure}[t]{0.334\textwidth} 
        \centering
        \includegraphics[height=1.99in,width=1.99in]{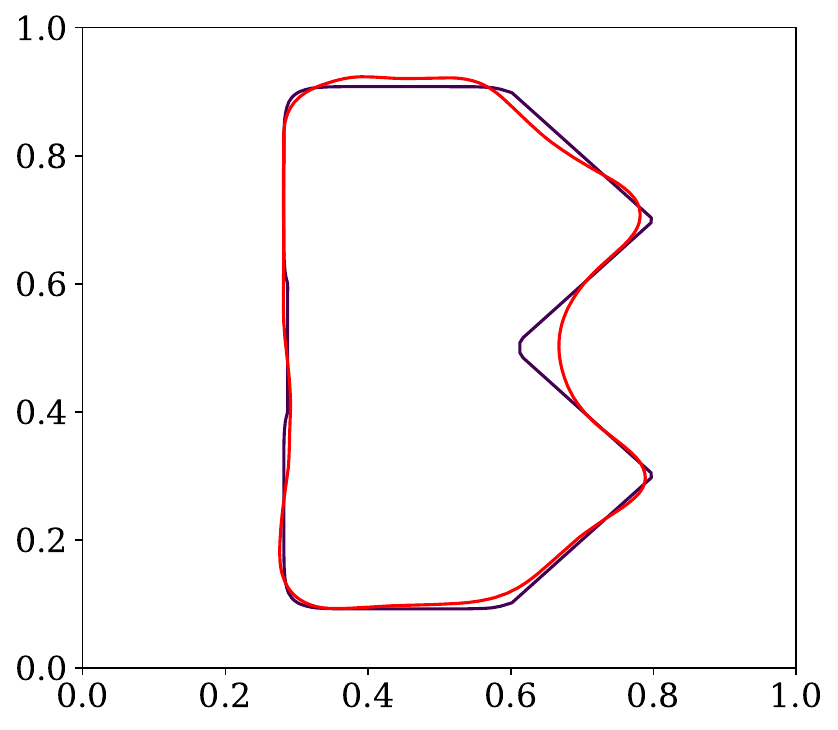}
        \caption{reconstruction (red) and ground truth (black)}
    \end{subfigure}
    \caption{Reconstruction of a concave shape using $I=3$ currents  and $K=34$ point measurements with $0.55\%$ noise.}\label{fig:concave}
\end{figure}
% ================================================================
% Concave shape, 3 currents, compare reconstructions when increasing number of points
% ================================================================
\begin{figure}[ht]
    \centering
    \begin{subfigure}[t]{0.334\textwidth}
        \centering
% \definecolor{preto}{RGB}{0,0,0}
% \definecolor{vermelho}{RGB}{254,49,49}
% \definecolor{azul_escuro}{RGB}{0,0,255}
% \definecolor{cinza}{RGB}{195,195,195}
% \definecolor{papel_amarelo}{RGB}{255, 249, 240}

\begin{tikzpicture}[scale=4.1]
	% background square color	
	%\fill[papel_amarelo, opacity=0.4] (0,0) rectangle (1,1);
	% grid lines
	%\draw[step=0.2cm,gray,very thin, opacity=0.5] (1,0) grid (0,1);
	\draw[line width=0.5mm, preto](0.4,1) -- (0.6,1);
        \draw[line width=0.5mm, preto](0.4,0) -- (0.6,0);
	% border lines left
	\draw[line width=0.7mm, vermelho](0,1) -- (0.4,1);
	\draw[line width=0.7mm, vermelho](0,0) -- (0,1);	
	\draw[line width=0.7mm, vermelho] (0,0) -- (0.4,0);
	% border lines right
    \draw[line width=0.7mm, vermelho](0.6,1) -- (1,1);
	\draw[line width=0.7mm, vermelho](1,0) -- (1,1);
	\draw[line width=0.7mm, vermelho] (0.6,0) -- (1,0);
	% tic top - right
	\draw[line width=0.9mm, preto](0.6,0.97) -- (0.6,1.03);
	% tic top - left
	\draw[line width=0.9mm, preto](0.4,0.97) -- (0.4,1.03);
	% tic bottom - right
	\draw[line width=0.9mm, preto](0.6,-0.03) -- (0.6,0.03);
	% tic bottom - left
	\draw[line width=0.9mm, preto](0.4,-0.03) -- (0.4,0.03);
	%defining dots size
	\def\raio{0.015}
	% defining opacity level 
	\def\transparenc{0.7}	
	% defining step size	
	\def\h{0.2};

	% draw border dots left
	\foreach \x in {0,\h,...,1.0}{
		 \draw[fill= preto, opacity=\transparenc] (0,\x) circle (\raio cm);
}
	% draw border dots right
	\foreach \x in {0,\h,...,1.0}{
		 \draw[fill= preto, opacity=\transparenc] (1,\x) circle (\raio cm);
}
	% draw border dots top-left
	\foreach \x in {\h}{
		 \draw[fill= preto, opacity=\transparenc] (\x,1) circle (\raio cm);
}
	% draw border dots top-righ
	\foreach \x in {4*\h}{
		 \draw[fill= preto, opacity=\transparenc] (\x,1) circle (\raio cm);
}
	% draw border dots bottom-left
	\foreach \x in {\h}{
		 \draw[fill= preto, opacity=\transparenc ] (\x,0) circle (\raio cm);
}
	% draw border dots bottom-righ
	\foreach \x in {4*\h}{
		 \draw[fill= preto, opacity=\transparenc] (\x,0) circle (\raio cm);
}
	% draw Gamma
	\draw (0.5, -0.08) node[anchor= center] {$\Gamma_0$};
	\draw (0.5, 1.08) node[anchor= center] {$\Gamma_0$};

\begin{scope}[shift={(-0.145,-0.06)} ,scale=0.303]
      \path[draw=black, line join=round, line cap=butt, miter limit=10.00, line width=0.3mm, opacity=0.85] 
        (1.520565,0.584200) -- (1.493806,0.587955) -- (1.467047,0.594101)
        -- (1.440288,0.603245) -- (1.428873,0.608925) -- (1.413529,0.619704) --
        (1.400350,0.632931) -- (1.384588,0.656938) -- (1.376503,0.680944) --
        (1.371107,0.704950) -- (1.367595,0.728956) -- (1.364009,0.776969) --
        (1.362501,0.848988) -- (1.362278,1.089050) -- (1.362915,1.377125) --
        (1.364525,1.425138) -- (1.368951,1.473150) -- (1.373207,1.497156) --
        (1.379672,1.521163) -- (1.381099,1.545169) -- (1.380360,1.641194) --
        (1.380713,2.097313) -- (1.381099,2.121319) -- (1.379672,2.145325) --
        (1.373207,2.169331) -- (1.368951,2.193338) -- (1.364525,2.241350) --
        (1.362915,2.289363) -- (1.362285,2.409394) -- (1.363203,2.865513) --
        (1.365372,2.913525) -- (1.367595,2.937531) -- (1.371107,2.961538) --
        (1.376503,2.985544) -- (1.386770,3.013969) -- (1.400350,3.033556) --
        (1.413529,3.046783) -- (1.428873,3.057563) -- (1.440288,3.063243) --
        (1.467047,3.072386) -- (1.493806,3.078533) -- (1.520565,3.082287) --
        (1.574083,3.085293) -- (1.654360,3.086593) -- (1.948708,3.086798) --
        (2.269816,3.086079) -- (2.323334,3.084420) -- (2.361765,3.081569) --
        (2.376852,3.079484) -- (2.403611,3.073823) -- (2.457129,3.057562) --
        (3.126104,2.457406) -- (3.125238,2.433400) -- (2.671184,2.025294) --
        (2.510646,1.881256) -- (2.495429,1.857250) -- (2.495429,1.809237) --
        (2.510647,1.785231) -- (3.125238,1.233087) -- (3.126104,1.209081) --
        (2.457129,0.608925) -- (2.403611,0.592664) -- (2.361765,0.584919) --
        (2.323334,0.582067) -- (2.243057,0.580057) -- (2.055744,0.579690) --
        (1.600842,0.580507) -- (1.547324,0.582345) -- (1.520565,0.584200) --
        (1.520565,0.584200) -- cycle;
        
              \fill[vermelho, opacity=0.2]  
        (1.520565,0.584200) -- (1.493806,0.587955) -- (1.467047,0.594101)
        -- (1.440288,0.603245) -- (1.428873,0.608925) -- (1.413529,0.619704) --
        (1.400350,0.632931) -- (1.384588,0.656938) -- (1.376503,0.680944) --
        (1.371107,0.704950) -- (1.367595,0.728956) -- (1.364009,0.776969) --
        (1.362501,0.848988) -- (1.362278,1.089050) -- (1.362915,1.377125) --
        (1.364525,1.425138) -- (1.368951,1.473150) -- (1.373207,1.497156) --
        (1.379672,1.521163) -- (1.381099,1.545169) -- (1.380360,1.641194) --
        (1.380713,2.097313) -- (1.381099,2.121319) -- (1.379672,2.145325) --
        (1.373207,2.169331) -- (1.368951,2.193338) -- (1.364525,2.241350) --
        (1.362915,2.289363) -- (1.362285,2.409394) -- (1.363203,2.865513) --
        (1.365372,2.913525) -- (1.367595,2.937531) -- (1.371107,2.961538) --
        (1.376503,2.985544) -- (1.386770,3.013969) -- (1.400350,3.033556) --
        (1.413529,3.046783) -- (1.428873,3.057563) -- (1.440288,3.063243) --
        (1.467047,3.072386) -- (1.493806,3.078533) -- (1.520565,3.082287) --
        (1.574083,3.085293) -- (1.654360,3.086593) -- (1.948708,3.086798) --
        (2.269816,3.086079) -- (2.323334,3.084420) -- (2.361765,3.081569) --
        (2.376852,3.079484) -- (2.403611,3.073823) -- (2.457129,3.057562) --
        (3.126104,2.457406) -- (3.125238,2.433400) -- (2.671184,2.025294) --
        (2.510646,1.881256) -- (2.495429,1.857250) -- (2.495429,1.809237) --
        (2.510647,1.785231) -- (3.125238,1.233087) -- (3.126104,1.209081) --
        (2.457129,0.608925) -- (2.403611,0.592664) -- (2.361765,0.584919) --
        (2.323334,0.582067) -- (2.243057,0.580057) -- (2.055744,0.579690) --
        (1.600842,0.580507) -- (1.547324,0.582345) -- (1.520565,0.584200) --
        (1.520565,0.584200) -- cycle;
\end{scope}
\draw[line width=0.9mm, preto] (0.5,0.5) node[anchor= center] {$\Omega$};

\end{tikzpicture}
    \end{subfigure}%
    ~ 
\begin{subfigure}[t]{0.334\textwidth}
        \centering
% \definecolor{preto}{RGB}{0,0,0}
% \definecolor{vermelho}{RGB}{254,49,49}
% \definecolor{azul_escuro}{RGB}{0,0,255}
% \definecolor{cinza}{RGB}{195,195,195}
% \definecolor{papel_amarelo}{RGB}{255, 249, 240}

%y=15.0mm, x=15.0mm, yscale=1.0, xscale=1.0, inner sep=0pt, outer sep=0pt
\begin{tikzpicture}[scale=4.1]
% background square color	
	%\fill[papel_amarelo, opacity=0.1] (0,0) rectangle (1,1);
	% grid lines
	%\draw[step=0.2cm,gray,very thin, opacity=0.5] (1,0) grid (0,1);
	\draw[line width=0.5mm, preto](0.4,1) -- (0.6,1);
        \draw[line width=0.5mm, preto](0.4,0) -- (0.6,0);
	% border lines left
	\draw[line width=0.7mm, vermelho](0,1) -- (0.4,1);
	\draw[line width=0.7mm, vermelho](0,0) -- (0,1);	
	\draw[line width=0.7mm, vermelho] (0,0) -- (0.4,0);
	% border lines right
    \draw[line width=0.7mm, vermelho](0.6,1) -- (1,1);
	\draw[line width=0.7mm, vermelho](1,0) -- (1,1);
	\draw[line width=0.7mm, vermelho] (0.6,0) -- (1,0);
	% tic top - right
	\draw[line width=0.9mm, preto](0.6,0.97) -- (0.6,1.03);
	% tic top - left
	\draw[line width=0.9mm, preto](0.4,0.97) -- (0.4,1.03);
	% tic bottom - right
	\draw[line width=0.9mm, preto](0.6,-0.03) -- (0.6,0.03);
	% tic bottom - left
	\draw[line width=0.9mm, preto](0.4,-0.03) -- (0.4,0.03);
	
	%defining dots size
	\def\raio{0.015}
	% defining opacity level 
	\def\transparenc{0.7}	
	% defining step size	
	\def\h{0.1};
	% draw border dots left
	\foreach \x in {0,\h,...,1.1}{
		 \draw[fill= preto, opacity=\transparenc] (0,\x) circle (\raio cm);
}
	% draw border dots right
	\foreach \x in {0,\h, ...,1.1}{
		 \draw[fill= preto, opacity=\transparenc] (1,\x) circle (\raio cm);
}
	% draw border dots top-left
	\foreach \x in {0.0,\h , ..., 0.4}{
		 \draw[fill= preto, opacity=\transparenc] (\x,1) circle (\raio cm);
}
	% draw border dots top-righ
	\foreach \x in {0.7, 0.8, 0.9}{
		 \draw[fill= preto, opacity=\transparenc] (\x,1) circle (\raio cm);
}
	% draw border dots bottom-left
	\foreach \x in {0.1, 0.2,0.3}{
		 \draw[fill= preto, opacity=\transparenc ] (\x,0) circle (\raio cm);
}
	% draw border dots bottom-righ
	\foreach \x in {0.7,0.8, 0.9}{
		 \draw[fill= preto, opacity=\transparenc] (\x,0) circle (\raio cm);
}
	% draw Gamma
	\draw (0.5, -0.08) node[anchor= center] {$\Gamma_0$};
	\draw (0.5, 1.08) node[anchor= center] {$\Gamma_0$};

\begin{scope}[shift={(-0.145,-0.06)} ,scale=0.303]
      \path[draw=black, line join=round, line cap=butt, miter limit=10.00, line width=0.3mm, opacity=0.85] 
        (1.520565,0.584200) -- (1.493806,0.587955) -- (1.467047,0.594101)
        -- (1.440288,0.603245) -- (1.428873,0.608925) -- (1.413529,0.619704) --
        (1.400350,0.632931) -- (1.384588,0.656938) -- (1.376503,0.680944) --
        (1.371107,0.704950) -- (1.367595,0.728956) -- (1.364009,0.776969) --
        (1.362501,0.848988) -- (1.362278,1.089050) -- (1.362915,1.377125) --
        (1.364525,1.425138) -- (1.368951,1.473150) -- (1.373207,1.497156) --
        (1.379672,1.521163) -- (1.381099,1.545169) -- (1.380360,1.641194) --
        (1.380713,2.097313) -- (1.381099,2.121319) -- (1.379672,2.145325) --
        (1.373207,2.169331) -- (1.368951,2.193338) -- (1.364525,2.241350) --
        (1.362915,2.289363) -- (1.362285,2.409394) -- (1.363203,2.865513) --
        (1.365372,2.913525) -- (1.367595,2.937531) -- (1.371107,2.961538) --
        (1.376503,2.985544) -- (1.386770,3.013969) -- (1.400350,3.033556) --
        (1.413529,3.046783) -- (1.428873,3.057563) -- (1.440288,3.063243) --
        (1.467047,3.072386) -- (1.493806,3.078533) -- (1.520565,3.082287) --
        (1.574083,3.085293) -- (1.654360,3.086593) -- (1.948708,3.086798) --
        (2.269816,3.086079) -- (2.323334,3.084420) -- (2.361765,3.081569) --
        (2.376852,3.079484) -- (2.403611,3.073823) -- (2.457129,3.057562) --
        (3.126104,2.457406) -- (3.125238,2.433400) -- (2.671184,2.025294) --
        (2.510646,1.881256) -- (2.495429,1.857250) -- (2.495429,1.809237) --
        (2.510647,1.785231) -- (3.125238,1.233087) -- (3.126104,1.209081) --
        (2.457129,0.608925) -- (2.403611,0.592664) -- (2.361765,0.584919) --
        (2.323334,0.582067) -- (2.243057,0.580057) -- (2.055744,0.579690) --
        (1.600842,0.580507) -- (1.547324,0.582345) -- (1.520565,0.584200) --
        (1.520565,0.584200) -- cycle;
        
              \fill[vermelho, opacity=0.2]  
        (1.520565,0.584200) -- (1.493806,0.587955) -- (1.467047,0.594101)
        -- (1.440288,0.603245) -- (1.428873,0.608925) -- (1.413529,0.619704) --
        (1.400350,0.632931) -- (1.384588,0.656938) -- (1.376503,0.680944) --
        (1.371107,0.704950) -- (1.367595,0.728956) -- (1.364009,0.776969) --
        (1.362501,0.848988) -- (1.362278,1.089050) -- (1.362915,1.377125) --
        (1.364525,1.425138) -- (1.368951,1.473150) -- (1.373207,1.497156) --
        (1.379672,1.521163) -- (1.381099,1.545169) -- (1.380360,1.641194) --
        (1.380713,2.097313) -- (1.381099,2.121319) -- (1.379672,2.145325) --
        (1.373207,2.169331) -- (1.368951,2.193338) -- (1.364525,2.241350) --
        (1.362915,2.289363) -- (1.362285,2.409394) -- (1.363203,2.865513) --
        (1.365372,2.913525) -- (1.367595,2.937531) -- (1.371107,2.961538) --
        (1.376503,2.985544) -- (1.386770,3.013969) -- (1.400350,3.033556) --
        (1.413529,3.046783) -- (1.428873,3.057563) -- (1.440288,3.063243) --
        (1.467047,3.072386) -- (1.493806,3.078533) -- (1.520565,3.082287) --
        (1.574083,3.085293) -- (1.654360,3.086593) -- (1.948708,3.086798) --
        (2.269816,3.086079) -- (2.323334,3.084420) -- (2.361765,3.081569) --
        (2.376852,3.079484) -- (2.403611,3.073823) -- (2.457129,3.057562) --
        (3.126104,2.457406) -- (3.125238,2.433400) -- (2.671184,2.025294) --
        (2.510646,1.881256) -- (2.495429,1.857250) -- (2.495429,1.809237) --
        (2.510647,1.785231) -- (3.125238,1.233087) -- (3.126104,1.209081) --
        (2.457129,0.608925) -- (2.403611,0.592664) -- (2.361765,0.584919) --
        (2.323334,0.582067) -- (2.243057,0.580057) -- (2.055744,0.579690) --
        (1.600842,0.580507) -- (1.547324,0.582345) -- (1.520565,0.584200) --
        (1.520565,0.584200) -- cycle;
\end{scope}
\draw[line width=0.9mm, preto] (0.5,0.5) node[anchor= center] {$\Omega$};

\end{tikzpicture}
    \end{subfigure}%
~
\begin{subfigure}[t]{0.334\textwidth}
        \centering
% \definecolor{preto}{RGB}{0,0,0}
% \definecolor{vermelho}{RGB}{254,49,49}
% \definecolor{azul_escuro}{RGB}{0,0,255}
% \definecolor{cinza}{RGB}{195,195,195}
% \definecolor{papel_amarelo}{RGB}{255, 249, 240}

\begin{tikzpicture}[scale=4.1]
	% background square color	
	%\fill[papel_amarelo, opacity=0.1] (0,0) rectangle (1,1);
	% grid lines
	%\draw[step=0.2cm,gray,very thin, opacity=0.5] (1,0) grid (0,1);
	\draw[line width=0.5mm, preto](0.4,1) -- (0.6,1);
        \draw[line width=0.5mm, preto](0.4,0) -- (0.6,0);
	% border lines left
	\draw[line width=0.7mm, vermelho](0,1) -- (0.4,1);
	\draw[line width=0.7mm, vermelho](0,0) -- (0,1);	
	\draw[line width=0.7mm, vermelho] (0,0) -- (0.4,0);
	% border lines right
    \draw[line width=0.7mm, vermelho](0.6,1) -- (1,1);
	\draw[line width=0.7mm, vermelho](1,0) -- (1,1);
	\draw[line width=0.7mm, vermelho] (0.6,0) -- (1,0);
	% tic top - right
	\draw[line width=0.9mm, preto](0.6,0.97) -- (0.6,1.03);
	% tic top - left
	\draw[line width=0.9mm, preto](0.4,0.97) -- (0.4,1.03);
	% tic bottom - right
	\draw[line width=0.9mm, preto](0.6,-0.03) -- (0.6,0.03);
	% tic bottom - left
	\draw[line width=0.9mm, preto](0.4,-0.03) -- (0.4,0.03);
	% ellipses : lengths in the brackets separated by an and, are the x-direction radius and the y-direction radius respectively
	% ellipse above
	%\draw[line width=0.9mm, preto] (0.35,0.65) ellipse (0.2cm and 0.1cm) node[anchor= center] {$\Omega$};
	% below below
    %\draw[line width=0.9mm, preto] (0.75,0.35) ellipse (0.1cm and 0.2cm) node[anchor= center] {$\Omega$};
	%defining dots size
	\def\raio{0.015}
	% defining opacity level 
	\def\transparenc{0.7}	
	% defining step size	
	\def\h{0.05};
	% draw border dots left
	\foreach \x in {0,\h,...,1.0}{
		 \draw[fill= preto, opacity=\transparenc] (0,\x) circle (\raio cm);
}
	% draw border dots right
	\foreach \x in {0,\h, ...,1.0}{
		 \draw[fill= preto, opacity=\transparenc] (1,\x) circle (\raio cm);
}
	% draw border dots top-left
	\foreach \x in {0.0,\h , ..., 0.4}{
		 \draw[fill= preto, opacity=\transparenc] (\x,1) circle (\raio cm);
}
	% draw border dots top-righ
	\foreach \x in {0.65, 0.7, ...,1.05}{
		 \draw[fill= preto, opacity=\transparenc] (\x,1) circle (\raio cm);
}
	% draw border dots bottom-left
	\foreach \x in {0.0, 0.05, ...,0.4}{
		 \draw[fill= preto, opacity=\transparenc ] (\x,0) circle (\raio cm);
}
	% draw border dots bottom-righ
	\foreach \x in {0.65,0.7, ..., 1}{
		 \draw[fill= preto, opacity=\transparenc] (\x,0) circle (\raio cm);
}
	% draw Gamma
	\draw (0.5, -0.08) node[anchor= center] {$\Gamma_0$};
	\draw (0.5, 1.08) node[anchor= center] {$\Gamma_0$};

\begin{scope}[shift={(-0.145,-0.06)} ,scale=0.303]
      \path[draw=black, line join=round, line cap=butt, miter limit=10.00, line width=0.3mm, opacity=0.85] 
        (1.520565,0.584200) -- (1.493806,0.587955) -- (1.467047,0.594101)
        -- (1.440288,0.603245) -- (1.428873,0.608925) -- (1.413529,0.619704) --
        (1.400350,0.632931) -- (1.384588,0.656938) -- (1.376503,0.680944) --
        (1.371107,0.704950) -- (1.367595,0.728956) -- (1.364009,0.776969) --
        (1.362501,0.848988) -- (1.362278,1.089050) -- (1.362915,1.377125) --
        (1.364525,1.425138) -- (1.368951,1.473150) -- (1.373207,1.497156) --
        (1.379672,1.521163) -- (1.381099,1.545169) -- (1.380360,1.641194) --
        (1.380713,2.097313) -- (1.381099,2.121319) -- (1.379672,2.145325) --
        (1.373207,2.169331) -- (1.368951,2.193338) -- (1.364525,2.241350) --
        (1.362915,2.289363) -- (1.362285,2.409394) -- (1.363203,2.865513) --
        (1.365372,2.913525) -- (1.367595,2.937531) -- (1.371107,2.961538) --
        (1.376503,2.985544) -- (1.386770,3.013969) -- (1.400350,3.033556) --
        (1.413529,3.046783) -- (1.428873,3.057563) -- (1.440288,3.063243) --
        (1.467047,3.072386) -- (1.493806,3.078533) -- (1.520565,3.082287) --
        (1.574083,3.085293) -- (1.654360,3.086593) -- (1.948708,3.086798) --
        (2.269816,3.086079) -- (2.323334,3.084420) -- (2.361765,3.081569) --
        (2.376852,3.079484) -- (2.403611,3.073823) -- (2.457129,3.057562) --
        (3.126104,2.457406) -- (3.125238,2.433400) -- (2.671184,2.025294) --
        (2.510646,1.881256) -- (2.495429,1.857250) -- (2.495429,1.809237) --
        (2.510647,1.785231) -- (3.125238,1.233087) -- (3.126104,1.209081) --
        (2.457129,0.608925) -- (2.403611,0.592664) -- (2.361765,0.584919) --
        (2.323334,0.582067) -- (2.243057,0.580057) -- (2.055744,0.579690) --
        (1.600842,0.580507) -- (1.547324,0.582345) -- (1.520565,0.584200) --
        (1.520565,0.584200) -- cycle;
        
              \fill[vermelho, opacity=0.2]  
        (1.520565,0.584200) -- (1.493806,0.587955) -- (1.467047,0.594101)
        -- (1.440288,0.603245) -- (1.428873,0.608925) -- (1.413529,0.619704) --
        (1.400350,0.632931) -- (1.384588,0.656938) -- (1.376503,0.680944) --
        (1.371107,0.704950) -- (1.367595,0.728956) -- (1.364009,0.776969) --
        (1.362501,0.848988) -- (1.362278,1.089050) -- (1.362915,1.377125) --
        (1.364525,1.425138) -- (1.368951,1.473150) -- (1.373207,1.497156) --
        (1.379672,1.521163) -- (1.381099,1.545169) -- (1.380360,1.641194) --
        (1.380713,2.097313) -- (1.381099,2.121319) -- (1.379672,2.145325) --
        (1.373207,2.169331) -- (1.368951,2.193338) -- (1.364525,2.241350) --
        (1.362915,2.289363) -- (1.362285,2.409394) -- (1.363203,2.865513) --
        (1.365372,2.913525) -- (1.367595,2.937531) -- (1.371107,2.961538) --
        (1.376503,2.985544) -- (1.386770,3.013969) -- (1.400350,3.033556) --
        (1.413529,3.046783) -- (1.428873,3.057563) -- (1.440288,3.063243) --
        (1.467047,3.072386) -- (1.493806,3.078533) -- (1.520565,3.082287) --
        (1.574083,3.085293) -- (1.654360,3.086593) -- (1.948708,3.086798) --
        (2.269816,3.086079) -- (2.323334,3.084420) -- (2.361765,3.081569) --
        (2.376852,3.079484) -- (2.403611,3.073823) -- (2.457129,3.057562) --
        (3.126104,2.457406) -- (3.125238,2.433400) -- (2.671184,2.025294) --
        (2.510646,1.881256) -- (2.495429,1.857250) -- (2.495429,1.809237) --
        (2.510647,1.785231) -- (3.125238,1.233087) -- (3.126104,1.209081) --
        (2.457129,0.608925) -- (2.403611,0.592664) -- (2.361765,0.584919) --
        (2.323334,0.582067) -- (2.243057,0.580057) -- (2.055744,0.579690) --
        (1.600842,0.580507) -- (1.547324,0.582345) -- (1.520565,0.584200) --
        (1.520565,0.584200) -- cycle;
\end{scope}
\draw[line width=0.9mm, preto] (0.5,0.5) node[anchor= center] {$\Omega$};
\end{tikzpicture}
    \end{subfigure}
    
    \begin{subfigure}[t]{0.334\textwidth} 
        \centering
        \includegraphics[height=1.99in,width=1.99in]{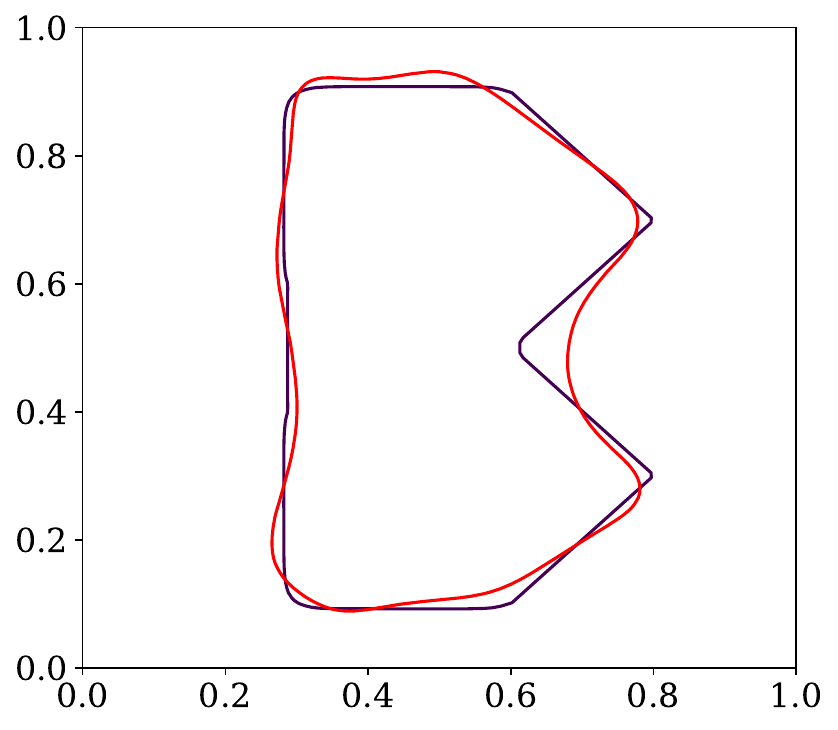}
        \caption{$K=16$ point measurements,\\ $1.03 \%$ noise, $7.53\%$ relative error}
    \end{subfigure}%
~
    \begin{subfigure}[t]{0.334\textwidth} 
        \centering
        \includegraphics[height=1.99in,width=1.99in]{./tests2/3-1-1-1/Difference.pdf}
        \caption{$K=34$ point measurements,\\ $0.73 \%$ noise, $4.77\%$ relative error}
    \end{subfigure}%
    ~
    \begin{subfigure}[t]{0.334\textwidth} 
        \centering
        \includegraphics[height=1.99in,width=1.99in]{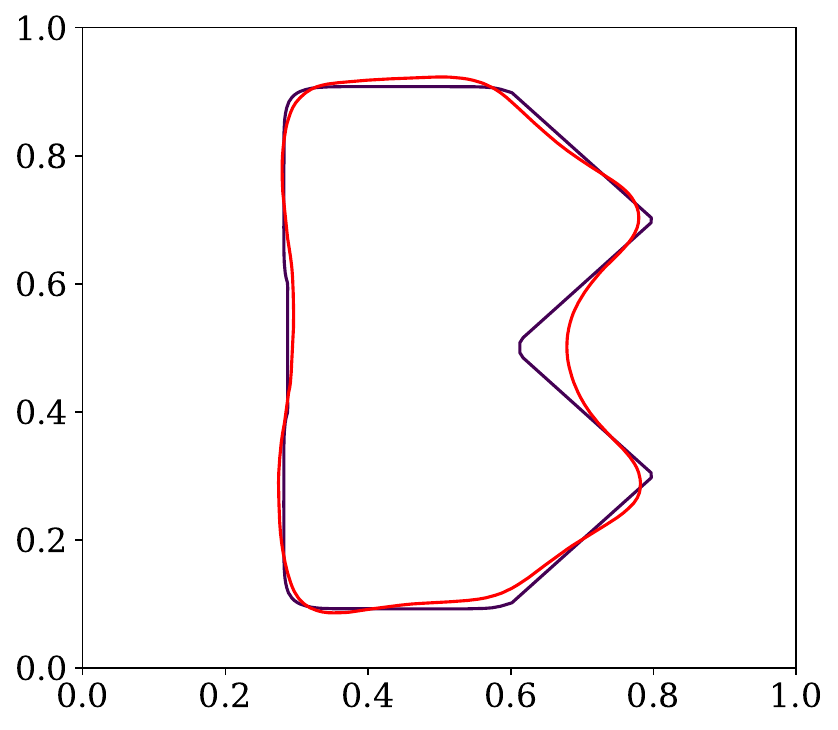}
        \caption{$K=70$ point measurements,\\ $0.71 \%$ noise, $5.37\%$ relative error}
    \end{subfigure}
    \caption{Reconstruction of a concave shape using $I=3$ currents  and three different sets of point measurements shown in the first row.}\label{concave_nbpt}
\end{figure}

% ================================================================
% influence of noise on the reconstruction of a concave shape
% ================================================================
\begin{table}[ht]
  \centering
  \begin{tabular}{cccc}
    \hline
     noise & $K=16$ points &  $K=34$ points &  $K=70$ points \\ \hline
     $0\%$ 
     &
    \begin{minipage}{0.25\textwidth}
            \centering{\vspace{0.1cm}{\scriptsize error: $8.40\%$}}\\
      \includegraphics[width=1\textwidth, height = 1\textwidth]{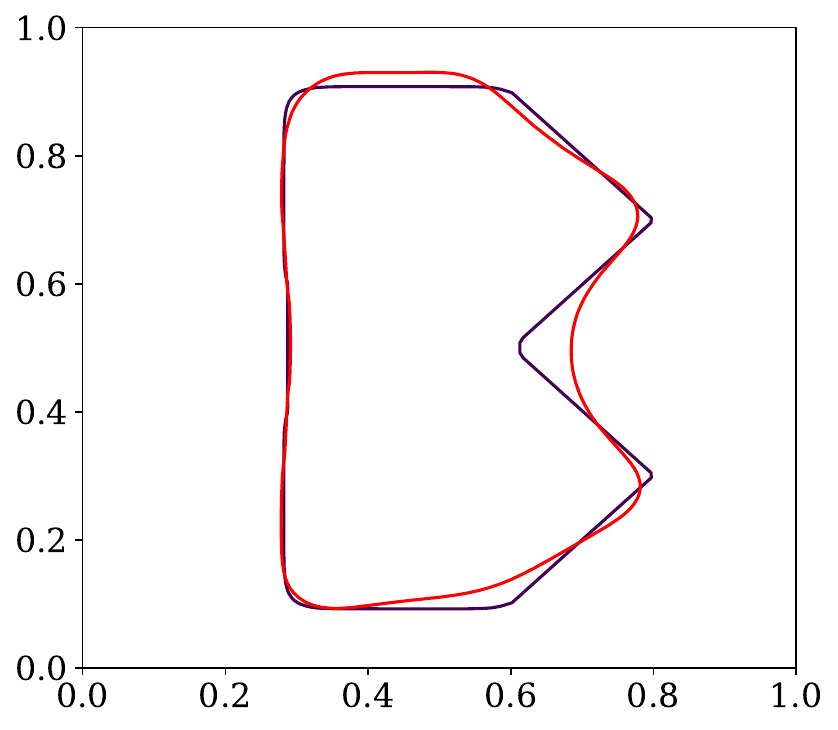}
      \end{minipage}
    &
    \begin{minipage}{0.25\textwidth}
                \centering{\vspace{0.1cm}{\scriptsize error: $5.84\%$}}\\
      \includegraphics[width=1\textwidth, height = 1\textwidth]{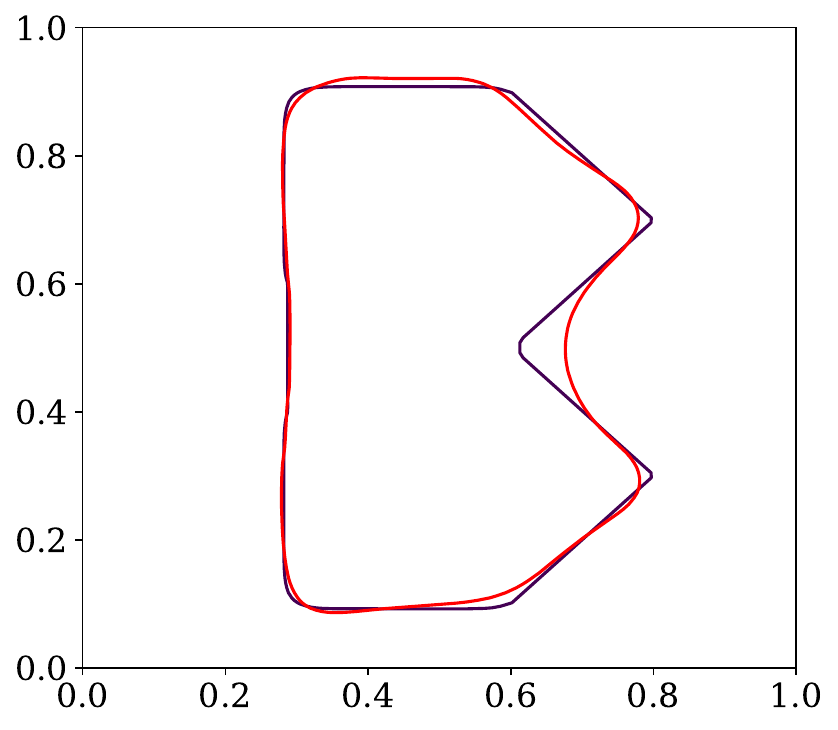}
    \end{minipage}
    & 
    \begin{minipage}{0.25\textwidth}
                \centering{\vspace{0.1cm}{\scriptsize error: $4.75\%$}}\\
      \includegraphics[width=1\textwidth, height = 1\textwidth]{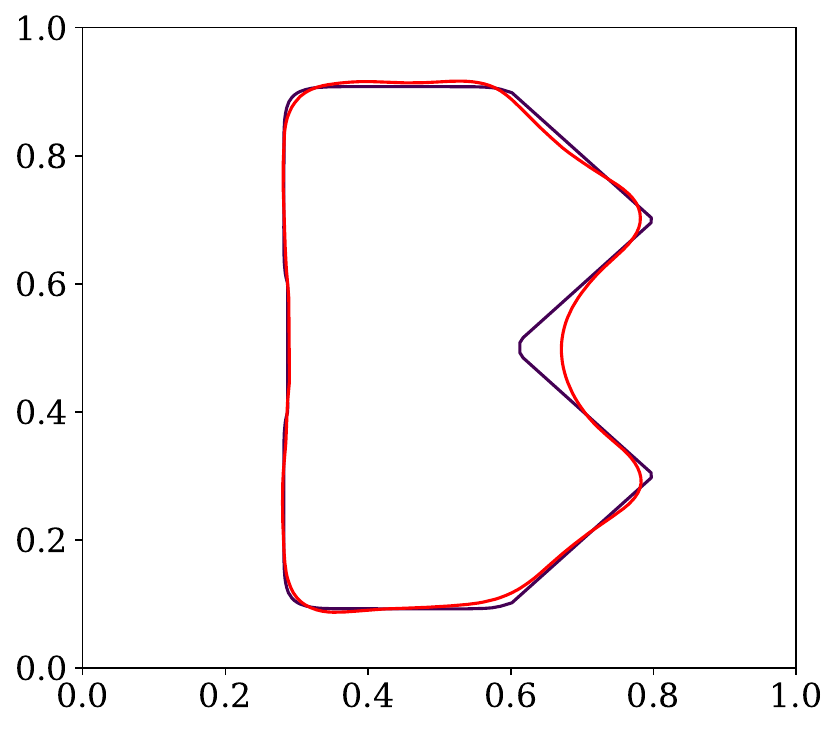}
    \end{minipage}
    \\ \hline   
    %%%%%%%%%%%%%%%%%
    $0.54\%$
    &
    \begin{minipage}{0.25\textwidth}
                \centering{\vspace{0.1cm}{\scriptsize error: $8.81\%$}}\\
      \includegraphics[width=1\textwidth, height = 1\textwidth]{./tests2/3-1-0-1/Difference.pdf}
    \end{minipage}
    &
    \begin{minipage}{0.25\textwidth}
                \centering{\vspace{0.1cm}{\scriptsize error: $5.28\%$}}\\
      \includegraphics[width=1\textwidth, height = 1\textwidth]{./tests2/3-1-1-1/Difference.pdf}
    \end{minipage}
    & 
    \begin{minipage}{0.25\textwidth}
                \centering{\vspace{0.1cm}{\scriptsize error: $6.77\%$}}\\
      \includegraphics[width=1\textwidth, height = 1\textwidth]{./tests2/3-1-2-1/Difference.pdf}
    \end{minipage}
    \\ \hline  
    %%%%%%%%%%%%%%%%%%
    $1.12\%$
    &
    \begin{minipage}{0.25\textwidth}
                \centering{\vspace{0.1cm}{\scriptsize error: $13.90\%$}}\\
      \includegraphics[width=1\textwidth, height = 1\textwidth]{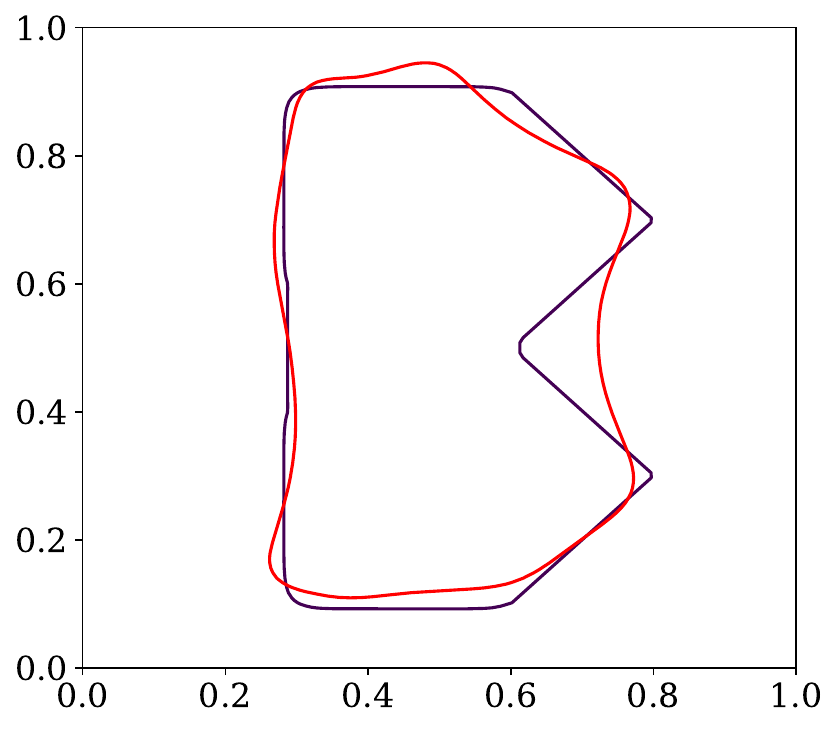}
    \end{minipage}
    &
    \begin{minipage}{0.25\textwidth}
                \centering{\vspace{0.1cm}{\scriptsize error: $8.28\%$}}\\
      \includegraphics[width=1\textwidth, height = 1\textwidth]{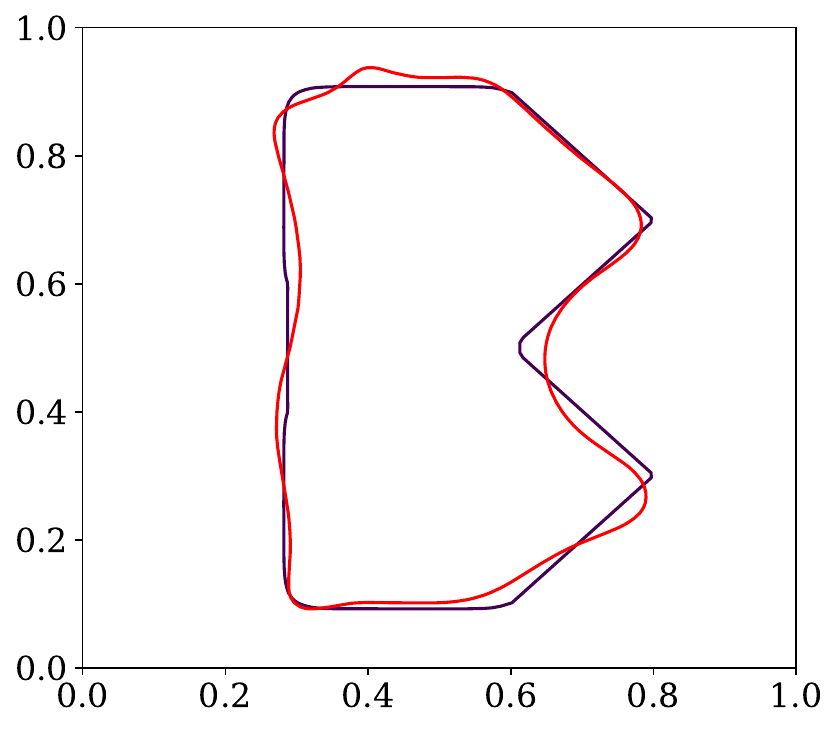}
    \end{minipage}
    & 
    \begin{minipage}{0.25\textwidth}
                \centering{\vspace{0.1cm}{\scriptsize error: $10.03\%$}}\\
      \includegraphics[width=1\textwidth, height = 1\textwidth]{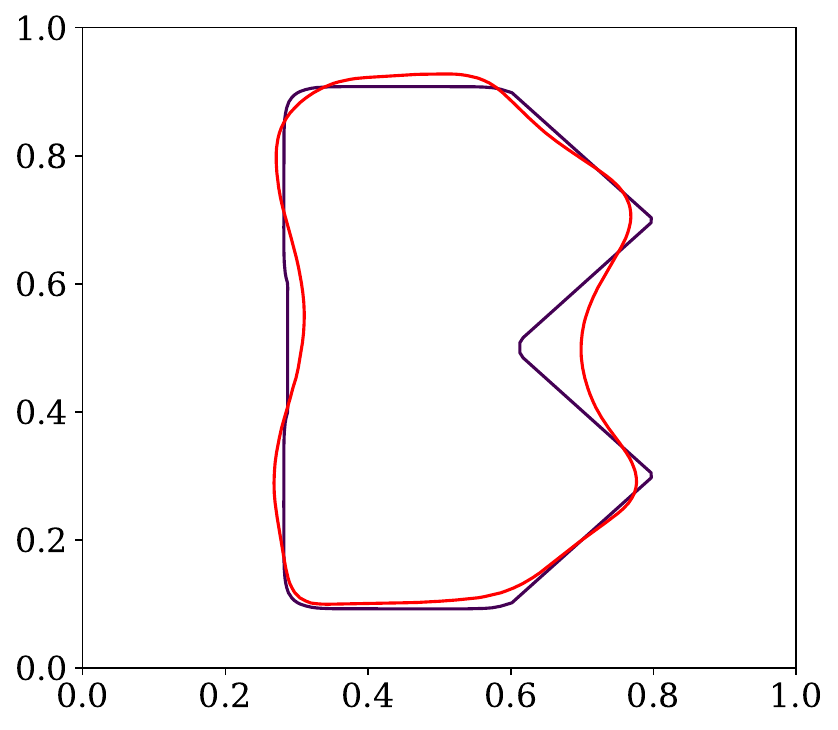}
    \end{minipage}
  \end{tabular}
  \caption{Influence of noise and number of point measurements on the reconstruction of a concave shape using $I=3$ currents (the noise value is the average over the noise values for the three levels of point measurements). }\label{fig:noise_influence_concave}
\end{table}

% ================================================================
% Three ellipses, 7 currents
% ================================================================
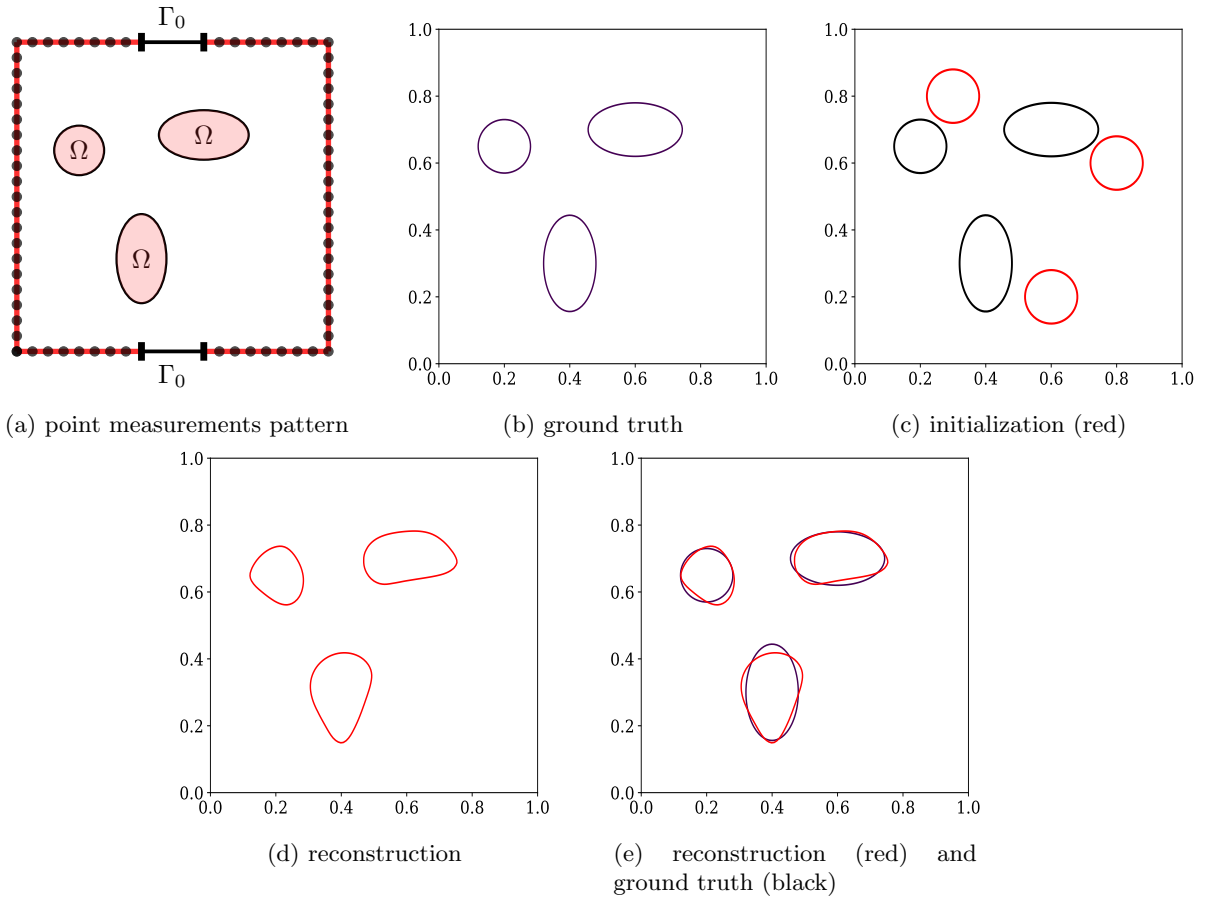
\begin{figure}[ht]
\centering
\begin{subfigure}[t]{0.334\textwidth}
    \centering
% \definecolor{preto}{RGB}{0,0,0}
% \definecolor{vermelho}{RGB}{254,49,49}
% \definecolor{azul_escuro}{RGB}{0,0,255}
% \definecolor{cinza}{RGB}{195,195,195}
% \definecolor{papel_amarelo}{RGB}{255, 249, 240}

\begin{tikzpicture}[scale=4.1]
	% background square color
	%\fill[papel_amarelo, opacity=0.4] (0,0) rectangle (1,1);
	%%%%%%%%%%%%%%%%
\draw[line width=0.5mm, preto](0.4,1) -- (0.6,1);
%%%%%%%%%%%%%%%%
\draw[line width=0.5mm, preto](0.4,0) -- (0.6,0);
	% grid lines
	%\draw[step=0.2cm,gray,very thin, opacity=0.5] (1,0) grid (0,1);
	% border lines left
	\draw[line width=0.7mm, vermelho](0,1) -- (0.4,1);
	\draw[line width=0.7mm, vermelho](0,0) -- (0,1);	
	\draw[line width=0.7mm, vermelho] (0,0) -- (0.4,0);
	% border lines right
    \draw[line width=0.7mm, vermelho](0.6,1) -- (1,1);
	\draw[line width=0.7mm, vermelho](1,0) -- (1,1);
	\draw[line width=0.7mm, vermelho] (0.6,0) -- (1,0);
	% tic top - right
	\draw[line width=0.9mm, preto](0.6,0.97) -- (0.6,1.03);
	% tic top - left
	\draw[line width=0.9mm, preto](0.4,0.97) -- (0.4,1.03);
	% tic bottom - right
	\draw[line width=0.9mm, preto](0.6,-0.03) -- (0.6,0.03);
	% tic bottom - left
	\draw[line width=0.9mm, preto](0.4,-0.03) -- (0.4,0.03);
	% ellipses : lengths in the brackets separated by an and, are the x-direction radius and the y-direction radius respectively
	\draw[line width=0.3mm, preto] (0.6,0.7) ellipse (0.144cm and 0.08cm) node[anchor= center] {$\Om$};
\fill[vermelho, opacity=0.2] (0.6,0.7) ellipse (0.144cm and 0.08cm);
% below below
\draw[line width=0.3mm, preto] (0.4,0.3) ellipse (0.08cm and 0.144cm) node[anchor= center] {$\Om$};
\fill[vermelho, opacity=0.2] (0.4,0.3) ellipse (0.08cm and 0.144cm);
% below below
\draw[line width=0.3mm, preto] (0.2,0.65) ellipse (0.08cm and 0.08cm) node[anchor= center] {$\Om$};
\fill[vermelho, opacity=0.2] (0.2,0.65) ellipse (0.08cm and 0.08cm);	
	%defining dots size
	%defining dots size
	\def\raio{0.015}
	% defining opacity level 
	\def\transparenc{0.7}	
	% defining step size	
	\def\h{0.05};
	% draw border dots left
	\foreach \x in {0,\h,...,1.0}{
		 \draw[fill= preto, opacity=\transparenc] (0,\x) circle (\raio cm);
}
	% draw border dots right
	\foreach \x in {0,\h, ...,1.0}{
		 \draw[fill= preto, opacity=\transparenc] (1,\x) circle (\raio cm);
}
	% draw border dots top-left
	\foreach \x in {0.0,\h , ..., 0.4}{
		 \draw[fill= preto, opacity=\transparenc] (\x,1) circle (\raio cm);
}
	% draw border dots top-righ
	\foreach \x in {0.65, 0.7, ...,1.05}{
		 \draw[fill= preto, opacity=\transparenc] (\x,1) circle (\raio cm);
}
	% draw border dots bottom-left
	\foreach \x in {0.0, 0.05, ...,0.4}{
		 \draw[fill= preto, opacity=\transparenc ] (\x,0) circle (\raio cm);
}
	% draw border dots bottom-righ
	\foreach \x in {0.65,0.7, ..., 1}{
		 \draw[fill= preto, opacity=\transparenc] (\x,0) circle (\raio cm);
}
	% draw Gamma
	\draw (0.5, -0.08) node[anchor= center] {$\Gamma_0$};
	\draw (0.5, 1.08) node[anchor= center] {$\Gamma_0$};
\end{tikzpicture}
    \caption{point measurements pattern}
\end{subfigure}%
~    
    \begin{subfigure}[t]{0.334\textwidth} 
        \centering
        \includegraphics[height=1.99in,width=1.99in]{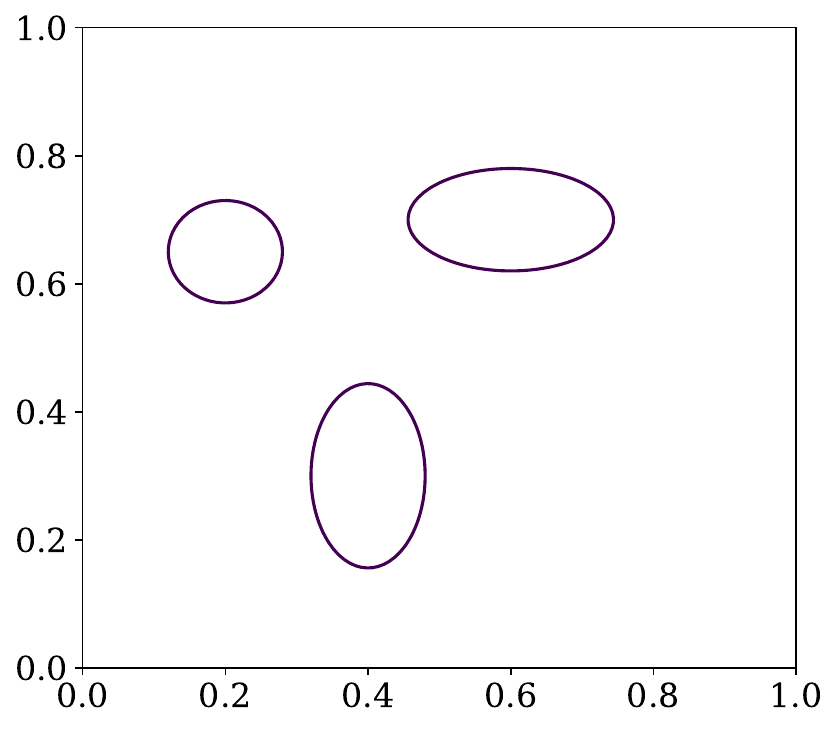}
        \caption{ground truth}
    \end{subfigure}%
~
    \begin{subfigure}[t]{0.334\textwidth} 
        \centering
        \includegraphics[height=1.99in,width=1.99in]{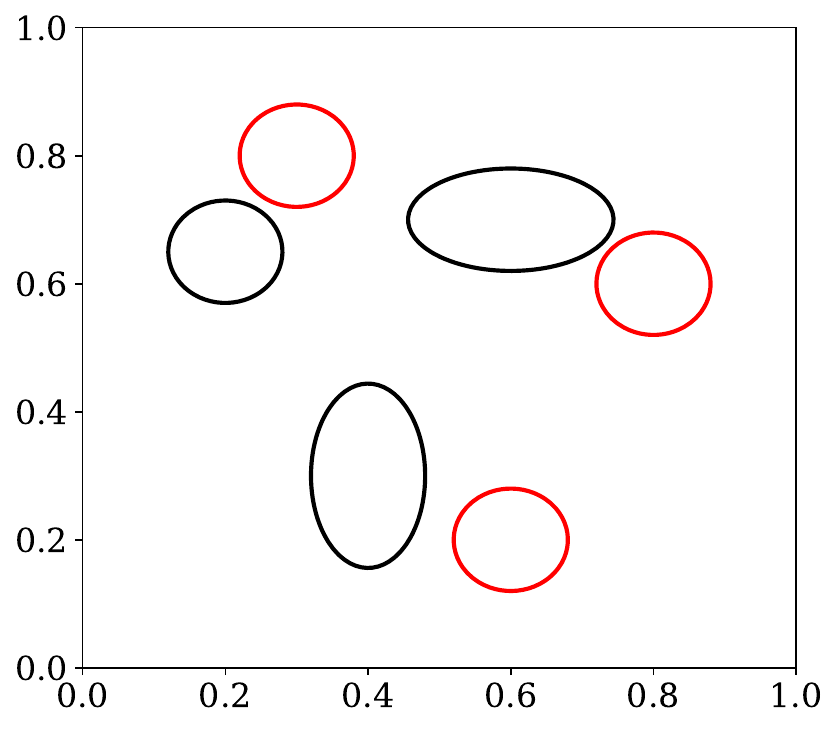}
        \caption{initialization (red)}
    \end{subfigure}%
  
    \begin{subfigure}[t]{0.334\textwidth} 
        \centering
        \includegraphics[height=1.99in,width=1.99in]{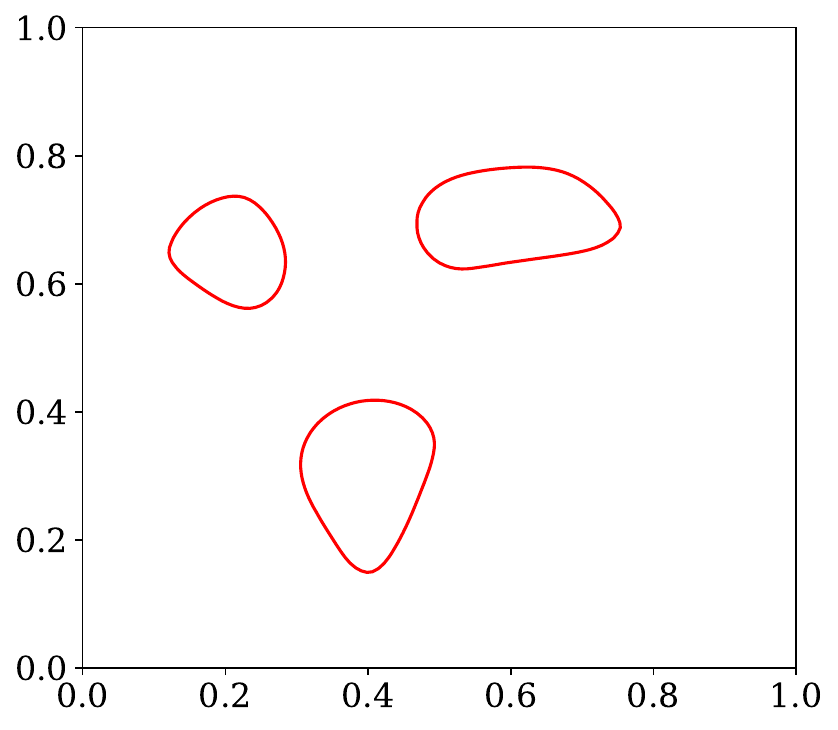}
        \caption{reconstruction}
    \end{subfigure}
    ~
    \begin{subfigure}[t]{0.334\textwidth} 
        \centering
        \includegraphics[height=1.99in,width=1.99in]{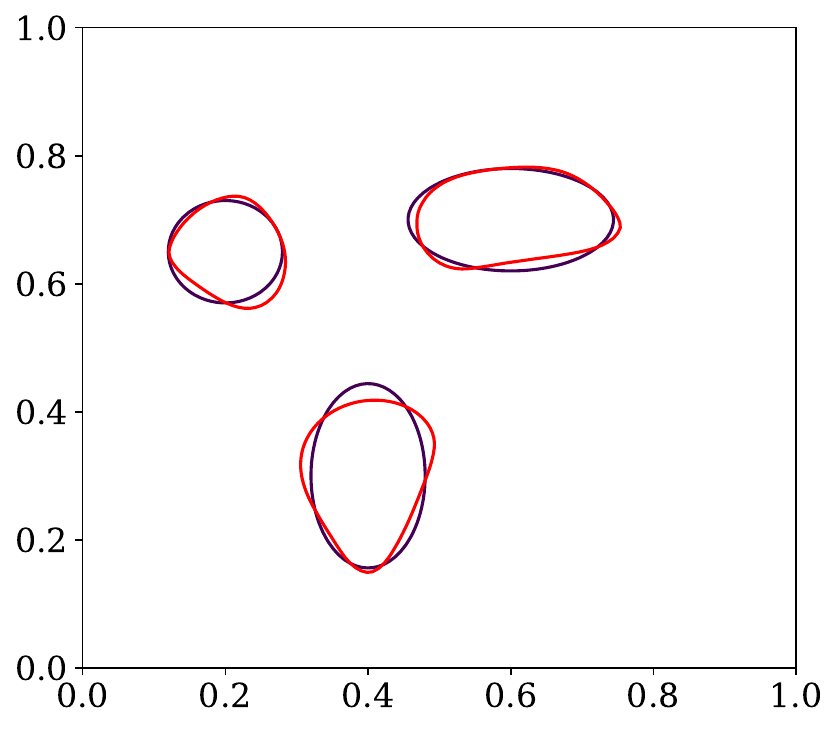}
        \caption{reconstruction (red) and ground truth (black)}
    \end{subfigure}
    \caption{Reconstruction of three ellipses using $I=7$ currents  and $K=70$ point measurements with $0.63\%$ noise.}\label{fig:3ellipses}
\end{figure}

% ================================================================
% three ellipses, 7 currents, compare reconstructions when increasing number of points
% ================================================================
\begin{figure}[ht]
    \centering
    \begin{subfigure}[t]{0.334\textwidth}
        \centering
\definecolor{preto}{RGB}{0,0,0}
\definecolor{vermelho}{RGB}{254,49,49}
\definecolor{azul_escuro}{RGB}{0,0,255}
\definecolor{cinza}{RGB}{195,195,195}
\definecolor{papel_amarelo}{RGB}{255, 249, 240}

\begin{tikzpicture}[scale=4.1]
	% background square color	
	%\fill[papel_amarelo, opacity=0.4] (0,0) rectangle (1,1);
	%%%%%%%%%%%%%%%%
\draw[line width=0.5mm, preto](0.4,1) -- (0.6,1);
%%%%%%%%%%%%%%%%
\draw[line width=0.5mm, preto](0.4,0) -- (0.6,0);
	% grid lines
	%\draw[step=0.2cm,gray,very thin, opacity=0.5] (1,0) grid (0,1);
	% border lines left
	\draw[line width=0.7mm, vermelho](0,1) -- (0.4,1);
	\draw[line width=0.7mm, vermelho](0,0) -- (0,1);
	\draw[line width=0.7mm, vermelho] (0,0) -- (0.4,0);
	% border lines right
    \draw[line width=0.7mm, vermelho](0.6,1) -- (1,1);
	\draw[line width=0.7mm, vermelho](1,0) -- (1,1);
	\draw[line width=0.7mm, vermelho] (0.6,0) -- (1,0);
	% tic top - right
	\draw[line width=0.9mm, preto](0.6,0.97) -- (0.6,1.03);
	% tic top - left
	\draw[line width=0.9mm, preto](0.4,0.97) -- (0.4,1.03);
	% tic bottom - right
	\draw[line width=0.9mm, preto](0.6,-0.03) -- (0.6,0.03);
	% tic bottom - left
	\draw[line width=0.9mm, preto](0.4,-0.03) -- (0.4,0.03);
	% ellipses : lengths in the brackets separated by an and, are the x-direction radius and the y-direction radius respectively
	% ellipse above
	\draw[line width=0.3mm, preto] (0.6,0.7) ellipse (0.144cm and 0.08cm) node[anchor= center] {$\Om$};
\fill[vermelho, opacity=0.2] (0.6,0.7) ellipse (0.144cm and 0.08cm);
% below below
\draw[line width=0.3mm, preto] (0.4,0.3) ellipse (0.08cm and 0.144cm) node[anchor= center] {$\Om$};
\fill[vermelho, opacity=0.2] (0.4,0.3) ellipse (0.08cm and 0.144cm);
% below below
\draw[line width=0.3mm, preto] (0.2,0.65) ellipse (0.08cm and 0.08cm) node[anchor= center] {$\Om$};
\fill[vermelho, opacity=0.2] (0.2,0.65) ellipse (0.08cm and 0.08cm);	
	%defining dots size
	\def\raio{0.015}
	% defining opacity level 
	\def\transparenc{0.7}	
	% defining step size	
	\def\h{0.2};

	% draw border dots left
	\foreach \x in {0,\h,...,1.0}{
		 \draw[fill= preto, opacity=\transparenc] (0,\x) circle (\raio cm);
}
	% draw border dots right
	\foreach \x in {0,\h,...,1.0}{
		 \draw[fill= preto, opacity=\transparenc] (1,\x) circle (\raio cm);
}
	% draw border dots top-left
	\foreach \x in {\h}{
		 \draw[fill= preto, opacity=\transparenc] (\x,1) circle (\raio cm);
}
	% draw border dots top-righ
	\foreach \x in {4*\h}{
		 \draw[fill= preto, opacity=\transparenc] (\x,1) circle (\raio cm);
}
	% draw border dots bottom-left
	\foreach \x in {\h}{
		 \draw[fill= preto, opacity=\transparenc ] (\x,0) circle (\raio cm);
}
	% draw border dots bottom-righ
	\foreach \x in {4*\h}{
		 \draw[fill= preto, opacity=\transparenc] (\x,0) circle (\raio cm);
}
	% draw Gamma
	\draw (0.5, -0.08) node[anchor= center] {$\Gamma_0$};
	\draw (0.5, 1.08) node[anchor= center] {$\Gamma_0$};

\end{tikzpicture}
    \end{subfigure}%
    ~ 
\begin{subfigure}[t]{0.334\textwidth}
        \centering
\definecolor{preto}{RGB}{0,0,0}
\definecolor{vermelho}{RGB}{254,49,49}
\definecolor{azul_escuro}{RGB}{0,0,255}
\definecolor{cinza}{RGB}{195,195,195}
\definecolor{papel_amarelo}{RGB}{255, 249, 240}

\begin{tikzpicture}[scale=4.1]
	% background square color	
	%\fill[papel_amarelo, opacity=0.4] (0,0) rectangle (1,1);
	%%%%%%%%%%%%%%%%
\draw[line width=0.5mm, preto](0.4,1) -- (0.6,1);
%%%%%%%%%%%%%%%%
\draw[line width=0.5mm, preto](0.4,0) -- (0.6,0);
	% grid lines
	%\draw[step=0.2cm,gray,very thin, opacity=0.5] (1,0) grid (0,1);
	% border lines left
	\draw[line width=0.7mm, vermelho](0,1) -- (0.4,1);
	\draw[line width=0.7mm, vermelho](0,0) -- (0,1);	
	\draw[line width=0.7mm, vermelho] (0,0) -- (0.4,0);
	% border lines right
    \draw[line width=0.7mm, vermelho](0.6,1) -- (1,1);
	\draw[line width=0.7mm, vermelho](1,0) -- (1,1);
	\draw[line width=0.7mm, vermelho] (0.6,0) -- (1,0);
	% tic top - right
	\draw[line width=0.9mm, preto](0.6,0.97) -- (0.6,1.03);
	% tic top - left
	\draw[line width=0.9mm, preto](0.4,0.97) -- (0.4,1.03);
	% tic bottom - right
	\draw[line width=0.9mm, preto](0.6,-0.03) -- (0.6,0.03);
	% tic bottom - left
	\draw[line width=0.9mm, preto](0.4,-0.03) -- (0.4,0.03);
	% ellipses : lengths in the brackets separated by an and, are the x-direction radius and the y-direction radius respectively
	\draw[line width=0.3mm, preto] (0.6,0.7) ellipse (0.144cm and 0.08cm) node[anchor= center] {$\Om$};
\fill[vermelho, opacity=0.2] (0.6,0.7) ellipse (0.144cm and 0.08cm);
% below below
\draw[line width=0.3mm, preto] (0.4,0.3) ellipse (0.08cm and 0.144cm) node[anchor= center] {$\Om$};
\fill[vermelho, opacity=0.2] (0.4,0.3) ellipse (0.08cm and 0.144cm);
% below below
\draw[line width=0.3mm, preto] (0.2,0.65) ellipse (0.08cm and 0.08cm) node[anchor= center] {$\Om$};
\fill[vermelho, opacity=0.2] (0.2,0.65) ellipse (0.08cm and 0.08cm);	
	%defining dots size
	%defining dots size
	\def\raio{0.015}
	% defining opacity level 
	\def\transparenc{0.7}	
	% defining step size	
	\def\h{0.1};
	% draw border dots left
	\foreach \x in {0,\h,...,1.1}{
		 \draw[fill= preto, opacity=\transparenc] (0,\x) circle (\raio cm);
}
	% draw border dots right
	\foreach \x in {0,\h, ...,1.1}{
		 \draw[fill= preto, opacity=\transparenc] (1,\x) circle (\raio cm);
}
	% draw border dots top-left
	\foreach \x in {0.0,\h , ..., 0.4}{
		 \draw[fill= preto, opacity=\transparenc] (\x,1) circle (\raio cm);
}
	% draw border dots top-righ
	\foreach \x in {0.7, 0.8, 0.9}{
		 \draw[fill= preto, opacity=\transparenc] (\x,1) circle (\raio cm);
}
	% draw border dots bottom-left
	\foreach \x in {0.1, 0.2,0.3}{
		 \draw[fill= preto, opacity=\transparenc ] (\x,0) circle (\raio cm);
}
	% draw border dots bottom-righ
	\foreach \x in {0.7,0.8, 0.9}{
		 \draw[fill= preto, opacity=\transparenc] (\x,0) circle (\raio cm);
}
	% draw Gamma
	\draw (0.5, -0.08) node[anchor= center] {$\Gamma_0$};
	\draw (0.5, 1.08) node[anchor= center] {$\Gamma_0$};
\end{tikzpicture}
    \end{subfigure}%
~
\begin{subfigure}[t]{0.334\textwidth}
        \centering
% \definecolor{preto}{RGB}{0,0,0}
% \definecolor{vermelho}{RGB}{254,49,49}
% \definecolor{azul_escuro}{RGB}{0,0,255}
% \definecolor{cinza}{RGB}{195,195,195}
% \definecolor{papel_amarelo}{RGB}{255, 249, 240}

\begin{tikzpicture}[scale=4.1]
	% background square color
	%\fill[papel_amarelo, opacity=0.4] (0,0) rectangle (1,1);
	%%%%%%%%%%%%%%%%
\draw[line width=0.5mm, preto](0.4,1) -- (0.6,1);
%%%%%%%%%%%%%%%%
\draw[line width=0.5mm, preto](0.4,0) -- (0.6,0);
	% grid lines
	%\draw[step=0.2cm,gray,very thin, opacity=0.5] (1,0) grid (0,1);
	% border lines left
	\draw[line width=0.7mm, vermelho](0,1) -- (0.4,1);
	\draw[line width=0.7mm, vermelho](0,0) -- (0,1);	
	\draw[line width=0.7mm, vermelho] (0,0) -- (0.4,0);
	% border lines right
    \draw[line width=0.7mm, vermelho](0.6,1) -- (1,1);
	\draw[line width=0.7mm, vermelho](1,0) -- (1,1);
	\draw[line width=0.7mm, vermelho] (0.6,0) -- (1,0);
	% tic top - right
	\draw[line width=0.9mm, preto](0.6,0.97) -- (0.6,1.03);
	% tic top - left
	\draw[line width=0.9mm, preto](0.4,0.97) -- (0.4,1.03);
	% tic bottom - right
	\draw[line width=0.9mm, preto](0.6,-0.03) -- (0.6,0.03);
	% tic bottom - left
	\draw[line width=0.9mm, preto](0.4,-0.03) -- (0.4,0.03);
	% ellipses : lengths in the brackets separated by an and, are the x-direction radius and the y-direction radius respectively
	\draw[line width=0.3mm, preto] (0.6,0.7) ellipse (0.144cm and 0.08cm) node[anchor= center] {$\Om$};
\fill[vermelho, opacity=0.2] (0.6,0.7) ellipse (0.144cm and 0.08cm);
% below below
\draw[line width=0.3mm, preto] (0.4,0.3) ellipse (0.08cm and 0.144cm) node[anchor= center] {$\Om$};
\fill[vermelho, opacity=0.2] (0.4,0.3) ellipse (0.08cm and 0.144cm);
% below below
\draw[line width=0.3mm, preto] (0.2,0.65) ellipse (0.08cm and 0.08cm) node[anchor= center] {$\Om$};
\fill[vermelho, opacity=0.2] (0.2,0.65) ellipse (0.08cm and 0.08cm);	
	%defining dots size
	%defining dots size
	\def\raio{0.015}
	% defining opacity level 
	\def\transparenc{0.7}	
	% defining step size	
	\def\h{0.05};
	% draw border dots left
	\foreach \x in {0,\h,...,1.0}{
		 \draw[fill= preto, opacity=\transparenc] (0,\x) circle (\raio cm);
}
	% draw border dots right
	\foreach \x in {0,\h, ...,1.0}{
		 \draw[fill= preto, opacity=\transparenc] (1,\x) circle (\raio cm);
}
	% draw border dots top-left
	\foreach \x in {0.0,\h , ..., 0.4}{
		 \draw[fill= preto, opacity=\transparenc] (\x,1) circle (\raio cm);
}
	% draw border dots top-righ
	\foreach \x in {0.65, 0.7, ...,1.05}{
		 \draw[fill= preto, opacity=\transparenc] (\x,1) circle (\raio cm);
}
	% draw border dots bottom-left
	\foreach \x in {0.0, 0.05, ...,0.4}{
		 \draw[fill= preto, opacity=\transparenc ] (\x,0) circle (\raio cm);
}
	% draw border dots bottom-righ
	\foreach \x in {0.65,0.7, ..., 1}{
		 \draw[fill= preto, opacity=\transparenc] (\x,0) circle (\raio cm);
}
	% draw Gamma
	\draw (0.5, -0.08) node[anchor= center] {$\Gamma_0$};
	\draw (0.5, 1.08) node[anchor= center] {$\Gamma_0$};
\end{tikzpicture}
    \end{subfigure}
    
    \begin{subfigure}[t]{0.334\textwidth} 
        \centering
        \includegraphics[height=1.99in,width=1.99in]{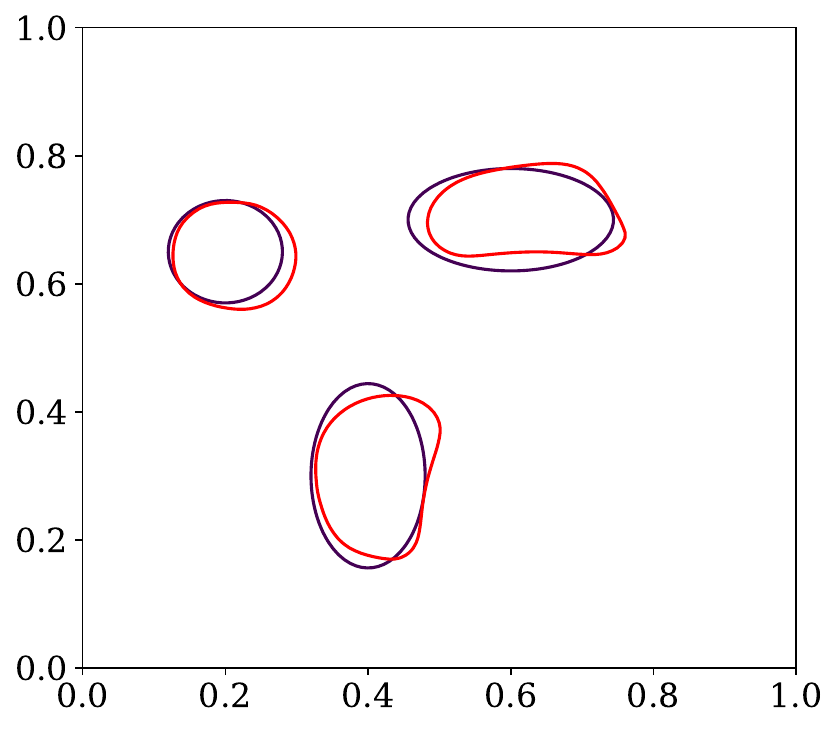}
        \caption{$K=16$ point measurements,\\ $0.59 \%$ noise, $25.52\%$ relative error}
    \end{subfigure}%
~
    \begin{subfigure}[t]{0.334\textwidth} 
        \centering
        \includegraphics[height=1.99in,width=1.99in]{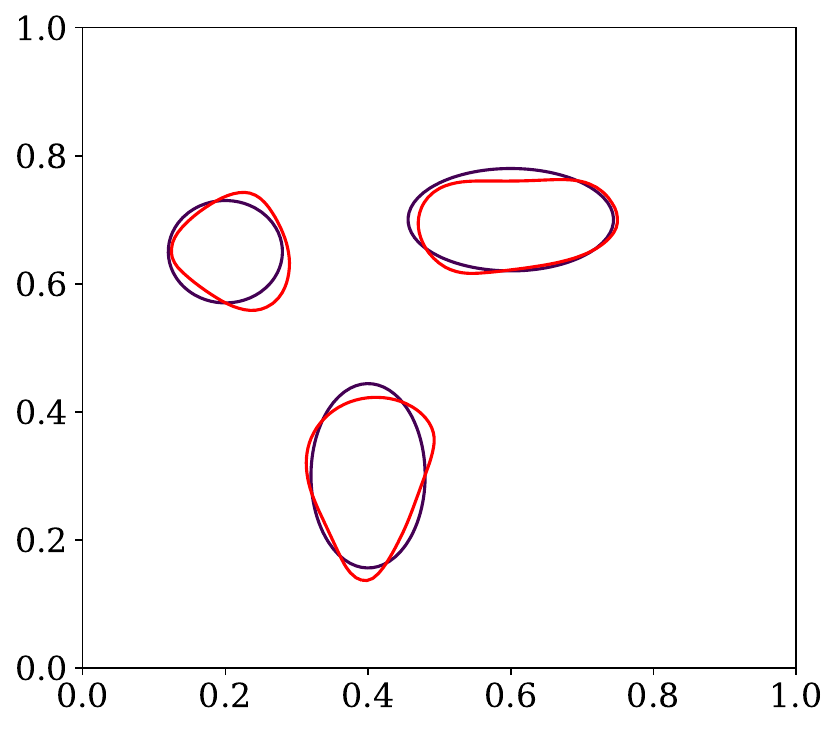}
        \caption{$K=34$ point measurements,\\ $0.54 \%$ noise, $18.8\%$ relative error}
    \end{subfigure}%
    ~
    \begin{subfigure}[t]{0.334\textwidth} 
        \centering
        \includegraphics[height=1.99in,width=1.99in]{./tests2/2-2-2-1/Difference.pdf}
        \caption{$K=70$ point measurements,\\ $0.63 \%$ noise, $16.1\%$ relative error}
    \end{subfigure}
    \caption{Reconstruction of three ellipses using $I=7$ currents  and different sets of point measurements shown in the first row.}\label{3ellipses_nbpt}
\end{figure}

% ================================================================
% influence of noise on the reconstruction of three ellipses
% ================================================================
\begin{table}[ht]
  \centering
  \begin{tabular}{cccc}
    \hline
     noise & $K=16$ points &  $K=34$ points &  $K=70$ points \\ \hline
     $0\%$ 
     &
    \begin{minipage}{0.25\textwidth}
            \centering{\vspace{0.1cm}{\scriptsize error: $25.3\%$}}\\
      \includegraphics[width=1\textwidth, height = 1\textwidth]{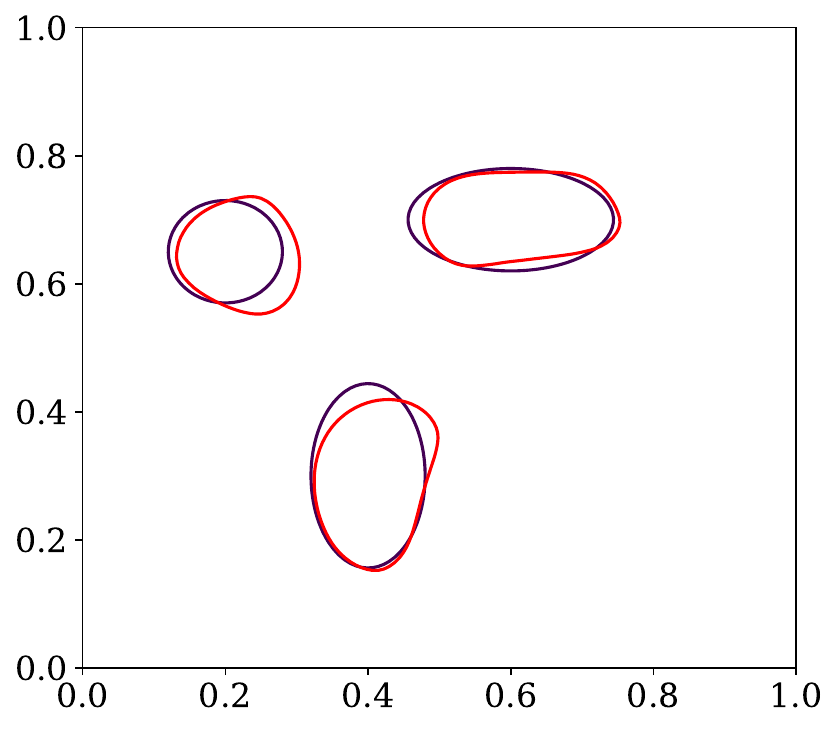}
      \end{minipage}
    &
    \begin{minipage}{0.25\textwidth}
                \centering{\vspace{0.1cm}{\scriptsize error: $17.5\%$}}\\
      \includegraphics[width=1\textwidth, height = 1\textwidth]{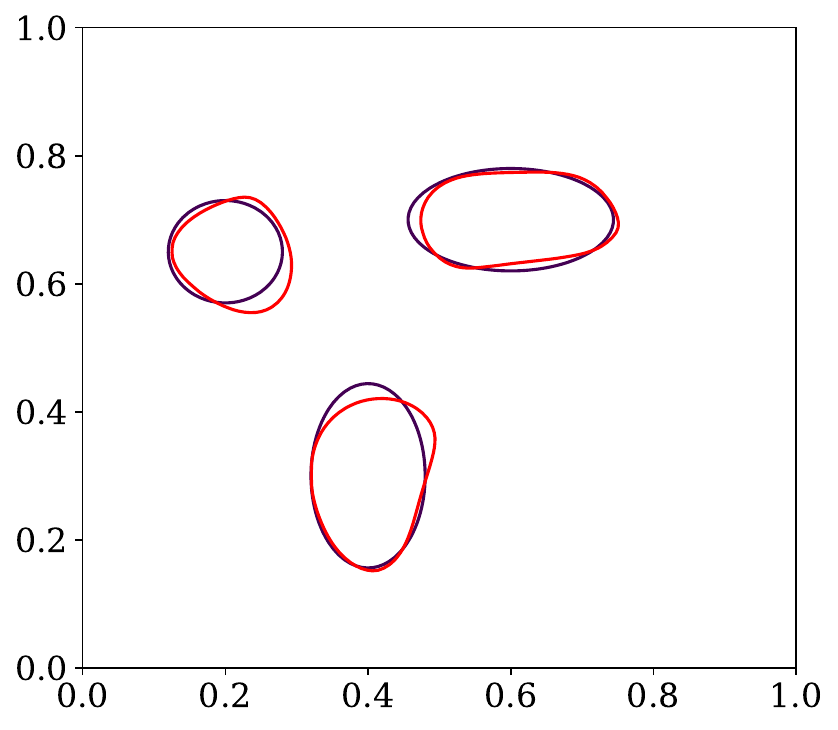}
    \end{minipage}
    & 
    \begin{minipage}{0.25\textwidth}
                \centering{\vspace{0.1cm}{\scriptsize error: $16.3\%$}}\\
      \includegraphics[width=1\textwidth, height = 1\textwidth]{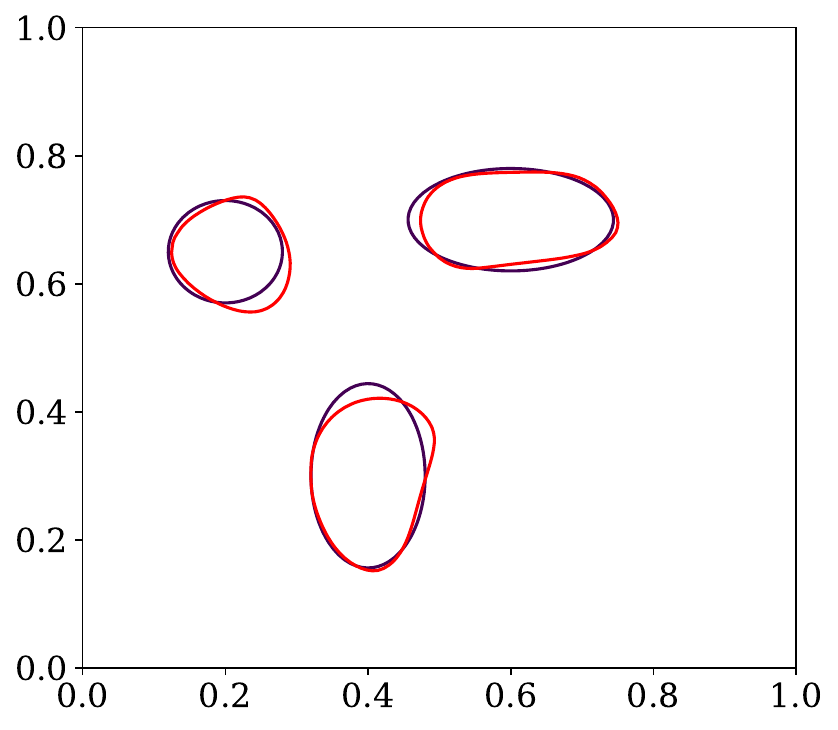}
    \end{minipage}
    \\ \hline   
    %%%%%%%%%%%%%%%%%
    $0.59\%$
    &
    \begin{minipage}{0.25\textwidth}
                \centering{\vspace{0.1cm}{\scriptsize error: $25.5\%$}}\\
      \includegraphics[width=1\textwidth, height = 1\textwidth]{./tests2/2-2-0-1/Difference.pdf}
    \end{minipage}
    &
    \begin{minipage}{0.25\textwidth}
                \centering{\vspace{0.1cm}{\scriptsize error: $18.8\%$}}\\
      \includegraphics[width=1\textwidth, height = 1\textwidth]{./tests2/2-2-1-1/Difference.pdf}
    \end{minipage}
    & 
    \begin{minipage}{0.25\textwidth}
                \centering{\vspace{0.1cm}{\scriptsize error: $16.1\%$}}\\
      \includegraphics[width=1\textwidth, height = 1\textwidth]{./tests2/2-2-2-1/Difference.pdf}
    \end{minipage}
    \\ \hline  
    %%%%%%%%%%%%%%%%%%
    $1.14\%$
    &
    \begin{minipage}{0.25\textwidth}
                \centering{\vspace{0.1cm}{\scriptsize error: $38.4\%$}}\\
      \includegraphics[width=1\textwidth, height = 1\textwidth]{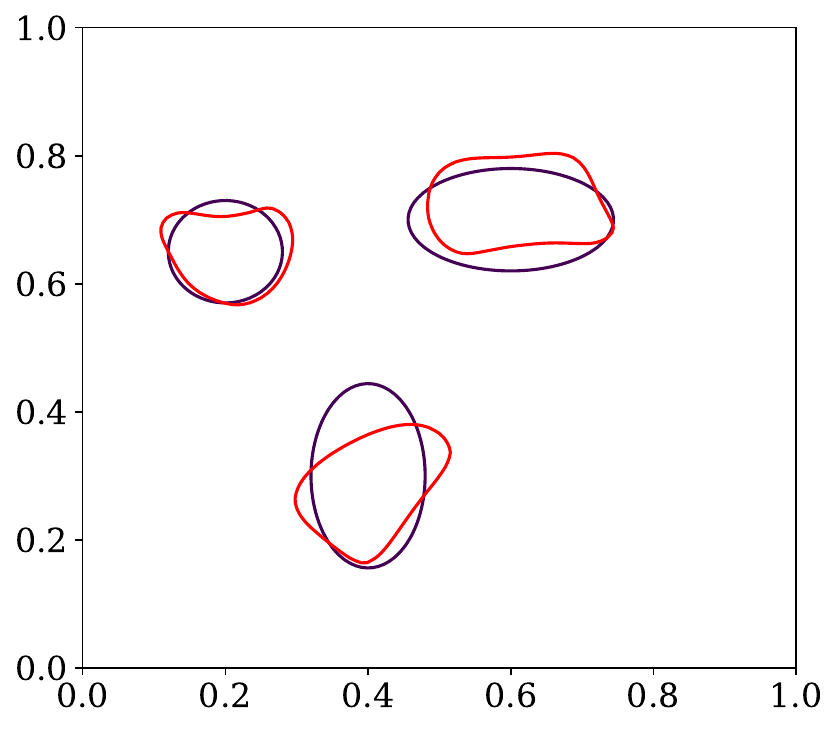}
    \end{minipage}
    &
    \begin{minipage}{0.25\textwidth}
                \centering{\vspace{0.1cm}{\scriptsize error: $38.9\%$}}\\
      \includegraphics[width=1\textwidth, height = 1\textwidth]{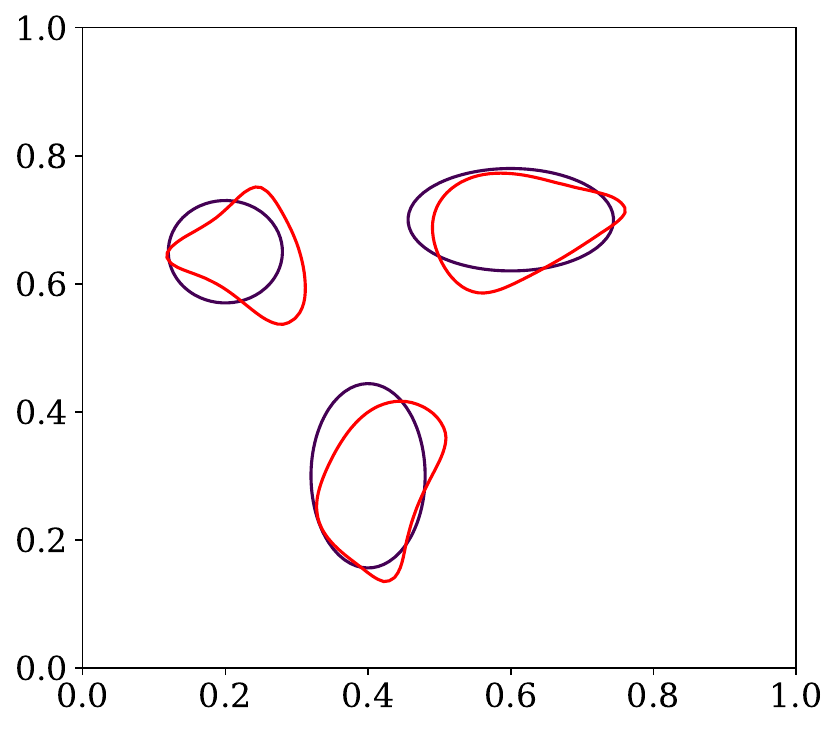}
    \end{minipage}
    & 
    \begin{minipage}{0.25\textwidth}
                \centering{\vspace{0.1cm}{\scriptsize error: $29.7\%$}}\\
      \includegraphics[width=1\textwidth, height = 1\textwidth]{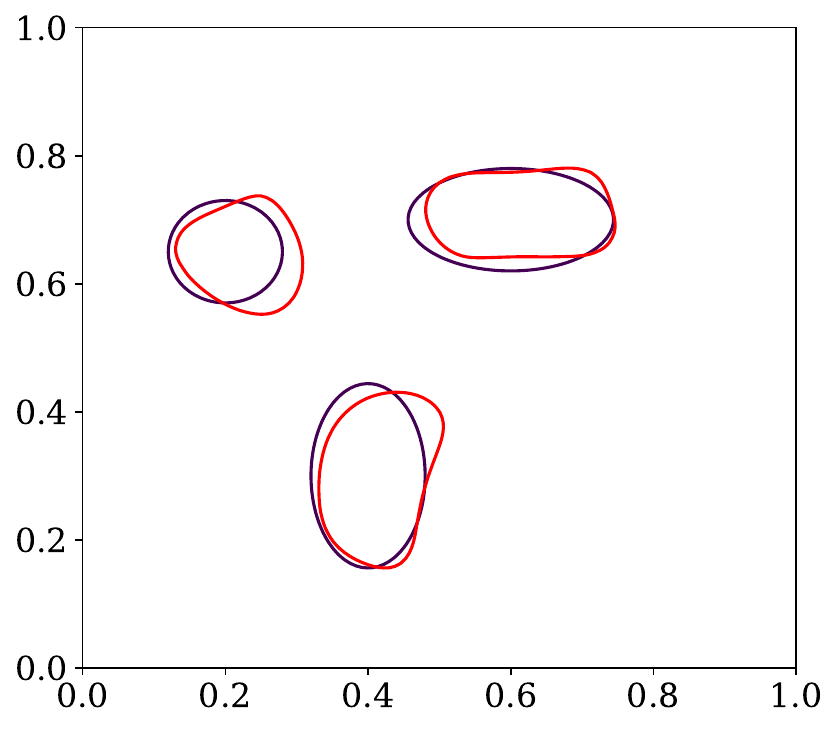}
    \end{minipage}
  \end{tabular}
  \caption{Influence of noise and number of point measurements on the reconstruction of three ellipses using $I=7$ currents (the noise value is the average over the noise values for the three levels of point measurements).}\label{fig:noise_influence}
\end{table}

%%%
\noindent{\bf Acknowledgements.} Yuri Flores Albuquerque and Antoine Laurain gratefully acknowledge support of the RCGI - Research Centre for Gas Innovation, hosted by the University of São Paulo (USP) and sponsored by FAPESP - S\~ao Paulo Research Foundation (2014/50279-4) and Shell Brasil. This research was carried out in association with the ongoing R\&D project registered as ANP 20714-2 - Desenvolvimento de t\'ecnicas num\'ericas e software para problemas de invers\~ao com aplica\c{c}\~oes em processamento s\'ismico (USP / Shell Brasil / ANP), sponsored by Shell Brasil under the ANP R\&D levy as ``Compromisso de Investimentos com Pesquisa e Desenvolvimento''.
Antoine Laurain gratefully acknowledges the support of FAPESP, process: 2016/24776-6 ``Otimiza\c{c}\~ao de forma e problemas de fronteira livre'', and of the Brazilian National Council for Scientific and Technological Development  (Conselho Nacional de Desenvolvimento Cient\'ifico e Tecnol\'ogico - CNPq) through the process: 408175/2018-4 ``Otimiza\c{c}\~ao de forma n\~ao suave e controle de problemas de fronteira livre'', and through the program  ``Bolsa de Produtividade em Pesquisa - PQ 2018'', process: 304258/2018-0.

%%%%%%%%%%%%%%%%%%%%%%%%%%%%%%%%%%%%
\section{Appendix 1: averaged adjoint method}\label{sec:appendix}
%%%%%%%%%%%%%%%%%%%%%%%%%%%%%%%%%%%%%

Let $\tz>0$ be given and $E = E(\D), F=F(\D)$ be  two Banach spaces, and consider a parameterization 
$\Omega_t = \Tt(\Omega)$ for $t\in [0,\tz]$ such that $\Tt(\D)=\D$, i.e. which leaves $\D$ globally invariant. 
Our goal is to differentiate shape functions of the type 
$J(\Omega_t)$ which can be written using a Lagrangian as $J(\Omega_t) = \mathcal{L}(\Omega_t, u^t,\hat\psi)$,
where $u^t\in E(\D)$ and $\hat\psi\in F(\D) $.
The main appeal of the Lagrangian is that 
we actually only need to compute the partial derivative with respect to $t$ of $\mathcal{L}(\Omega_t,\hat\varphi,\hat\psi)$ to compute the derivative of $J(\Omega_t)$,  indeed this is the main result of Theorem \ref{thm:sturm}.

In order to differentiate $\mathcal{L}(\Omega_t, \hat\varphi,\hat\psi)$, the change of coordinates $x\mapsto T_t(x)$ is used in the integrals.
In the process appear the pullbacks $\hat\varphi\circ\Tt\in E(\D)$ and $\hat\psi\circ\Tt\in F(\D)$ which depend on $t$.
The usual procedure in shape optimization to compensate this effect is to use a reparameterization 
$\mathcal{L}(\Omega_t, \Psi_t(\varphi), \Psi_t (\psi))$ instead of $\mathcal{L}(\Omega_t,\hat\varphi,\hat\psi)$, where 
$\Psi_t$ is an appropriate bijection of 
$E(\D)$ and $F(\D)$, and $\varphi\in E(\D)$, $\psi\in F(\D)$.
Now the change of variable in the integrals yields functions $\varphi$ and $\psi$ in the integrands, which are independent of $t$.
In this paper we take $E(\D) = W^1_{\Gamma,p}(\D)$, $F(\D) = W^1_{\Gamma,p'}(\D)$, and $\Psi_t(\psi) = \psi\circ\Tt^{-1}$ is then a bijection of $E(\D)$ and $F(\D)$; see \cite[Theorem 2.2.2, p.52]{b_ZI_1989a}.

Thus we consider the so-called {\it shape-Lagrangian} $G:[0,\tz]\times E\times F \rightarrow \R$ with
$$ G(t,\varphi,\psi):
=\mathcal{L}(\Om_t,\varphi\circ\Tt^{-1},\psi\circ\Tt^{-1}).$$
The main result of this section, Theorem \ref{thm:sturm}, shows that in order to obtain the shape derivative of $\mathcal{L}$, it is enough to compute the partial derivative with respect to $t$ of $G$ while assigning the values $\varphi=u$ and $\psi=p$, where $u$ is the state and $p$ is the adjoint state. 
The main ingredient is the introduction of the averaged adjoint equation described below.

Let us assume that for each $t\in [0,\tz]$ the equation
\ben\label{eq:state_G}
d_\psi G(t,u^t,0;\hat\psi) = 0\;\text{ for all } \hat\psi \in  F.
\een
admits a unique solution $u^t\in E$.
Further, we make the following assumptions for  $G$.
%%%%%%%%%%%%%%%%%%
\begin{assumption}
\label{amp:gateaux_diffbar_G}
\label{amp:affine-linear}
For every $(t,\psi)\in [0,\tz]\times F$
\begin{enumerate}
\item[(i)]  $[0,1]\ni s\mapsto G(t,su^t + (1-s)u^0),\psi)$ is absolutely continuous. 
\item[(ii)] $[0,1]\ni s\mapsto d_\varphi G(t,su^t+(1-s)u^0,\psi;\hat{\varphi})$ belongs to $L^1(0,1)$ for all $\hat{\varphi}\in E$.
\end{enumerate}
\end{assumption}
When Assumption \ref{amp:affine-linear} is satisfied,  for $t\in [0,\tz]$ we introduce the \textit{averaged adjoint equation} associated with $u^t$ and $u^0$: 
find $p^t\in F$ such that 
\begin{equation}\label{averated_}
\int_0^1 d_\varphi G(t,su^t+(1-s)u^0,p^t;\hat{\varphi})\, ds =0 \quad \text{ for all } \hat{\varphi}\in E.
\end{equation}
In view of Assumption \ref{amp:affine-linear} we have 
\begin{equation}
\label{eq:main_averaged}
G(t,u^t,p^t)-G(t,u^0,p^t) = \int_0^1 d_\varphi G(t,su^t+(1-s)u^0,p^t;u^t-u^0)\, ds =0\quad \text{ for all } t\in[0,\tz].
\end{equation}
We can now state the main result of this section.
\begin{assumption}\label{H1}
We assume that
$$ \lim_{t\searrow 0} \frac{G(t,u^0,p^t)-G(0,u^0,p^t)}{t}=\partial_tG(0,u^0,p^0).$$
\end{assumption}
%%%%%%%%%%%%%%%%%%%
\begin{theorem}
\label{thm:sturm}
Let  Assumption \ref{amp:affine-linear} and Assumption \ref{H1} be satisfied and assume there exists a unique solution $p^t$  of the averaged adjoint equation \eqref{averated_}.
%%%%%%%%%%%%%
Then for all $\psi \in F$  we obtain
\begin{equation}\label{eq:dt_G_single}
\dt b(t,u^t) |_{t=0} = \dt(G(t,u^t,\psi))|_{t=0}=\partial_t G(0,u^0,p^0).
\end{equation}
\end{theorem}

%%%%%%%%%%%%%%%%%%%%%%%%%%%%%%%%%%%%%%%%%%%%%%%%%%
\section{Appendix 2: proof of Theorem \ref{thm01}}
%%%%%%%%%%%%%%%%%%%%%%%%%%%%%%%%%%%%%%%%%%%%%%%%%%

For the convenience of the reader we write here the proof of Theorem~\ref{thm01}, which is essentially the same as the proof of \cite[Theorem 1]{a_GRRE_1989a}. 
We recall from \cite{MR990595} that if $\D\cup \Gamma$ is regular in the sense of Gr\"oger, then the mapping $\Jop:W^1_{\Gamma,q}(\D) \to (W^1_{\Gamma,q'}(\D))^*$ defined by  
$$\langle \Jop v,w\rangle := \int_\D \nabla v\cdot \nabla w + vw $$ 
is onto and hence 
the inverse $\Jop^{-1}:(W^1_{\Gamma,q'}(\D))^*\to W^1_{\Gamma,q}(\D)$ is well-defined. 
%%%%%%%%%%%%%%%%%%%
\begin{proof}[Proof of Theorem~\ref{thm01}]
Let $f\in (W_{\Gamma,q'}^1(\D))^*$ be given. 
As in  \cite[Theorem 1]{a_GRRE_1989a}  we define for $s>0$ the mapping 
\begin{align*}
Q_f: W_{\Gamma,q}^1(\D) & \to W_{\Gamma,q}^1(\D),\\
    u & \mapsto \Jop^{-1}(D^*BDu + tf),
\end{align*}
where $D:W_{\Gamma,2}^1(\D)\to L^2(\D,\bbR^3)$ is defined by $u\mapsto (u,\nabla u)$, $D^*: L^2(\D,\bbR^3)^*\to (W_{\Gamma,2}^1(\D))^*$ is the adjoint of $D$ and $By:= y - s\widehat\ma y$ for $y=(y_0,y_1,y_2)\in L^2(\D,\bbR^3)$ and $\widehat\ma y: = (0,\ma (y_1,y_2)^\transp)$. 
We observe that $D^*BDu = D^*Du -s D^*(0,\ma \nabla u)$ and $D^*D=\Jop$, which yields 
\begin{equation*}
Q_fu = u - s\Jop^{-1}( D^*(0,\ma \nabla u) - f).
\end{equation*}
If $Q_f$ has a fixed point in $W_{\Gamma,q}^1(\D)$, then we obtain $D^*(0,\ma \nabla u) = f$ in $(W_{\Gamma,q'}^1(\D))^*$ which is equivalent to $\mathcal{A}_q  u=f$ in $(W_{\Gamma,q'}^1(\D))^*$. 
The proper choice of $s$ allows to show that $Q_f$ is a contraction and the result follows 
from Banach's fixed point theorem.
Note that $\|D\|_{L^2} = \|D^*\|_{L^2}=1$.
Then for all $v,w\in W_{\Gamma,q}^1(\D)$ we have
\begin{equation} \label{E:GR_1}
\|Q_fv-Q_fw\|_{W^1_{q}} \le \|\Jop^{-1}\|_{L(W^{-1}_{\Gamma,q},W^1_{\Gamma,q})}  \|D^*\|_{L^2}  \|BD(v-w)\|_{L^q}.
\end{equation}
Now, using assumptions~\eqref{assump2} yields, for all $s>0$,
\ben\label{E:stackrel_By}
\begin{split}
|By(x)|^2 & = |y(x)|^2 - 2s \widehat\ma(x)y(x)\cdot y(x) + s^2 |\widehat\ma(x)y(x)|^2 \\
&\le |y(x)|^2 - 2ms|y(x)|^2 + s^2 M^2|y(x)|^2.
\end{split}
\een
Hence, choosing $s= m/M^2$ yields $|By(x)|\le k|y(x)|$ with $k:= (1-m^2/M^2)^{1/2}$ and thus
\ben\label{E:GR_2}
\|BD(u-v)\|_{L^q} \stackrel{\eqref{E:stackrel_By}}{\le} k\|D(u-v)\|_{L^q}.
\een
Combining \eqref{E:GR_1}, \eqref{E:GR_2}, $M_q = \|\Jop^{-1}\|_{L(W^{-1}_{\Gamma,q},W^1_{\Gamma,q})} $ and  $\|D\|_{L^2} =1$ yields
\begin{equation*}
\|Q_fu-Q_fv\|_{W^1_{q}} \le k M_q\|u-v\|_{W^1_q}.
\end{equation*}
Since we have assumed that $k M_q<1$, it follows that $Q_f$ is a contraction.

Let $u$ and $v$ be the fixed points of $Q_f, Q_g$, respectively.
Then we have
\begin{align*}
\|u-v\|_{W^1_{q}} &= \|Q_f u-Q_g v\|_{W^1_{q}} \le k M_q \|u-v\|_{W^1_q} + \|Q_f v-Q_g v\|_{W^1_{q}}\\
&\leq  k M_q\|u-v\|_{W^1_q} + t  \|\Jop^{-1}\|_{L(W^{-1}_{\Gamma,q},W^1_{\Gamma,q})}\|f - g\|_{W^{-1}_{\Gamma,q}}.
\end{align*}
Using $M_q = \|\Jop^{-1}\|_{L(W^{-1}_{\Gamma,q},W^1_{\Gamma,q})} $  we obtain
\begin{align*}
(1-k M_q)\|u-v\|_{W^1_{q}} & \leq  t  M_q \|f - g\|_{W^{-1}_{\Gamma,q}}.
\end{align*}
which proves \eqref{Aq_iso}.
By Lemma \ref{lemma01} the hypothesis $M_q k<1$ of Theorem \ref{thm01} is satisfied if $q>2$ is sufficiently close to $2$.
\end{proof}

\bibliographystyle{abbrv}
\bibliography{EIT_pointwise_arxiv}

\end{document}